\documentclass[11pt]{article}
\usepackage{CJKutf8}
\usepackage{amssymb,amsmath,verbatim,amsthm}
\usepackage{amsmath,cite,tikz,multirow,amssymb,mathrsfs,latexsym,bm,inputenc}
\usepackage{longtable}

\newtheorem{theorem}{Theorem}[section]
\newtheorem{lemma}[theorem]{Lemma}

\newtheorem{corollary}[theorem]{Corollary}

\newtheorem{remark}[theorem]{Remark}
\newtheorem{construction}[theorem]{Construction}
\newtheorem{conjecture}[theorem]{Conjecture}

\newcommand{\ignore}[1]{}

\begin{document}
\begin{CJK}{UTF8}{<font>}

\title{Optimal data placements for triple replication\footnote{Supported by NSFC Grant 11971053.}}

\author{\small Ruijing Liu, Junling Zhou\\
\small Department of Mathematics \\
\small Beijing Jiaotong University\\
\small Beijing 100044, P. R. China\\
\small 18118006@bjtu.edu.cn\\
\small jlzhou@bjtu.edu.cn}

\date{ }
\maketitle

\begin{abstract}
Given a set $V$ of $v$ servers along with $b$ files (data), each file is replicated (placed) on exactly $k$ servers and thus a file can be represented by a set of $k$ servers. Then we produce a data placement consisting of $b$ subsets of $V$ called blocks, each of size $k$. Each server has some probability to fail and we want to find a placement that minimizes the variance of the number of
 available files. It was conjectured that there always exists an optimal data placement (with variance better than any other placement for any value of the probability of failure).
  An optimal data placement for triple replication  with $b$ blocks (of size three) on a $v$-set was proved to exist by Wei et al. if $v$ and $b$ are not excluded by two conditions.
   This article concentrates on the parameters $v,  b$ satisfying the two conditions and characterizes  the combinatorial properties of the corresponding optimal data placements.
    Nearly well-balanced triple systems (NWBTSs) are defined to produce optimal data placements. Many constructions for NWBTSs are developed, mainly by constructing
    candelabra systems with various desirable partitions. The main result of this article is that there always exist optimal data placements for triple replication with $b$ blocks on a $v$-set
     possibly except when $v\equiv 4$ (mod 24) or $v=50,74$, and $\frac{\lambda v(v-1)}{6}-\frac{v}{6} < b < \frac{\lambda v(v-1)}{6}+\frac{v}{6}$ for an odd integer $\lambda$.

%\medskip\noindent{\bf MSC}: 05B30, 94B65, 94C30

\medskip\noindent {\bf Keywords}: data placement, optimal, triple system, $s$-fan design, candelabra system, nearly well-balanced, partitionable
\end{abstract}

%%%%%%%%%%%%%%%%%%%%%%%%%%%%%%%%%%%%%%%%%%%%%%%%%%%%%%%%%%%%%%%%%%%%%%%%%

%\section{Introduction}

\section{Introduction}

Large scale distributed storage systems play a vital role in maintaining data across storage locations globally. %These systems use replication as the default mechanism for providing fault-tolerance.
These systems are built using cheap commodity
hardware, that can be unreliable. Hence, it is essential for the
system to store data in replication to ensure reliability and fault-tolerance.
%The simplest form of redundancy is replication, which is adopted in many large scale commercial storage systems. In replication, the same copy of the data is %replicated and stored at multiple locations for protecting the data from node failures, thereby ensuring fault-tolerance and reliability.
There is growing interest of
using distributed systems for massive video distribution. Documents replication and repartition can also
solve some performance problems for video distribution,
and is an effective way to improve the performance of video
download time. Those problems have been studied extensively and their connections to combinatorial objects have been displayed \cite{AJ,2,3,13}.

In this article we consider data placements for triple replication, which is the
default replication storage in Hadoop's Distributed File
System. Here, we
use the terminology of combinatorial design  theory.
 We are given a set $V$ of $v$ servers along with $b$ files (data, documents). Each file is replicated (placed) on exactly $k$ servers (nodes). The set of servers containing file $i$ is therefore a subset of size $k$, which will be called a \emph{block} and denoted by $B_{i}$. Then the set system $(V,\mathcal{F})$ or the family $\mathcal{F}$ of the $b$ blocks $B_{i}$, $1 \leq i \leq b$, is called a \emph{data placement}.  A {\em data placement for triple replication} is a triple system $(V,{\cal F})$, also denoted by  TS$(v;b)$, where every block of ${\cal F}$ has cardinality three.

It is interesting to study the situation where some of the servers might not be available when asked to provide a file. We assume that a server is available (online) with some probability $\delta$, and so unavailable (off-line, failed) with the probability $1-\delta$. The file $i$ is said to be available if one of the servers containing it is available, or equivalently the file is unavailable if all the servers containing it are unavailable.
In \cite{2,3,13}, the authors studied the random variable $\Lambda$, the number of available files, and they proved that the expectation $E(\Lambda)$ is independent of the data placement. However, they proved that the variance of $\Lambda$ depends on the data placement $(V,\mathcal{F})$. Jean-Marie et al. \cite{13} showed that minimizing the variance of $\Lambda$ corresponds to
minimizing the polynomial $$P(\mathcal{F}, x) =\sum^{k}_{j=0} v_{j}x^{j},$$ where $x = \frac{1}{1-\delta}$ (so $x \geq 1$) and $v_{j}$ denotes the number of ordered pairs of blocks intersecting in exactly $j$ elements of $V$.

A basic problem is to find an \emph{optimal} data placement $(V,\mathcal{F})$, which minimizes $P(\mathcal{F}, x)$ for all the values of $x\geq1$. In \cite{13}, the following conjecture is proposed.

\begin{conjecture}\rm{\cite{13}}\label{CC}
For any nonnegative integers $v, k$ and $b$ there exists an optimal data placement consisting of $b$ blocks of size $k$ on a $v$-set.
\end{conjecture}

Bermond et al. \cite{wbd} gave an equivalent formulation of the polynomial $P(\mathcal{F}, x)$.
For a set system $(V,\mathcal{F})$, let $\lambda^{\mathcal{F}} _{x_{1},\ldots,x_{j}}$ (or $\lambda _{x_{1},\ldots,x_{j}}$) denote the number of blocks of $\mathcal{F}$ containing the $j$-subset $\{x_{1},\ldots,x_{j}\}$ of $V$.

\begin{lemma}\rm{\cite{wbd}}\label{jian}
For a data placement $(V,\mathcal{F})$ consisting of $b$ blocks of size $k$, $$P(\mathcal{F}, x) = \sum^{k}_{j=1}\sum _{\{x_{1},\ldots,x_{j}\}\in {V\choose j}}\lambda_{x_{1},\ldots,x_{j}}^{2}(x-1)^{j} - bx^{k} + b^{2},~x \geq 1,$$ where  ${V\choose j}$ denotes the collection of all $j$-subsets of $V$.
\end{lemma}

A set system $(V,\mathcal{F})$ is \emph{$j$-balanced} if the $\lambda _{x_{1},\ldots,x_{j}}$'s are all equal or almost equal, that is, if $|\lambda _{x_{1},\ldots,x_{j}}-\lambda _{y_{1},\ldots,y_{j}}| \leq 1$ for any two $j$-subsets $\{x_{1},\ldots,x_{j}\}$ and $\{y_{1},\ldots,y_{j}\}$ of $V$. Furthermore, $(V,\mathcal{F})$ is \emph{well-balanced} if it is $j$-balanced for all $1 \leq j \leq k$, where $k$ is the size of each block.
%Recall that a \emph{Steiner $t$-system} (or S$_{\lambda}(t, k, v)$) is a family of blocks, each of size $k$, such that each $t$-subset of a given $v$-set appears in exactly $\lambda$ blocks.A family $\mathcal{F}$ consisting of $b$ blocks, each of size $3$ on a $v$-set, is called a \textit{triple system}, denoted by TS$(v;b)$.

\begin{lemma}\rm{\cite{wbd}}\label{balanced}
Let $\mathcal{F}$ consist of $b$ blocks of size $k$ on a $v$-set $V$. For $1\leq j \leq k$, if $(V,\mathcal{F})$ is $j$ -balanced, then $\sum_{\{x_{1},\ldots,x_{j}\}\in {V\choose j}}\lambda^{2}_{x_{1},\ldots,x_{j}}$ is minimized.
\end{lemma}

Bermond et al. \cite{wbd} used  Lemmas  \ref{jian} and \ref{balanced} to prove that a well-balanced set system gives rise to an optimal data placement. They proved the truth of Conjecture \ref{CC} for $k=2$ and provided some positive instances for  $k=3$. Afterwards, Wei et al. \cite{wbt} determined completely the spectrum of pairs $(v,b)$ for which there exists a well-balanced triple system, simply denoted  WBTS$(v;b)$.

\begin{theorem}\rm{\cite[Theorem 8.1]{wbt}}\label{NF}
%Let $v \geq 3$ and $b \geq 0$ be integers. Then
A WBTS$(v; b)$  exists for any positive integers $v, b$ and $v \geq 3$, with definite exceptions that

(i) $b\in \{\lfloor \frac{\lambda v(v-1)}{6} \rfloor, \lceil \frac{\lambda v(v-1)}{6} \rceil\}$ for an integer $\lambda$ such that $\lambda v(v-1) \not\equiv0 \pmod {6}$; and

(ii) $v$ is even and $\frac{\lambda v(v-1)}{6}-\frac{v}{6} < b < \frac{\lambda v(v-1)}{6}+\frac{v}{6}$ for an odd integer $\lambda$.
\end{theorem}

\begin{corollary}\label{opt}
 An optimal data placement TS$(v; b)$  exists for any positive integers $v, b$ and $v \geq 3$,  possibly except that they satisfy (i) or (ii) of Theorem \ref{NF}.
\end{corollary}

\begin{remark}\label{R} \rm\cite[Propositions 6, 7]{wbd}
When $(v,  b)$ satisfies $(i)$ or $(ii)$ of Theorem \ref{NF}, a $2$-balanced TS$(v;b)$ does not exist.
\end{remark}

 By Theorem \ref{NF} and Corollary \ref{opt}, the existence problem of well-balanced triple systems has been resolved, while the existence of optimal data placements for triple replication remains open. In this article we  concentrate on optimal data placements TS$(v;b)$s with $v,b$ meeting (i) or (ii) of Theorem \ref{NF}. For convenience, we redistribute these parameters by two exclusive conditions as follows:

(C1) $v \equiv 2$ (mod 3) and $b\in \{\lfloor \frac{\lambda v(v-1)}{6}\rfloor, \lceil \frac{\lambda v(v-1)}{6}\rceil\}$, where $\lambda \equiv 1,2$ (mod 3)  if $v \equiv 5$ (mod 6) and $\lambda \equiv 2,4$ (mod 6) if  $v \equiv 2$ (mod 6).

(C2)  $v$ is even  and $\frac{\lambda v(v-1)}{6}-\frac{v}{6} < b < \frac{\lambda v(v-1)}{6}+\frac{v}{6}$ for an odd $\lambda$.

 The rest of this article is organized as follows.
In Section 2, we characterize the combinatorial properties of an optimal data placement TS$(v; b)$ by defining a defect graph for a triple system and introducing the concept of a nearly well-balanced triple system (NWBTS). We show that an NWBTS$(v;b)$ gives rise to an optimal data placement. In Section 3 we introduce a vital design named candelabra system and define its desirable partitions to produce NWBTSs when $(v,  b)$ satisfies (C1) or (C2). Then in Section 4 we study the case of (C1) and completely prove the existence of  NWBTSs in this case. In Sections 5 and 6 we turn to  (C2) and generalize the concept of partitionable candelabra systems to treat this case. Effective recursive constructions are developed; necessary small  input designs are constructed directly; a large portion of NWBTSs with parameters satisfying (C2) is proved to exist. In Section 7 we summarize our main results and present a promising approach to resolving entirely the existence problem of NWBTSs and optimal data placements.

\section{Combinatorial descriptions}

In this section we study the combinatorial properties of a TS$(v;b)$ corresponding to
an optimal data placement for triple replication when $(v,  b)$ satisfies (C1) or (C2). We will define nearly well-balanced triple systems to produce optimal data placements.
 %TS$(v;b)$ exists for all positive integers $v\ge 3$ and $b$, with possible exceptions in (C1) or (C2).

For a TS$(v;b)$, where $\frac{\lambda v(v-1)}{6}-\frac{v}{6} < b < \frac{\lambda v(v-1)}{6}+\frac{v}{6}$ for some nonnegative integer $\lambda$, we may write
\begin{align}\label{associate}
3b = \lambda{v\choose 2}+\varepsilon,\ {\rm where\ } -\frac{v}{2}<\varepsilon<\frac{v}{2}\ {\rm and\ }  \varepsilon\ {\rm is~an\ integer}.
\end{align}
% It is easy to see that the values of $\lambda$ and $\varepsilon$ are uniquely determined by $v$ and $b$.
 %, and they will also be written as $\lambda(v,b)$ and $\varepsilon(v,b)$  in the context.
In this article we always use the notion of $(\lambda,\varepsilon)$, which is associated with the given parameters $v$ and $b$ by (\ref{associate}). In particular, for (C1) we have $\varepsilon\in\{-1,2\}$ if $\lambda \equiv 1$ (mod 3) and $\varepsilon\in\{-2,1\}$ if $\lambda \equiv 2$ (mod 3).
 We will study the combinatorial properties of a TS$(v;b)$ which produces an optimal data placement for (C1) or (C2). By Remark \ref{R}, the optimal data placement  TS$(v;b)$ cannot be 2-balanced. However, we will show that a little weaker balance should hold for pairs of elements.

Let $(V,\mathcal{F})$ be a TS$(v;b)$ where $(v,b)$ associates with $(\lambda,\varepsilon)$ as in (\ref{associate}).
For any nonzero integer $i$, define the \emph{defect graph of $(V,\mathcal{F})$ at level $i$} by $$D_{i}=\{\{x,y\}\in \tbinom{V}{2}: \lambda_{x,y}=\lambda+i\}$$ and define the \emph{defect} of $(V,\mathcal{F})$ by the largest $\beta$ such that $D_{\beta}\neq\emptyset$ or $D_{-\beta}\neq\emptyset$. Then the \emph{defect graph} of $(V,\mathcal{F})$ is defined to be $D=\cup_{i=1}^{\beta}(D_{i}\cup D_{-i})$.
%A {\em Steiner triple system  with index $\lambda$}, denoted by STS$_{\lambda}(v)$, is a triple system defined over a $v$-set such that every pair of elements is contained in exactly $\lambda$ blocks. Thus an STS$_{\lambda}(v)$ has an empty defect graph.
A 2-balanced triple system has a defect graph $D=D_{1}\cup D_{-1}$ with $D_{1}=\emptyset$ or $D_{-1}=\emptyset$.

\begin{lemma}\label{c1balanced}
Let $(v,  b)$ satisfy (C1) and they associate with  $(\lambda,\varepsilon)$ by (\ref{associate}). %$3b = \lambda{v\choose 2}+\varepsilon$  with $\lambda\ge 0$  and  $\varepsilon\in\{\pm1,\pm2\}$.
 Then the minimum of $\big\{\sum_{\{x,y\}\in \binom{V}{2}}(\lambda_{x,y}^{\cal F})^{2}: (V,\mathcal{F})$ is a TS$(v;b)\big\}$ is achieved when $(V,\mathcal{F})$ has a defect  graph $D$ meeting one of the followings:

 (i)  $D=D_{1}\cup D_{-1}$ with $(|D_{1}|, |D_{-1}|)=\begin{cases}
(1+\varepsilon,1), & if\ \varepsilon\in\{1,2\},\\
&\\
(1,1-\varepsilon), & if\ \varepsilon\in\{-1,-2\};
\end{cases}$

(ii) if $\varepsilon\in\{\pm2\}$,  we may also have   $D=D_{\varepsilon}$ and $|D|=1$.

\end{lemma}

\proof Suppose that $(V,\mathcal{F})$ is a TS$(v;b)$ with defect $\beta$ and its defect graph $D=\cup_{i=1}^{\beta}(D_{i}\cup D_{-i})$. Denote $|D_{i}|=a_{i}$, $-\beta\leq i\leq\beta~(i\neq0)$. Further denote $a_{0}=|\{\{x,y\}\in {V\choose 2}:\lambda_{x,y}^{\cal F}=\lambda\}|$.
It is immediate that
$$\sum_{i=-\beta}^{\beta}a_{i}(\lambda+i)=3b~\text{and}~\sum_{i=-\beta}^{\beta}a_{i}={v\choose 2}.$$

\noindent Since $3b = \lambda {v\choose 2}+\varepsilon$ where $\varepsilon\in \{\pm1, \pm2\}$, we have
\begin{align}\label{cap1}
\sum_{i=-\beta}^{\beta} ia_{i}=\varepsilon,
\end{align}
\begin{align}\label{cap2}
\sum_{\{x,y\}\in {V\choose 2}}(\lambda_{x,y}^{\cal F})^{2}=\sum_{i=-\beta}^{\beta} a_{i}(\lambda+i)^{2}=\binom{v}{2}\lambda^{2}+2\varepsilon\lambda+\sum_{i=-\beta}^{\beta} i^{2}a_{i}.
\end{align}
If the defect $\beta\geq 2$, then (\ref{cap2}) gives
\begin{align}\label{cap3}
\sum_{\{x,y\}\in {V\choose 2}}(\lambda_{x,y}^{\cal F})^{2}\geq\binom{v}{2}\lambda^{2}+2\varepsilon\lambda+4~~(\text{if}~\beta\geq2).
\end{align}
Next we let $\beta=1$. By (\ref{cap2}) we have $\sum_{\{x,y\}\in {V\choose 2}}(\lambda_{x,y}^{\cal F})^{2}=\binom{v}{2}\lambda^{2}+2\varepsilon\lambda+a_{1}+a_{-1}$.
It follows from (\ref{cap1}) that $a_{1}-a_{-1}=\varepsilon$. Notice that $a_{1}, a_{-1}\geq1$; otherwise $(V,\mathcal{F})$ is a 2-balanced triple system, which contradicts Remark \ref{R}. Hence we have
\begin{align}\label{cap4}
\sum_{\{x,y\}\in {V\choose 2}}(\lambda_{x,y}^{\mathcal{F}})^{2}\geq\binom{v}{2}\lambda^{2}+2\varepsilon\lambda+2+|\varepsilon| ~~({\rm if}~\beta=1).
\end{align}

Comparing (\ref{cap3}) and (\ref{cap4}) yields that, if  $\varepsilon\in\{\pm1\}$, then $\sum_{\{x,y\}\in {V\choose 2}}(\lambda_{x,y}^{\cal F})^{2}$ attains the minimum in  (\ref{cap4}) with $(a_{1}, a_{-1})$ as assumed in the lemma; and if $\varepsilon\in\{\pm2\}$, then it  attains the minimum both in (\ref{cap3}) and (\ref{cap4})  when the defect graph meets  (i) or (ii) of the lemma.
\qed

Now we continue considering the case that $(V,\mathcal{F})$ is a TS$(v;b)$ with $(v,  b)$ satisfying (C1).  Assume that  $(V,\mathcal{F})$ has a defect graph $D$ satisfying (i) of Lemma \ref{c1balanced}. It is easy to depict the graph $D$ in detail.
Denote by $d_{D}(x)$ the degree of a vertex $x$ in the graph $D$.
For any $x\in V$, we have
\begin{align}
\lambda_{x}^{\mathcal{F}}&=\frac{1}{2}((\lambda+1)d_{D_{1}}(x)+(\lambda-1)d_{D_{-1}}(x)+\lambda(v-1-d_{D_{1}}(x)-d_{D_{-1}}(x))\notag\\
&=\frac{1}{2}(d_{D_{1}}(x)-d_{D_{-1}}(x)+\lambda(v-1)).\notag
\end{align}
Then we get that $d_{D_{1}}(x)\equiv d_{D_{-1}}(x) \pmod{2}$ as $\lambda(v-1)$ is even. It follows from Lemma \ref{c1balanced} that the  defect graph $D=D_{1}\cup D_{-1}$ must be isomorphic to $G_{\varepsilon}$ in Figure \ref{f1}, where an edge of $D_{1}$ or $D_{-1}$ is displayed in a solid line or dotted line, respectively.

\begin{figure}[htbp]
  \centering
  \includegraphics[width=0.31\textwidth]{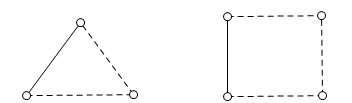}~~~~~
  \includegraphics[width=0.31\textwidth]{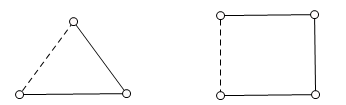}\\
  \small\small$G_{-1}$~~~~~~~~~~~~~~~~~~$G_{-2}$~~~~~~~~~~~~~~~~~~~~$G_{1}$~~~~~~~~~~~~~~~~~~~~$G_{2}$\medskip\\

  \caption{Defect graphs with defect 1 for (C1)}\label{f1}
\end{figure}

\begin{lemma}\label{c2balanced}
Let $(v,  b)$ satisfy (C2) and they associate with  $(\lambda,\varepsilon)$ by (\ref{associate}). Then the minimum of $\big\{\sum_{\{x,y\}\in \binom{V}{2}}(\lambda_{x,y}^{\cal F})^{2}: (V,\mathcal{F})$ is a TS$(v;b)\big\}$ is attained when $(V,\mathcal{F})$ has a defect  graph $D$ isomorphic to $H_{v,\varepsilon}^0$  if $\varepsilon\equiv\frac{v}{2}\pmod{2}$ or isomorphic to one of $H_{v,\varepsilon}^i$ with $1\le i\le 4$ if $\varepsilon\not\equiv\frac{v}{2}\pmod{2}$ (see Figure \ref{f2}, also a solid line stands for an edge of $D_1$ and a dotted line for an edge of $D_{-1}$).
\end{lemma}

\begin{figure}[htbp]
  \centering
  \includegraphics[width=0.18\textwidth]{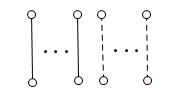}~~~~~~~~
  \includegraphics[width=0.23\textwidth]{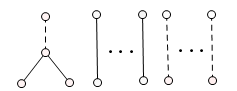}~~~~~~~~
  \includegraphics[width=0.23\textwidth]{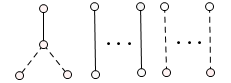}\\
  \small~~$\underbrace{}$~~~~$\underbrace{}$~~~~~~~~~~~~~~~~~~~~~~~~~~~~$\underbrace{}$~~~~$\underbrace{}$~~~~~~~~~~~~~~~~~~~~~~~~~~~$\underbrace{}$~~~~$\underbrace{}$\medskip\\
  \small$~~~~~~\frac{v+2\varepsilon}{4}$~~~$\frac{v-2\varepsilon}{4}$~~~~~~~~~~~~~~~~~~~~~~~~~~$\frac{v-6+2\varepsilon}{4}$~$\frac{v-2-2\varepsilon}{4}$
  ~~~~~~~~~~~~~~~~~~~~~~$\frac{v-2+2\varepsilon}{4}$~$\frac{v-6-2\varepsilon}{4}$~~\medskip\\
  \small~~~~~$H_{v,\varepsilon}^0$~~~~~~~~~~~~~~~~~~~~~~~~~~~~~~~~~~$H_{v,\varepsilon}^1$~~~~~~~~~~~~~~~~~~~~~~~~~~~~~~~~~~~~~$H_{v,\varepsilon}^2$~~~~~~\medskip\\

  \centering
  \includegraphics[width=0.234\textwidth]{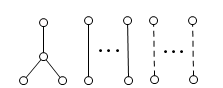}~~~~~
  \includegraphics[width=0.234\textwidth]{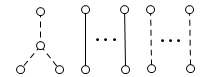}\\
  \small~~~~~~~~~$\underbrace{}$~~~~$\underbrace{}$~~~~~~~~~~~~~~~~~~~~~~~$\underbrace{}$~~~~$\underbrace{}$\medskip\\
  \small~~~~~~~~~~~~~~~$\frac{v-10+2\varepsilon}{4}$~~$\frac{v+2-2\varepsilon}{4}$~~~~~~~~~~~~~~~~~~$\frac{v+2+2\varepsilon}{4}$~$\frac{v-10-2\varepsilon}{4}$~~~~~\medskip\\
  \small~~~~~~~~~$H_{v,\varepsilon}^{3}$~~~~~~~~~~~~~~~~~~~~~~~~~~~~~~~~~$H_{v,\varepsilon}^{4}$~~~~~~~~~\medskip\\

  \caption{Defect graphs for (C2)}\label{f2}

\end{figure}

\proof Suppose that $(V,\mathcal{F})$ is a TS$(v;b)$ with defect $\beta$ and its defect graph  $D=\cup_{i=1}^{\beta}(D_{i}\cup D_{-i})$. Let $D_{0}=\{\{x,y\}\in \binom{V}{2}: \lambda_{x,y}^{\cal F}=\lambda\}$. Denote $|D_{i}|=a_{i}$, $-\beta\leq i\leq\beta$.
For any $x\in V$, we have
\begin{align}
\lambda_{x}^{\cal F}&=\frac{1}{2}(\sum _{i=-\beta}^{\beta}(\lambda+i)d_{D_{i}}(x))=\frac{1}{2}(\lambda\sum _{i=-\beta}^{\beta}d_{D_{i}}(x)+\sum _{i=-\beta}^{\beta}id_{D_{i}}(x))\notag\\
&=\frac{1}{2}(\lambda(v-1)+\sum _{i=-\beta}^{\beta}id_{D_{i}}(x)).\notag
\end{align}
Since $v$ is even and $\lambda$ is odd,  we have  $\sum _{i=-\beta}^{\beta}id_{D_{i}}(x)\equiv1 \pmod{2}$ to ensure that $\lambda_{x}^{\cal F}$ is an integer. Hence we have $d_D(x)\ge 1$. It follows from the arbitrariness of $x$ that  $|D|\ge {v\over 2}$.

Similarly to the proof of Lemma \ref{c1balanced}, we have (\ref{cap1}) and (\ref{cap2}). It is obvious that
\begin{align}\label{cap6}
\sum_{i=-\beta}^{\beta} i^{2}a_{i}&= \sum_{i=1}^{\beta}(a_{i}+ a_{-i})+\sum_{i=2}^{\beta}(i^{2}-1)(a_{i}+ a_{-i})\notag\\
&\geq \sum_{i=1}^{\beta}(a_{i}+ a_{-i})=|D|,\notag
\end{align}
where the  equality  occurs if and only if $\beta=1$.

So we let $\beta=1$ next. Note that $|D|\geq\frac{v}{2}$ and $a_{1}-a_{-1}=\varepsilon$ from (\ref{cap1}).
We have that $|D|$ is minimized to $|D|=\frac{v}{2}$ if $\frac{v}{2}\equiv\varepsilon\pmod{2}$, $(a_{1}, a_{-1})=(\frac{v+2\varepsilon}{4}, \frac{v-2\varepsilon}{4})$ and $D$ is a perfect matching isomorphic to the graph $H_{v,\varepsilon}^0$ in Figure \ref{f2}; and that $|D|$ is minimized to $|D|=\frac{v}{2}+1$ if $\frac{v}{2}\not\equiv\varepsilon\pmod{2}$, $(a_{1}, a_{-1})=(\frac{v+2+2\varepsilon}{4}, \frac{v+2-2\varepsilon}{4})$ so that $D$ is isomorphic to one of the graphs $H_{v,\varepsilon}^i$, $1\le i\le 4$, in Figure \ref{f2}. From (\ref{cap2}), when $|D|$ attains the minimum, $\sum_{\{x,y\}\in {V\choose 2}}(\lambda_{x,y}^{\cal F})^{2}$ also does. This proves the lemma.
\qed

Recall that an optimal data placement TS$(v;b)$ minimizes the polynomial $$P(\mathcal{F}, x) = \sum^{3}_{j=1}\sum _{\{x_{1},\ldots,x_{j}\}\in{V\choose j}}(\lambda^{\cal F}_{x_{1},\ldots,x_{j}})^2(x-1)^{j} - bx^{3} + b^{2}$$ for all $x \geq 1.$  From Lemma \ref{balanced}, the values $\sum_{x\in V}(\lambda_{x}^{\cal F})^{2}$ and $\sum_{\{x,y,z\}\in {V\choose 3}}(\lambda_{x,y,z}^{\cal F})^{2}$ are minimized when $(V,{\cal F})$ is 1-balanced and 3-balanced, respectively.  Lemmas \ref{c1balanced} and \ref{c2balanced} describe the  combinatorial property of an optimal TS$(v;b)$ in term of its defect graph when the property of 2-balance cannot hold. For (C1)  its defect graph $D$ with defect one must be isomorphic to  $G_{\varepsilon}$  and $D$ must be isomorphic to a graph in Figure \ref{f2} for (C2).  It is not difficult to see that if $D$ is isomorphic to $H_{v,\varepsilon}^3$ or $H_{v,\varepsilon}^4$,  the family $\mathcal{F}$ is not 1-balanced. For convenience, we then define  $\mathcal{F}$ to be \emph{nearly $2$-balanced} if for (C1) its defect graph $D$ is isomorphic to  $G_{\varepsilon}$, or for (C2)  $D$ is isomorphic to $H_{v,\varepsilon}^0$ if $\varepsilon\equiv\frac{v}{2}\pmod{2}$ or to one of $H_{v,\varepsilon}^1$ and $H_{v,\varepsilon}^2$ if $\varepsilon\not\equiv\frac{v}{2}\pmod{2}$. Furthermore, a TS$(v;b)$ is said to be \emph{nearly well-balanced}, denoted by NWBTS$(v; b)$, if it is  nearly $2$-balanced (hence 1-balanced) and  $3$-balanced. With these preparations, we have the main result of this section recorded in the following theorem.

%Let $v,b,\lambda,\varepsilon$ satisfy (C1) and (C1'), or (C2) and (C2'). Let $(V,\mathcal{F})$ be a TS$(v;b)$. From Lemmas \ref{balanced}, \ref{c1balanced} and  \ref{c2balanced}, to produce an optimal data placement,  $\mathcal{F}$ must be 1-balanced, 3-balanced, and  its defect graph $D$ must be isomorphic to  one of the graphs  in Figure 1 for (C1') or Figure 2 for (C2').  It is not difficult to see that if $D$ is isomorphic to $H_{\varepsilon}^3$ or $H_{\varepsilon}^4$ for (C2')

\begin{theorem}\label{kk} An NWBTS$(v; b)$ yields an optimal data placement for triple replication when $(v,  b)$ satisfies (C1) or (C2).
%If a triple system $\mathcal{F}^{*}$ consisting of $b$ blocks on a $v$-set is nearly well-balanced, then $\mathcal{F}^{*}$ is optimal, that is, $P(\mathcal{F}^{*}, x)\leq P(\mathcal{F}, x)$ for any $\mathcal{F}$ and any $x \geq 1$.
\end{theorem}
%\proof If a triple system $\mathcal{F}^{*}$ is nearly well-balanced, then all the coefficients of the polynomial as expressed in the Theorem \ref{jian} are minimized by Lemmas \ref{c1balanced}, \ref{c2balanced}, and \ref{balanced}, and so $\mathcal{F}^{*}$ is optimal.\qed
 In the following of the article when referring to a nearly well-balanced or a nearly 2-balanced TS$(v; b)$ we always imply that the parameter pair $(v,  b)$ satisfies (C1) or (C2)  and  the pair $(\lambda,\varepsilon)$ is associated with $(v,  b)$ by (\ref{associate}). By the definition of a nearly 2-balanced triple system, the defect graph depends only on $\varepsilon$ for (C1) and on $v,\varepsilon$ for (C2), which will be often used in later constructions.

A triple system $(V,\mathcal{F})$ is called \emph{simple} if it contains no repeated blocks. Obviously a simple TS$(v; b)$ is 3-balanced and we have the following lemma.

\begin{lemma}\label{NBTS}
A nearly $2$-balanced simple TS$(v; b)$ is an NWBTS$(v; b)$.
\end{lemma}
%
%Simple set systems play vital roles in considering optimal placements, as was shown in \cite{wbd}. In what follows, we will restrict ourselves to the case $b\leq \binom{v}{k}$. In fact the following lemma shows that we only need to consider the case of $b\leq\frac{1}{2}\binom{v}{k}$.
%The following simple lemma will be frequently used.

\begin{lemma}\label{bb} Let $v,h,b$ be positive integers and $b<\binom{v}{3}$.
\begin{enumerate}
\item[(i)] If there exists an NWBTS$(v; b)$, then so does an NWBTS$(v; b_1)$ where $b_1 = h\binom{v}{3}+ b$.

\item[(ii)]  If there exists an NWBTS$(v; b)$, then so does an NWBTS$(v; b_2)$ where  $b_2={v\choose 3}-b$.
\end{enumerate}
\end{lemma}

\proof %The second statement was proved from \cite[]{wbd} by taking the complement. So we only prove (1).
Suppose that $(V,\mathcal{F})$ is an NWBTS$(v; b)$. Then $\mathcal{F}$ is simple as $b <\binom{v}{3}$ and it is nearly 2-balanced. Let $\mathcal{F}$ have a defect graph $D=D_1\cup D_{-1}$. Assume  $3b = \lambda\binom{v}{2} +\varepsilon$  as in (\ref{associate}).

(i) Let $\mathcal{F}_1$ consist of all blocks of $\mathcal{F}$ and all triples of $V$ with each triple taken $h$ times. Obviously every triple of $V$ occurs $h$ or $h+1$ times in $\mathcal{F}_1$ and thus $\mathcal{F}_1$ is 3-balanced. We also learn that every pair of $D_1$ and $D_{-1}$ is respectively contained in $\lambda_1+1$ and $\lambda_1-1$ blocks of ${\cal F}_1$ and all other pairs in $\lambda_1$ blocks, where $\lambda_1=\lambda+h(v-2)$. It follows that $D$ is also the defect graph of $\mathcal{F}_1$, which fulfils  the definition of nearly 2-balance as  $3b_1=3b+3h\binom{v}{3}=\lambda_1 {v\choose2}+\varepsilon$. Now that $(V,\mathcal{F}_1)$ is nearly 2-balanced and 3-balanced, it is nearly well-balanced with $b_1$ blocks.

(2) Let $\mathcal{F}_2={V\choose 3}\setminus {\cal F}$. Clearly ${\cal F}_2$ is 3-balanced with $b_2$ blocks. Denote the defect graph of ${\cal F}_2$ by $D'$. It is easy to see that $D'$ has defect 1 and $D'=D_{-1}\cup D_1$, which means every pair of $D_{-1}$ and $D_1$ is respectively contained in $\lambda_2+1$ and $\lambda_2-1$ blocks of ${\cal F}_2$ and all other pairs in $\lambda_2$ blocks where $\lambda_2=v-2-\lambda$. Noticing  $3b_2=3\binom{v}{3}-3b=\lambda_2{v\choose 2}-\varepsilon$, we can check that ${\cal F}_2$ is also nearly 2-balanced in either case of $\varepsilon\equiv {v\over 2}$ (mod 2) or not. Thus we have an NWBTS$(v;b_2)$. \qed

%In what follows, we will restrict ourselves to the case $b\leq \binom{v}{k}$. In fact the following lemma shows that we only need to consider the case of $b\leq\frac{1}{2}\binom{v}{k}$.

%\begin{lemma}\rm{\cite{wbd}}\label{half}
%Let $v$ and $k$ be given. An optimal family $\bar{\mathcal{F}}$ for $\bar{b} = \binom{v}{k}-b$ can be
%obtained from an optimal family $\mathcal{F}$ for $b \leq \binom{v}{k}$ by taking as blocks the $k$-subsets on $v$-set that are not blocks of $\mathcal{F}$.
%\end{lemma}

\section{Candelabra systems with desirable partitions}

In this section we introduce the definition of candelabra systems and then define two kinds of partitions. The main effect is to produce NWBTSs with a large range of parameters. Analogous constructions are widely used for large sets, see for instance \cite{PCS,Lsts}.

%Let $v$ be a nonnegative integer, $t$ be a positive integer, and $K$ be a set of positive integers.
A \emph{candelabra $t$-system} (or $t$-CS as in \cite{CS}) of order $v$ and block sizes from $K$, denoted by CS$(t,K,v)$, is a
quadruple $(X, S, \mathcal{G}, \mathcal{A})$ that satisfies the following properties:

(1) $X$ is a set of $v$ elements (\emph{points});

(2) $S$ is a subset of $X$ (the \emph{stem}) of size $s$;

(3) $\mathcal{G} = \{G_{1}, G_{2},\ldots\}$ is a set of nonempty subsets (\emph{groups}) of $X\backslash S$ that partition $X\backslash S$;

(4) $\mathcal{A}$ is a family of subsets (\emph{blocks}) of $X$, each having cardinality from $K$; and

(5) each $t$-subset $T$ of $X$ with $|T \cap(S \cup G_{i})| < t$ for all $i$ is contained in a unique block and no $t$-subset of $S \cup G_{i}$ for any $i$ is contained in any block.

\noindent The  \emph{type} of a $t$-CS is the multiset list ($\{|G|: G \in \mathcal{G}\}: |S|$), listing the group sizes and stem size. When a $t$-CS has $n_{i}$ groups of size
$g_{i}$, $1 \leq i \leq r$, and stem size $s$, we use the notation $(g^{n_{1}}_{1} g^{n_{2}}_{2} \cdots g^{n_{r}}_{r} : s)$.
%A CS$(t,K,v)$ of type $(1^{v}:0)$ $(X, S, \mathcal{G}, \mathcal{A})$ is usually called a \emph{$t$-wise balanced design} and briefly denoted by S$(t,K,v)$.
We primarily employ candelabra systems with $t = 3$
and $K = \{3\}$, which will be denoted by CS$(g^{n_{1}}_{1} g^{n_{2}}_{2} \cdots g^{n_{r}}_{r} : s)$. %The block set of a CS$(g^{n_{1}}_{1} g^{n_{2}}_{2} \cdots g^{n_{r}}_{r} : s)$ consists of all triples that are not from the union of any group and the stem.

A \emph{$t$-wise balanced design} ($t$BD) with parameters $t$-$(v, K, \lambda)$ is a pair $(X, \mathcal{B})$ where $X$ is a set of $v$ points and $\mathcal{B}$ is a collection of subsets (blocks) of $X$ with the property that the size of every block is in the set $K$ and every $t$-subset of $X$ is contained in exactly $\lambda$ blocks. A $t$-$(v, K, \lambda)$ design is also denoted by S$_{\lambda}(t,K,v)$. When $K = \{k\}$, we simply write as S$_{\lambda}(t,k,v)$.

%Let $v$ be a nonnegative integer, $t$ be a positive integer, and $K$ be a set of positive integers.
A \emph{group divisible $t$-design} (or $t$-GDD) of order $v$, index $\lambda$, and block sizes from $K$, denoted by GDD$_{\lambda}(t,K,v)$, is a triple $(X, \mathcal{G}, \mathcal{B})$ such that

(1) $X$ is a set of $v$ elements (\emph{points});

(2) $\mathcal{G}$ is a set of nonempty subsets (\emph{groups}) of $X$ that partition $X$;

(3) $\mathcal{B}$ is a family of subsets (\emph{blocks}) of $X$, each of cardinality from $K$, such that each block intersects any given group in at most one point; and

(4) each $t$-subset of points from $t$ distinct groups is contained in exactly $\lambda$ blocks.

\noindent  %The \emph{type} of the GDD is the multiset list $(|G|:G \in \mathcal{G})$ of group sizes.
A GDD of type $g^{t_{1}}_{1}\cdots g^{t_{s}}_{s}$ is a GDD in which there are exactly $t_{i}$ groups of cardinality $g_{i}$ for $1 \leq i \leq s$.
When $\lambda$ is omitted in the notation, we mean $\lambda=1$.

A candelabra system CS$(g^{n}$~: $s)~(X,S,\mathcal{G},\mathcal{A})$ with $s\geq 2$, is called {\em partitionable} and denoted by PCS$(g^{n}: s)$ (as in \cite{PCS}), if the block set $\mathcal{A}$ can be partitioned into $\mathcal{A}_{x}$, $x\in G, G \in \mathcal{G},$ and $\mathcal{A}_{1}, \mathcal{A}_{2},\ldots, \mathcal{A}_{s-2}$ so that

(i) for each $x\in G$ and $G\in \mathcal{G}$, $\mathcal{A}_{x}$ is the block set of a GDD$(2, 3, gn + s)$ of type $1^{gn-g}(g + s)^{1}$ with $G\cup S$ as the long group; and

(ii) for $1\leq i\leq s-2$, $(X\setminus S, \mathcal{G}, \mathcal{A}_{i})$ is a GDD$(2, 3, gn)$ of type $g^{n}$.

 Let $a,g$ be  positive integers and let $(X, S, \mathcal{G}, \mathcal{A})$ be a CS$((2ag)^{n}:s)$. It is said to be {\em $g$-partitionable with index two} and denoted by $g$-PCS$_{2}((2ag)^{n} : s)$ if its block set $\mathcal{A}$ contains a subset $\mathcal{A}'$ and $\mathcal{A}'$ can be partitioned into $gn$ pairwise disjoint parts $\mathcal{A}_{i}$, $0 \leq i \leq gn-1$, such that, for each group $G\in \mathcal{G}$, there are exactly $g$ $\mathcal{A}_{i}$'s such that each $\mathcal{A}_{i}$ is the block set of a GDD$_{2}(2, 3, 2agn + s)$ of type $1^{2ag(n-1)}(2ag + s)^{1}$ with $G\cup S$ as its long group.

Making use of a PCS$(g^{n}: s)$ or a $g$-PCS$_{2}((2ag)^{n} : s)$ as a master design will enable us to produce NWBTSs. By this in Section 4 we determine the existence of NWBTSs for (C1) and in Sections 5 and 6 we will define candelabra systems with another several types of partitions to treat (C2). Before presenting the constructions, we prove two auxiliary lemmas.

\begin{lemma}\label{tee}
 Suppose that $(V,\mathcal{F})$ is an NWBTS$(v; b)$ where $b<\binom{v}{3}$. Further suppose that $(V,\mathcal{F}')$ is a triple system comprising of $q$ pairwise disjoint simple S$_{p}(2,3,v)$s. If $\mathcal{F}$ and $\mathcal{F}'$ are disjoint, then the triple system $(V,\mathcal{F}\cup\mathcal{F}')$ forms an NWBTS$(v; c)$ where $c=b+\frac{pqv(v-1)}{6}$.
\end{lemma}
\proof Clearly $(V,\mathcal{F})$ is a simple NWBTS$(v; b)$ and $(V,\mathcal{F}\cup\mathcal{F}')$ is a 3-balanced TS$(v; c)$ with $c=b+\frac{pqv(v-1)}{6}$. Let $(v,b)$ associate  with $(\lambda,\varepsilon)$. It follows from the construction that the pairs of $V$ appear in $pq+\lambda-1,pq+\lambda$, or $pq+\lambda+1$ blocks of $\mathcal{F}\cup\mathcal{F}'$ and that the defect graph of $\mathcal{F}\cup\mathcal{F}'$ is the same as that of $\mathcal{F}$. Noticing $3c=\frac{(pq+\lambda)v(v-1)}{2}+\varepsilon$ yields that $\mathcal{F}\cup\mathcal{F}'$ is  nearly 2-balanced. As a result, $\mathcal{F}\cup\mathcal{F}'$ is  an NWBTS$(v; c)$.
\qed

\begin{lemma}\label{tee2}

Let $v\equiv u\equiv 2,5\pmod{6}$, $c<\binom{u}{3}$, and $(u,c)$ satisfy (C1) associating with  $(\lambda,\varepsilon)$.
 %if $\lambda \equiv 1$ {\rm(mod 3)} and $\varepsilon\in\{-2,1\}$ if $\lambda \equiv 2$ {\rm(mod 3)}.
  Suppose that $(U,{\cal F})$ is an NWBTS$(u; c)$ and  $(V,\mathcal{F}')$ is a simple GDD$_{\lambda}(2,3,v)$ of type $1^{v-u}u^{1}$
with $U\subset V$ as its long group. Then $(V,{\cal F}\cup{\cal F}')$ forms an NWBTS$(v; b)$ where $3b=\lambda {v\choose 2}+\varepsilon$.
\end{lemma}
\proof Clearly $(V,{\cal F}\cup{\cal F}')$  has $b$ blocks where $$3b=3c+\frac{\lambda (v(v-1)-u(u-1))}{2}= \frac{\lambda u(u-1)}{2}+\varepsilon+\frac{\lambda (v(v-1)-u(u-1))}{2}=\lambda {v\choose 2}+\varepsilon.$$ It is readily checked that $\lambda_{x,y}^{{\cal F}\cup{\cal F}'}\in\{\lambda-1,\lambda,\lambda+1\}$ for any pair $\{x,y\}$ of $V$ and ${\cal F}\cup{\cal F}'$ has the same defect graph  as that of ${\cal F}$, which is isomorphic to $G_{\varepsilon}$ in Figure \ref{f1}. As a result, $\mathcal{F}\cup\mathcal{F}'$ is nearly 2-balanced. It is also simple and thus 3-balanced as $c<\binom{u}{3}$ and $\mathcal{F}'$ is simple. This proves that $(V,{\cal F}\cup{\cal F}')$ forms an NWBTS$(v; b)$.
\qed

\begin{construction}\label{tb}
Suppose that there exists a PCS$(g^{n}:s)$ where $g\equiv 0 \pmod{3}$, $s\equiv2\pmod{3}$, and  $g,n\ge 3$. Let $\lambda=1,2$.
Further suppose that there exist
\begin{enumerate}
\item[(i)] $l$ disjoint simple S$_{3}(2,3,g+s)$s defined over a $(g+s)$-set $X$ with each  block not strictly contained in an $s$-subset $Y\subset X$, where $l=1$ if $g=3$ and $l=2$ otherwise; and

\item[(ii)] an NWBTS$(g+s;c)$ where $(g+s,c)$ satisfies (C1) associating with $(\lambda,\varepsilon)$. %$3b_{\lambda}=\lambda\binom{g+s}{2}+\varepsilon$ with $\varepsilon\in\{-1,2\}$ if $\lambda=1$  and $\varepsilon\in\{-2,1\}$ if $\lambda= 2$.
\end{enumerate}
\noindent Then there exists an NWBTS$(gn+s;b)$ for any $b=q\binom{gn+s}{2}+b'$ where $0\leq q\leq l(n-1)$ and $3b'=\lambda\binom{gn+s}{2}+\varepsilon$.
\end{construction}

\proof
Let $V = (Z_{n} \times Z_{g})\cup S$, $|S|=s$, and $\mathcal{G} = \{G_{0}, G_{1},\ldots,G_{n-1}\}$ where $G_{i} = \{i\}\times Z_{g}$.
Let $(V, S, \mathcal{G}, \mathcal{A})$ be a PCS$(g^{n} : s)$. Then the block set $\mathcal{A}$ can be partitioned into $gn$ pairwise disjoint parts
 $\mathcal{A}(i, j)$, $0\leq i\leq n-1$ and $0\leq j\leq g-1$, such that each $\mathcal{A}(i, j)$ is the block set of a GDD$(2, 3, gn+s)$ of type $1^{g(n-1)}(g+s)^{1}$
 with long group $G_{i}\cup S$. For $\lambda=1,2$, construct an NWBTS$(g+s;c)$ on $G_{0}\cup S$ with block set ${\cal C}_{\lambda}$ where
 $3c=\lambda\binom{g+s}{2}+\varepsilon$, which exists by assumption. %Then ${\cal B}_{\lambda}$ has a defect graph $D$ isomorphic to $G_{\varepsilon}$.
Let $\mathcal{F}=\mathcal{A}(0,0)\cup{\cal C}_1$ if $\lambda=1$ and $\mathcal{F}=\mathcal{A}(0,0)\cup\mathcal{A}(0,1)\cup{\cal C}_2$ if $\lambda=2$. Then by Lemma \ref{tee2}, $\mathcal{F}$ is an NWBTS$(gn+s;b')$ where $3b'=\lambda\binom{gn+s}{2}+\varepsilon$. This proves the case of $q=0$.

For $1\leq i\leq n-1$, by assumption we may construct $l$ disjoint simple S$_{3}(2,3,g+s)$s over $G_{i}\cup S$ with block sets ${\cal C}(i,j),1\le j\le l,$ and each block is not contained in $S$.  Let $$\mathcal{F}(i,j) = \mathcal{A}(i, 3(j-1))\cup \mathcal{A}(i, 3(j-1)+1)\cup \mathcal{A}(i,3(j-1)+ 2)\cup \mathcal{C}(i,j),\ 1\le j\le l.$$ It is clear that $\mathcal{F}(i,j)$  is the block set of a simple S$_{3}(2,3,gn + s)$ and all $\mathcal{F}(i,j)$s are mutually disjoint. For $1\le q\le l(n-1)$, construct  $\mathcal{F}'$ by taking the union of  $\mathcal{F}$ and $q$ block sets in $\{\mathcal{F}(i,j): 1\leq i\leq n-1,1\le j\le l\}$. Then we get an NWBTS$(v; b)$ by Lemma \ref{tee} for any $b=q\binom{gn+s}{2}+b'$ where  $0\leq q\leq l(n-1)$ and $3b'=\lambda\binom{gn+s}{2}+\varepsilon$.\qed

\begin{construction}\label{tbb} Suppose that there exists  a $g$-PCS$_{2}((2ag)^{n}:2)$ where $g\equiv 0 \pmod{3}$ and $g, n\ge 3$.
Denote $v=2agn+2$ and let $\lambda=2,4$. Suppose that there exists a simple S$_{6}(2,3,2ag+2)$ and
an NWBTS$(2ag+2;c)$ where $(2ag+2,c)$ satisfies (C1) associating with $(\lambda,\varepsilon)$. % with $\varepsilon\in\{-2,1\}$ if $\lambda=2$ or  $\varepsilon\in\{-1,2\}$ if $\lambda=4$.
\noindent Then there exists an NWBTS$(v;b)$ for any $b=qv(v-1)+ b'$ where $0\leq q\leq n-1$  and
$3b'=\lambda{v\choose 2}+\varepsilon$.
\end{construction}

\proof Let $V = (Z_{n} \times Z_{g}\times Z_{2a})\cup S$, $|S|=2$, and $\mathcal{G} = \{G_{0}, G_{1},\ldots,G_{n-1}\}$ where $G_{i} = \{i\}\times Z_{g}\times Z_{2a}$. Let $(V, S, \mathcal{G}, \mathcal{A})$ be a $g$-PCS$_{2}((2ag)^{n} : 2)$. Then the block set $\mathcal{A}$ contains a subset $\mathcal{A}'$ and $\mathcal{A}'$ can be partitioned into $gn$ pairwise disjoint parts  $\mathcal{A}(i, j)$, $0\leq i\leq n-1$ and $0\leq j\leq g-1$, such that each $\mathcal{A}(i, j)$ is the block set of a GDD$_{2}(2, 3, v)$ of type $1^{2ag(n-1)}(2ag+2)^{1}$ with long group $G_{i}\cup S$. By assumption, for  $\lambda=2,4$, there exists an NWBTS$(2ag+2;c)$ on $G_{0}\cup S$ with block set ${\cal C}_{\lambda}$ where $3c=\lambda\binom{2ag+2}{2}+\varepsilon$.
Construct a family $\mathcal{F}$ containing all blocks of ${\cal C}_{\lambda}$ and all blocks of $\lambda/2$ sets in $\{\mathcal{A}(0,j):j=0,1\}$. Then by Lemma \ref{tee2}, $\mathcal{F}$ is an NWBTS$(v;b')$  where $3b'=\lambda{v\choose 2}+\varepsilon$. This proves the case of $q=0$.

For each $1\leq i \leq n- 1$, construct a simple S$_{6}(2,3,2ag+2)$  on $G_{i} \cup S$ with block set $\mathcal{B}(i)$, which exists by assumption.
Let $\mathcal{F}(i) = (\cup^{2}_{j=0}\mathcal{A}(i, j )) \cup \mathcal{B}(i)$. Then each $\mathcal{F}(i)$ is the block set of a simple
S$_{6}(2,3,v)$. Construct a family $\mathcal{F}'$ containing $q$ $\mathcal{F}(i)$s.
Then $\mathcal{F} \cup \mathcal{F}'$ is an NWBTS$(v;b)$ by Lemma \ref{tee} for any $b=qv(v-1)+ b'$ where $0\leq q\leq n-1$ and $3b'=\lambda{v\choose 2}+\varepsilon$.
\qed

\section{NWBTSs for (C1)}

This section shows the existence of an NWBTS$(v;b)$ if $(v,  b)$ satisfies (C1).

%\subsection{ $v \equiv 5$ (mod 6)}

\begin{lemma}\label{5}
There exists an NWBTS$(5; b)$ where $b\equiv3,4,6,7 \pmod{10}$ and an NWBTS$(11; b)$ where $b\equiv18,19,36,37 \pmod{55}$.
\end{lemma}
\proof
From Lemma \ref{bb}, we only need to consider the existence of an NWBTS$(v; b)$ where $b\in\{3,4\}$ if $v=5$ and   $b\in\{18,19,36,37,73,74\}$ if $v=11$.  The case of $v=5$ is solved by  \cite{wbd} and the needed small examples of $v=11$  exist by Lemma \ref{11} in Appendix.
\qed
%
%\begin{lemma}\rm{\cite[Lemma 3]{JI}}\label{352}
%There exists a PCS$(3^{5} : 2)$.
%\end{lemma}

\begin{lemma}\label{17}
There exists an NWBTS$(17; b)$ whenever $b \equiv 45, 46,90, 91 \pmod{136}$.
\end{lemma}

\proof  By Lemma \ref{bb}, we can assume that $b \leq 680$. %Apply  Construction \ref{tb} by taking $g=3,n=5,s=2,l=1,\lambda\in\{1,2\},a=42,b_1\in\{3,4\},b_2\in\{6,7\}$.
There exists a PCS$(3^{5} : 2)$ from \cite[Example 2.3]{Lsts}.  An NWBTS$(5; c)$ exists for any $c$ if $(5,c)$ satisfies (C1) from Lemma \ref{5}. Let  $(5,c)$ associate with $(\lambda,\varepsilon)$ where $\lambda=1,2$. A simple S$_{3}(2,3,5)$ exists trivially. Apply Construction \ref{tb}  to get an NWBTS$(17;b)$ for  $3b=136q+{136\lambda+\varepsilon\over 3}$, where $0\le q\le 4$, $\varepsilon\in\{-1,2\}$ if $\lambda=1$  and $\varepsilon\in\{-2,1\}$ if $\lambda= 2$. This covers all admissible block numbers $b \equiv 45, 46,90, 91 \pmod{136}$ and  $b \leq 680$. \qed

%
%\begin{theorem}\rm{\cite{Lsts,JI}}\label{6n5}
%There exists a PCS$(6^{n} : 5)$ whenever $n\geq3$.
%\end{theorem}

%\begin{lemma}\rm{\cite[Lemma 5.4]{wbt}}\label{XY}
%Let $X$ be a set of 11 points and $Y$ be a subset of $X$ of size 5. Then, all
%triples of $\binom{X}{3}\setminus\binom{Y}{3}$ can be partitioned into two STS$_{3}(11)$s and one GDD$_{3}(2, 3, 11)$ of type $1^{6}5^{1}$.
%\end{lemma}

\begin{lemma}\label{v56}
Let $v \equiv 5 \pmod{6}$ and  $(v,  b)$ satisfy (C1). An NWBTS$(v; b)$ exists.
\end{lemma}

\proof Lemmas \ref{5} and \ref{17} proved the cases $v=5,11,17$. Next we let $v\ge 23$. By Lemma \ref{bb}, we can assume that $b\leq \frac{v(v-1)(v-2)}{12}$. Write $v = 6n + 5$ and $b = q\frac{v(v-1)}{2} + b'$ where $0 \leq q \leq \left \lfloor\frac{v-2}{6}\right \rfloor=n$ and $b' \in\{ \lfloor \frac{\lambda v(v-1)}{6}\rfloor,
\lceil \frac{\lambda v(v-1)}{6} \rceil:\lambda=1,2\}$.
Thus $(v,b')$ associates with $(\lambda,\varepsilon)$ where $\lambda=1,2$.
Apply Construction \ref{tb} from
a PCS$(6^{n}:5)$, which exists by \cite[Lemma 11]{JI}. An NWBTS$(11; c)$ exists for any $c$ if $(11,c)$ satisfies (C1)  from Lemma \ref{5}. There exist two disjoint simple S$_{3}(2,3,11)$s defined over a 11-set with each block not contained in a fixed 5-subset  from \cite[Lemma 5.4]{wbt}.
Then there exists an NWBTS$(v; b)$, where $b=q\frac{v(v-1)}{2}+b'$ with $0\le q\le 2(n-1)$; this completes the proof.\qed

%\subsection{$v \equiv 2$ (mod 6)}

%\begin{lemma}\rm{\cite[Lemma 6.8]{wbt}}\label{6n2}
%There exists a $3$-PCS$_{2}(6^{n}:2)$ for every $n\geq 3$, $n \equiv 0, 1, 3 \pmod{4}$, and $n\not\in \{8, 12\}$.
%\end{lemma}

\begin{lemma}\label{1v24}
Let $v \equiv 2,8,20 \pmod{24}$, $v \not\in\{50,74\}$, and   $(v,  b)$ satisfy (C1). There exists an NWBTS$(v; b)$.
\end{lemma}

\proof For $v = 8$, see Lemma \ref{81} in Appendix. For $v \geq 20$,
we can assume that $b\leq \frac{v(v-1)(v-2)}{6}$ by Lemma \ref{bb}. Write $v = 6n +2$ and $b = qv(v-1) + b'$ where $0\leq q\leq \frac{v-2}{6}-1 = n- 1$, $0\leq b' \leq v(v-1)$. According to Lemma \ref{bb}, we can assume that $0 \leq b' \leq \frac{v(v-1)}{2}$, i.e., $b'\in\{\lfloor \frac{v(v-1)}{3}\rfloor$, $\lceil \frac{v(v-1)}{3} \rceil\}$.
Thus $(v,b')$ associates with $(2,\varepsilon)$.
There exists from \cite[Lemma 6.8]{wbt}  a $3$-PCS$_{2}(6^{n}:2)$ for every $n \equiv 0, 1, 3 \pmod{4}$, $n\geq 3$,  and $n\not\in \{8, 12\}$. An NWBTS$(8; c)$ exists for any $c$ if $(8,c)$ satisfies (C1)  from Lemma \ref{81} in Appendix; and a simple S$_{6}(2,3,8)$ trivially exists.
Then there exists an NWBTS$(v;b)$ by Construction \ref{tbb}, where $b=qv(v-1)+b'$ and $0\le q\le n-1$.
\qed

%\begin{lemma}\rm{\cite[Lemma 6.14]{wbt}}\label{12n2}
%Let $n\equiv1 \pmod{2}$ and $n\geq3$. Then there exists a $3$-PCS$_{2}(12^{n}:2)$.
%\end{lemma}

\begin{lemma}\label{1v241}
Let $v \equiv 14 \pmod{24}$ and   $(v,  b)$ satisfy (C1). There exists an NWBTS$(v; b)$.
\end{lemma}

\proof For $v = 14$, see Lemma \ref{14} in Appendix.
For $v \geq 38$, we only need to consider $b\leq \frac{v(v-1)(v-2)}{12}$ by Lemma \ref{bb}. Write $v = 12n +2$ and $b = qv(v-1) + b'$ where $0\leq q\leq \frac{v-2}{12}-1 = n- 1$ and $b'\in\{\lfloor \frac{\lambda v(v-1)}{6} \rfloor$, $\lceil \frac{\lambda v(v-1)}{6}\rceil:\lambda=2,4\}$. Thus $(v,b')$ associates with $(\lambda,\varepsilon)$ where $\lambda=2,4$.
There exists a $3$-PCS$_{2}(12^{n} : 2)$ for $n\equiv1 \pmod{2}$ and $n\geq3$ from \cite[Lemma 6.14]{wbt}. An NWBTS$(14; c)$ exists  for any $c$ if $(14,c)$ satisfies (C1)  from Lemma \ref{14} in Appendix.  A simple S$_{6}(2,3,14)$ exists as a WBTS$(14;182)$ exists by Theorem \ref{NF}.
Then there exists an NWBTS$(v;b)$ by Construction \ref{tbb}, where $b=qv(v-1)+b'$ and $0\le q\le n-1$.\qed

\begin{lemma}\label{1v242}
Let $v \in\{50,74\}$ and   $(v,  b)$ satisfy (C1). There exists an NWBTS$(v; b)$.
\end{lemma}
\proof See Lemma \ref{t1v242} in Appendix.

Combing Lemmas \ref{v56}-\ref{1v242} we have established the main result of this section.

\begin{theorem}\label{v26}
 An NWBTS$(v; b)$ exists if  $(v,  b)$ satisfies (C1).
\end{theorem}

\section{Case $v \equiv 0\pmod 6$ for (C2)}

In this section we define a variant of partitionable candelabra systems with an empty stem, present a product-type construction  and prove the existence of an NWBTS$(v;b)$ for any $v \equiv 0$ (mod 6) and $(v,  b)$ satisfying (C2).

Let $m,n$, and $u$ be nonnegative integers. Suppose that $V$ is a set of $v=2(m+n)+u$ points and ${\cal G}=\{G_{1},\ldots, G_{m+n}, U\}$ is a partition of $V$ into groups where $|G_{i}|=2~(1\leq i \leq m+n)$ and $|U|=u$. Let $\mathcal{B}$ be a set of triples (blocks) on $V$. Then $(V,{\cal G},\mathcal{B})$ is a \emph{nearly} GDD$(2,3,v)$, simply by NGDD$(2,3,v)$, of type $2^{(m,n)}u^{1}$ if

(1) any pair of $G_{i}$ with $1\leq i\leq m$ is contained in exactly two blocks of $\mathcal{B}$,

(2) any pair in $\binom{U}{2}\cup \{G_{i}:m+1\leq i\leq m+n\}$ is contained in no blocks, and

(3) any of the other pairs of $V$ is contained in exactly one block.

Let $0\leq r\leq 3g$ be an integer. A CS$((2g)^{3}:0)$ $(X, \emptyset, \mathcal{G}, \mathcal{A})$ is partitionable and denoted by PICS$_{2}^{r}((2g)^{3}:0)$
 if its block set $\mathcal{A}$ can be partitioned into $4g-2$ subsets $\mathcal{A}_{1}, \mathcal{A}_{2},\ldots, \mathcal{A}_{4g-2}$ with the properties: (1) for $1\leq i \leq 2g$, each $\mathcal{A}_{i}$ is the block set of an S$_{2}(2,3,6g)$, and at least one of $\mathcal{A}_{i}$'s has a subset $\mathcal{A}'$ such that $\mathcal{A}'$ is the block set of an NGDD$(2, 3, 6g)$ of type $2^{(r,3g-r)}$; (2) for $2g + 1 \leq i \leq 4g - 2$, each $\mathcal{A}_{i}$ is the block set of a GDD$(2, 3, 6g)$ of type $(2g)^{3}$ with the group set $\mathcal{G}$. It is obvious that
a PICS$_{2}^{r}((2g)^{3}:0)$ is also a PICS$_{2}^{3g-r}((2g)^{3}:0)$.
%Wei et al. \cite{wbt} give some results for PICS$_{2}^{0}((2g)^{3}:0)$, also a PICS$_{2}^{3g}((2g)^{3}:0)$.
%
%\begin{lemma}\rm{\cite[Lemma 7.1]{wbt}}\label{230}
%There exists a PICS$_{2}^{r}(2^{3}:0)$, where $r=0, 3$.
%\end{lemma}

\begin{lemma}\label{430r}
There exists a PICS$_{2}^{r}(4^{3}:0)$ for any $r=0, 3, 6$.
\end{lemma}
\proof For $r=0, 6$, there exists a PICS$_{2}^{r}(4^{3}:0)$ by \cite[Lemma 7.2]{wbt}. Next, a PICS$_{2}^{3}(4^{3}:0)$ can be constructed on $Z_{2} \times Z_{2} \times Z_{3}$ with group set $\mathcal{G} = \{Z_{2} \times Z_{2} \times \{j\} :j = 0, 1, 2\}$. Here we write a point $(x,y,i)$ as $(x,y)_i$.

We first construct two GDD$(2, 3, 12)$s of type $4^{3}$ with the group set $\mathcal{G}$. Their blocks
are generated by two sets of base blocks (each set in a line) under the action of the group $Z_{2} \times Z_{2}$.
{\small\small\begin{longtable}{llll}
$\{(0, 0)_{0}, (0, 0)_{1}, (0, 0)_{2}\}$ & $\{(0, 0)_{0}, (0, 1)_{1}, (1, 0)_{2}\}$ & $\{(0, 0)_{0}, (1, 0)_{1}, (1, 1)_{2}\}$ & $\{(0, 0)_{0}, (1, 1)_{1}, (0, 1)_{2}\}$\\
$\{(0, 0)_{0}, (0, 0)_{1}, (0, 1)_{2}\}$ & $\{(0, 0)_{0}, (0, 1)_{1}, (1, 1)_{2}\}$ & $\{(0, 0)_{0}, (1, 0)_{1}, (1, 0)_{2}\}$ & $\{(0, 0)_{0}, (1, 1)1_{}, (0, 0)_{2}\}$\\
\end{longtable}}

\noindent Then, we construct an initial S$_{2}(2,3,12)$ on $Z_{2} \times Z_{2} \times Z_{3}$ with the following blocks; the underlined blocks form the blocks of an NGDD$(2, 3, 12)$ of type $2^{(3,3)}$, where $\lambda_{p}=2$ for $p\in\{\{(0,0)_{0},(0,1)_{0}\}$, $\{(0,0)_{1},(0,1)_{1}\}, \{(0,0)_{2},(0,1)_{2}\}\}$ and $\lambda_{q}=0$ for $q\in\{\{(1,0)_{0},(1,1)_{0}\}$, $\{(1,0)_{1}$, $(1,1)_{1}\}$, $\{(1,0)_{2}$, $(1,1)_{2}\}\}$.
{\small\begin{longtable}{llll}
$\underline{\{(0, 0)_{0}, (0, 1)_{0}, (1, 0)_{1}\}}$ & $\underline{\{(0, 0)_{0}, (0, 1)_{0}, (0, 1)_{2}\}}$ & $\underline{\{(0, 0)_{0}, (1, 0)_{0}, (0, 1)_{1}\}}$ & $\underline{\{(0, 0)_{0}, (1, 1)_{0}, (1, 1)_{1}\}}$\\
$\underline{\{(0, 0)_{0}, (0, 0)_{2}, (1, 1)_{2}\}}$ & $\underline{\{(0, 0)_{0}, (0, 0)_{1}, (1, 0)_{2}\}}$ & $\underline{\{(0, 1)_{0}, (1, 1)_{1}, (0, 0)_{2}\}}$ & $\underline{\{(0, 1)_{0}, (1, 0)_{0}, (1, 0)_{2}\}}$\\
$\underline{\{(0, 1)_{0}, (0, 0)_{1}, (0, 1)_{1}\}}$ & $\underline{\{(0, 0)_{1}, (0, 1)_{1}, (1, 1)_{2}\}}$ & $\underline{\{(0, 1)_{0}, (1, 1)_{0}, (1, 1)_{2}\}}$ & $\underline{\{(1, 0)_{0}, (0, 0)_{1}, (1, 0)_{1}\}}$\\
$\underline{\{(1, 0)_{0}, (1, 1)_{1}, (1, 1)_{2}\}}$ & $\underline{\{(1, 0)_{0}, (0, 0)_{2}, (0, 1)_{2}\}}$ & $\underline{\{(0, 1)_{1}, (0, 0)_{2}, (0, 1)_{2}\}}$ & $\underline{\{(1, 1)_{0}, (0, 0)_{1}, (0, 0)_{2}\}}$\\
$\underline{\{(1, 1)_{0}, (0, 1)_{1}, (1, 0)_{1}\}}$ & $\underline{\{(1, 1)_{0}, (0, 1)_{2}, (1, 0)_{2}\}}$ & $\underline{\{(0, 0)_{1}, (1, 1)_{1}, (0, 1)_{2}\}}$ & $\underline{\{(0, 1)_{1}, (1, 1)_{1}, (1, 0)_{2}\}}$\\
$\underline{\{(1, 0)_{1}, (0, 0)_{2}, (1, 0)_{2}\}}$ & $\underline{\{(1, 0)_{1}, (0, 1)_{2}, (1, 1)_{2}\}}$ &
$\{(0, 0)_{0}, (0, 0)_{1}, (1, 1)_{2}\}$ & $\{(1, 0)_{0}, (1, 1)_{0}, (1, 0)_{1}\}$\\
$\{(0, 1)_{0}, (1, 1)_{0}, (0, 1)_{1}\}$ & $\{(1, 1)_{0}, (0, 0)_{1}, (1, 1)_{1}\}$ & $\{(0, 0)_{0}, (0, 1)_{1}, (0, 0)_{2}\}$ & $\{(0, 1)_{0}, (1, 0)_{0}, (0, 0)_{1}\}$\\
$\{(1, 0)_{0}, (0, 1)_{1}, (1, 1)_{1}\}$ & $\{(0, 1)_{0}, (1, 1)_{1}, (0, 1)_{2}\}$ & $\{(0, 0)_{0}, (1, 0)_{1}, (1, 1)_{1}\}$ & $\{(0, 1)_{0}, (1, 0)_{1}, (1, 1)_{2}\}$\\
$\{(1, 0)_{0}, (1, 1)_{0}, (0, 0)_{2}\}$ & $\{(0, 0)_{0}, (1, 0)_{0}, (0, 1)_{2}\}$ & $\{(0, 0)_{0}, (1, 1)_{0}, (1, 0)_{2}\}$ & $\{(0, 1)_{0}, (0, 0)_{2}, (1, 0)_{2}\}$\\
$\{(1, 0)_{0}, (1, 0)_{2}, (1, 1)_{2}\}$ & $\{(1, 1)_{0}, (0, 1)_{2}, (1, 1)_{2}\}$ & $\{(1, 0)_{1}, (1, 1)_{1}, (1, 0)_{2}\}$ & $\{(0, 0)_{1}, (1, 0)_{1}, (0, 0)_{2}\}$\\
$\{(0, 0)_{1}, (0, 1)_{2}, (1, 0)_{2}\}$ & $\{(1, 1)_{1}, (0, 0)_{2}, (1, 1)_{2}\}$ & $\{(0, 1)_{1}, (1, 0)_{1}, (0, 1)_{2}\}$ & $\{(0, 1)_{1}, (1, 0)_{2}, (1, 1) _{2}\}$\\
\end{longtable}}

\noindent We obtain four S$_{2}(2,3,13)$s under the action of the group $Z_{2} \times Z_{2}$. All block sets of
these GDD$(2, 3, 12)$s and S$_{2}(2,3,12)$s are pairwise disjoint and they form a CS$(4^{3}:0)$.
\qed

%Some preliminary results can be used in the recursive construction.

Let $X$ be a set of cardinality $tu$, and $\mathcal{H}$ be a partition of $X$ into $u$ subsets of size
$t$ (elements of $\mathcal{H}$ are \emph{holes}). Let $L$ be a square array of size $tu$, indexed by $X$, which
satisfies:
(1) if $x, y \in H \in \mathcal{H}$, then $L(x, y)$ is empty, otherwise $L(x, y)$ contains a symbol of $X$; and
(2) row or column $x$ of $L$ contains the symbols in $X\backslash H$, where $x \in H \in \mathcal{H}$.
$L$ is called a \emph{partitioned incomplete Latin square} (or PILS) of type $t^{u}$. $L$ is \emph{symmetric} if
$L(x, y) = L(y, x)$ for all $x, y$ not in the same hole.

%\begin{theorem}\rm{\cite{Fu}}\label{Fu}
%Suppose $t$ is even and $u \geq 3$. Then a symmetric PILS of type $t^{u}$ exists.
%\end{theorem}
%
%\begin{theorem}\rm{\cite{1}}\label{te}
%For each integer $u > 2$ with $u \neq 6$, there exists a GDD$(2, 3, 3u)$ of type $u^{3}$ whose block set can be partitioned into parallel classes.
%\end{theorem}

Now we give a recursive construction for PICS$^{r}_{2}$, which is similar to the constructions in \cite{Chen, wbt, Lsts}.

\begin{construction}\label{C1}
Suppose that there exists a PICS$^{r}_{2}((2m)^{3}:0)$ for any $0\leq r\leq 3m$ with $r\equiv 0\pmod{3}$. Then for each integer $u > 2$ and $u \neq6$, there exists a PICS$^{e}_{2}((2mu)^{3}:0)$ for any $0\leq e\leq 3mu$ and $e\equiv 0\pmod{3}$.
\end{construction}

\proof We construct a PICS$^{e}_{2}((2mu)^{3}:0)$ on $X = Z_{2m} \times Z_{u} \times Z_{3}$, with the group set $\mathcal{G} = \{Z_{2m} \times Z_{u} \times \{i\} : 0 \leq i \leq 2\}$ for $0\leq e\leq 3mu$ and $e\equiv 0\pmod{3}$. The construction proceeds in three steps.

Step 1: Let $(Z_{u}\times Z_{3}, \mathcal{G}', \mathcal{A}_{0})$ be a resolvable GDD$(2, 3, 3u)$ of type $u^{3}$, which exists from \cite{1}, where $\mathcal{G}' = \{Z_{u} \times \{i\} : 0 \leq i \leq 2\}$. For each $h \in Z_{u}$, define $\mathcal{A}_{h} = \{\{(a, 0),(b, 1),(c + h, 2)\} : \{(a, 0),(b, 1),(c, 2)\} \in \mathcal{A}_{0}\}$. Then each $(Z_{u} \times Z_{3}, \mathcal{G}', \mathcal{A}_{h})$ is a GDD$(2, 3, 3u)$, and $\{\mathcal{A}_{h} : h \in Z_{u}\}$ is a partition of all triples from three distinct groups.

For $A = \{(a, 0),(b, 1),(c, 2)\} \in \mathcal{A}_{h}$ with $1 \leq h \leq u - 1$, construct $2m$ pairwise disjoint
GDD$(2, 3, 6m)$s on $Z_{2m} \times A$ with groups $Z_{2m}\times \{x\}, ~x \in A$. The required
$2m$ GDD$(2, 3, 6m)$s have block set $\mathcal{B}^{i}_{A}, i \in Z_{2m}$, which contains the blocks
$\{(i_{0}, a, 0),(i_{1}, b, 1),(i_{2}, c, 2)\},$ $ i_{0} + i_{1} + i_{2} \equiv i$ (mod $2m),i_{0}, i_{1}, i_{2} \in Z_{2m}$.

For $1 \leq h \leq u - 1$ and $i \in Z_{2m}$, define $\mathcal{F}_{(u-1)i+h} = \bigcup _{A\in \mathcal{A}_{h}} \mathcal{B}^{i}_{A}$. Then for $1 \leq l \leq 2m(u - 1)$, each $(X, \mathcal{G}, \mathcal{F}_{l})$ is a GDD$(2, 3, 6mu)$ of type $(2mu)^{3}$.

Step 2: Because $\mathcal{A}_{0}$ is resolvable, we can partition it into $u$ parallel classes, $P_{0}, P_{1},\ldots,P_{u-1}$, each of which partitions $Z_{u} \times Z_{3}$. For each $A \in P_{h}$ with $h \in Z_{u}$, we may take an $e_{A}$ with $0\leq e_{A}\leq 3m$ and $e_{A}\equiv 0\pmod{3}$ such that $e=\sum_{A\in P_{h}}e_{A}$. We construct a PICS$^{e_{A}}_{2}((2m)^{3}:0)$ on $Z_{2m} \times A$ with groups $Z_{2m} \times \{x\}$, $x\in A$, which exists by assumption. Denote its block set by $\mathcal{B}_{A}$. Then $\mathcal{B}_{A}$ can be partitioned into $4m - 2$ parts $\mathcal{B}^{i}_{A}$, $1 \leq i \leq 4m - 2$, so that each $\mathcal{B}^{i}_{A}$ $(1 \leq i \leq 2m)$ is the block set of an S$_{2}(2,3,6m)$, each $\mathcal{B}^{i}_{A}$ $(2m + 1\leq i \leq 4m -2)$ is the block set of a GDD$(2, 3, 6m)$ of type $(2m)^{3}$, and there is a subset $\mathcal{B}'_{A}\subset \mathcal{B}^{1}_{A}$ that is the block set of an NGDD$(2, 3, 6m)$ of type $2^{(e_{A},3m-e_{A})}$.

For $1 \leq i\leq 2m -2$, define $\mathcal{F}_{2m(u-1)+i} = \bigcup_{A\in \mathcal{A}_{0}}\mathcal{B}^{2m+i}_{A}$. Then for $2m(u- 1) + 1 \leq l \leq 2mu- 2$, each $(X, \mathcal{G}, \mathcal{F}_{l})$ is a GDD$(2, 3, 6mu)$ of type $(2mu)^{3}$.

Step 3: For each $P_{h} \in \{P_{0}, P_{1},\ldots,P_{u-1}\}$ and $i \in Z_{3}$, define a permutation $\alpha^{i}_{h}$ of $Z_{u}$ by
$\alpha^{i}_{h}(x) = y$ if the pair $\{(x, i),(y, i + 1)\}$ is contained in some block of $P_{h}$.

Next let $L$ be a symmetric PILS of type $(2m)^{u}$, having holes $Z_{2m} \times \{h\}, h \in Z_{u}$. Such a square exists by \cite{Fu}. Now for $h \in Z_{u}, i \in Z_{2m}$, define two GDD$(2, 3, 6mu)$s of type $(6m)^{u}$ on $X$, with group set $\{Z_{2m} \times A : A \in P_{h}\}$, having block sets

$\mathcal{C}^{0}_{i,h} = \big\{\{(a, b, i),(a', b', i),(a'' + i, \alpha^{i}_{h}(b''), i + 1)\}: a, a' \in Z_{2m},~b, b' \in Z_{u},~b < b',~i \in Z_{3}$, and $(a'', b'') = L((a, b),(a', b'))\big\}$,
and

$\mathcal{C}^{1}_{i,h} = \big\{\{(a, b, i),(a', b', i),(a'' + i, (\alpha^{i-1}_{h})^{-1}(b''), i - 1)\}: a, a' \in Z_{2m},~b, b' \in Z_{u},~b < b', ~i \in Z_{3}$, and $(a'', b'') = L((a, b),(a', b'))\big\}$.

For $h \in Z_{u}$, $i \in Z_{2m}$, define $\mathcal{F}_{i,h} = (\bigcup _{A\in P_{h}} \mathcal{B}^{i+1}_{A}) \cup \mathcal{C}^{0}_{i,h} \cup \mathcal{C}^{1}_{i,h}$ and $\mathcal{F}' =(\bigcup _{A\in P_{0}} \mathcal{B}'_{A}) \cup \mathcal{C}^{0}_{00}$. Then each $(X, \mathcal{F}_{i,h})$ is an S$_{2}(2,3,6mu)$ and $\mathcal{F}'\subset \mathcal{F}_{00}$ is the block set of an NGDD$(2, 3, 6mu)$ of type $2^{(e,3mu-e)}$.

All these block sets $\mathcal{F}_{l}$ and $\mathcal{F}_{i,h}$ are pairwise disjoint, so we obtain a PICS$^{e}_{2}((2mu)^{3}:0)$ for $0\leq e \leq mu$ and $e\equiv 0\pmod{3}$.\qed

\begin{lemma}\label{2n3}
There exists a PICS$^{e}_{2}((2n)^{3}:0)$ whenever $n$ is a positive integer, $0\leq e \leq 3n$, and $e\equiv 0\pmod{3}$.
\end{lemma}
\proof There exists a PICS$_{2}^{e}(2^{3}:0)$ by \cite[Lemma 7.1]{wbt} where $e=0, 3$. For $n=2$, see Lemma \ref{430r}.
For $n \geq 3$ and $n \neq 6$, apply Construction \ref{C1} with $m = 1$ and $u = n$. For $n = 6$, apply Construction \ref{C1} with $m = 2$ and $u = 3$.\qed

\begin{theorem}\label{v06}
Let $v \equiv 0 \pmod{6}$ and $(v,  b)$ satisfy (C2). There exists an NWBTS$(v;b)$.
\end{theorem}

\proof According to Lemma \ref{bb}, we can assume that $b \leq \frac{v(v-1)(v-2)}{12}$.
Write $v = 6n$ and $b = q \frac{v(v-1)}{3} + b'$ where $0 \leq q \leq  \left \lfloor\frac{v-2}{4}\right \rfloor\leq 2n - 1$ and $\frac{v(v-1)}{6}-\frac{v}{6} < b' < \frac{v(v-1)}{6}+\frac{v}{6}$.
We first deal with the case $q=0$ by considering the parity of $b=b'$.
%We first produce a nearly 2-balanced simple triple system of order $v$ and size $b'$ and apply Lemma \ref{NBTS} to form the NWBTS.
 For any even integer $b'$ in the range $\frac{v(v-1)}{6}-\frac{v}{6} < b' < \frac{v(v-1)}{6}+\frac{v}{6}$, we may write $b'=\frac{v(v-1)}{6}-\frac{v}{6}+2e=6n^{2}-2n+2e$ where $1\leq e \leq n-1$. For any odd integer $b'$ in this range, we  write $b'=\frac{v(v-1)}{6}-\frac{v}{6}+2e+1=6n^{2}-2n+2e+1$ where $0\leq e \leq n-1$.

Start with a PICS$^{3e}_{2}((2n)^{3}:0)$ $(V,\emptyset,\mathcal{G},\mathcal{A})$ where $0\leq e \leq n$, which exists from Lemma \ref{2n3}.
Its block set contains $2n$ pairwise disjoint subsets $\mathcal{A}_{l}$, $1 \leq l \leq 2n$, and each $\mathcal{A}_{l}$ is the block set of an S$_{2}(2,3,6n)$ and $\mathcal{A}_{1}$ has a subset $\mathcal{B}_{1}$ which is the block set of an NGDD$(2, 3, 6n)$ of type $2^{(3e,3n-3e)}$. Clearly $|\mathcal{B}_{1}|=6n^{2}-2n+2e$.

For the even case  $b'=6n^{2}-2n+2e$ where $1\leq e \leq n-1$,  we now show that $\mathcal{B}_{1}$ actually forms an NWBTS$(v;b')$. Since $\mathcal{B}_{1}$ is a simple NGDD$(2, 3, 6n)$ of type $2^{(3e,3n-3e)}$ and $3b' = \binom{v}{2}+ \varepsilon$  where $\varepsilon=-3n+6e\equiv {v\over 2}\pmod 2$, we only need to check that a nearly 2-balanced TS$(v;b')$ should have a defect graph $D=D_1\cup D_{-1}$ isomorphic to $H_{v,\varepsilon}^0$ with $|D_1|=3e$ and $|D_{-1}|=3n-3e$.  Noticing $({v+2\varepsilon\over 4},  {v-2\varepsilon\over 4})=(3e,3n-3e)$ yields that $\mathcal{B}_{1}$ is nearly 2-balanced and hence nearly well-balanced. Let $\mathcal{F}=\mathcal{B}_{1}$ in this case.

We then treat the case of $b'=6n^{2}-2n+2e+1$ where $0\leq e \leq n-1$ based on the previous block set $\mathcal{B}_{1}$, whose defect graph is $D=D_1\cup D_{-1}$ (noting ${\cal B}_{1}$ is a GDD and $D_1=\emptyset$ when $e=0$). If there  exists a block $\{x,y,z\}\in \mathcal{A}_{1}\backslash \mathcal{B}_{1}$ where $\{x,y\},\{z,w\}\in D_{-1}$ for some element $w\notin\{x,y,z\}$, then we let $\mathcal{F}=\mathcal{B}_{1}\cup\{\{x,y,z\}\}$ and it is readily checked  that $\mathcal{F}$ is an NWBTS$(v;b')$ by showing its defect graph isomorphic to $H_{v,\varepsilon}^1$ where $\varepsilon=-3n+6e+3$  and $({v+2+2\varepsilon\over 4},  {v+2-2\varepsilon\over 4})=(3e+2,3n-3e-1)$.   This solves the case of odd $b'$. Now we show such a block $\{x,y,z\}$ exists. We trace back to the proof of Construction \ref{C1}. Observe that $e\le n-1$ and we need to input at least one PICS$^{r}_{2}((2m)^{3}:0)$ with $r\le 3(m-1)$ and $r\equiv0\pmod{3}$ for  Construction \ref{C1}. By the construction method we only need to check that such an input design PICS$^{r}_{2}((2m)^{3}:0)$  with $r\le 3(m-1)$ and $r\equiv0\pmod{3}$ has such a desirable property. Checking the proof of  Lemma \ref{2n3}, we only need the input designs for $m=1,2$. A PICS$^{0}_{2}((2m)^{3}:0)$  certainly satisfies this and it also holds for the example of PICS$^{3}_{2}(4^{3}:0)$ in Lemma \ref{430r}.

Finally when $1\leq q \leq 2n-1$, construct a family $\mathcal{F}'$ consisting of $q$ $\mathcal{A}_{l}$'s for $2 \leq l \leq 2n$. Then by Lemma \ref{tee}, $\mathcal{F} \cup \mathcal{F}'$ is an NWBTS$(v;b)$, where $b = q \frac{v(v-1)}{3} + b'$.\qed

\section{Cases $v \equiv 2,4\pmod  6$ for (C2)}

In this section, we define candelabra systems with another three types of partitions and give three corresponding recursive constructions via $s$-fan designs. By constructing candelabra systems with special partitions, we establish the existence of a large portion of NWBTSs for the parameters satisfying (C2) and $v \equiv 2,4\pmod 6$.

\subsection{Recursive constructions}

Let $(X, S, \mathcal{G}, \mathcal{A})$ be a CS$((2ag)^{n}:s)$ with $a,g$ being positive integers. Further let $r$ be an integer with $0\leq r \leq ag(n-1)$. Then $(X, S, \mathcal{G}, \mathcal{A})$ is called {\em $(g,r)$-partitionable with index two} and denoted by $g$-PCS$^{r}_{2}((2ag)^{n} : s)$
if its block set $\mathcal{A}$ contains a subset $\mathcal{A}'$ and $\mathcal{A}'$ can be partitioned into $gn$ pairwise disjoint parts $\mathcal{A}_{i}$, $0 \leq i \leq gn-1$, fulfilling that

 (i) for each group $G\in \mathcal{G}$, there are exactly $g$ $\mathcal{A}_{i}$'s such that each $\mathcal{A}_{i}$ is the block set of a GDD$_{2}(2, 3, 2agn + s)$ of type $1^{2ag(n-1)}(2ag + s)^{1}$ with $G\cup S$ as its long group, and

  (ii) at least one of the $gn$ GDDs contains a sub-NGDD$(2, 3, 2agn + s)$ of type $2^{(r,ag(n-1)-r)}(2ag + s)^{1}$.

%It is obvious that a $g$-PCS$^{r}_{2}((2ag)^{n} : s)$ is also a $g$-PCS$^{ag(n-1)-r}_{2}((2ag)^{n} : s)$.

We strengthen a little on the conditions of $g$-PCS$^{r}_{2}((2ag)^{n} : s)$ to define $g^{+}$-PCS$^{r}_{2}((2ag)^{n} : s)$.
Let $(X, S, \mathcal{G}, \mathcal{A})$  be a CS$((2ag)^{n}:s)$  with $a,g$ being positive integers. It is said to be {\em $(g^+,r)$-partitionable with index two} and denoted by $g^{+}$-PCS$^{r}_{2}((2ag)^{n} : s)$ if its block set $\mathcal{A}$ contains a subset $\mathcal{A}'$ and $\mathcal{A}'$ can be partitioned into  $gn+1$ pairwise disjoint parts $\mathcal{A}_{i}$, $0 \leq i \leq gn$, when it satisfies that

(i) for a fixed group $G_{0}\in \mathcal{G}$, there are exactly $g+1$ $\mathcal{A}_{i}$'s such that each $\mathcal{A}_{i}$ is the block set of a GDD$_{2}(2, 3, 2agn + s)$ of type $1^{2ag(n-1)}(2ag + s)^{1}$ with long group $G_0\cup S$ and at least one of $\mathcal{A}_{i}$'s contains a subset $\mathcal{A}'$ such that $\mathcal{A}'$ is the block set of an NGDD$(2, 3, 2agn + s)$ of type $2^{(r,ag(n-1)-r)}(2ag + s)^{1}$, and

(ii) for each group $G\in \mathcal{G}\backslash\{G_{0}\}$, there are exactly $g$ $\mathcal{A}_{i}$'s such that each $\mathcal{A}_{i}$ is the block set of a GDD$_{2}(2, 3, 2agn + s)$ of type $1^{2ag(n-1)}(2ag + s)^{1}$ with long group $G\cup S$.

Let $(X, \mathcal{G}, \mathcal{A})$ be a GDD$(3, 3, 2g(n + 1))$ of type $(2g)^{n+1}$ and $G_{0}$ be one of its groups. Such a GDD is \emph{partitionable with index two}, denoted by PGDD$_{2}((2g)^{n+1})$ (as in \cite{wbt}), if the block set $\mathcal{A}$ can be partitioned into $\mathcal{A}_{i}$, $1 \leq i \leq gn + 2g$, so that

(i) for each $G \in \mathcal{G}$ with $G \neq G_{0}$, there are exactly $g$ $\mathcal{A}_{i}$'s such that each $\mathcal{A}_{i}$ is the block set of a GDD$_{2}(2, 3, 2gn)$ of type $(2g)^{n}$ with group set $\mathcal{G}\backslash \{G\}$, and

(ii) each $(X\backslash G_{0}, \mathcal{G}\backslash \{G_{0}\}, \mathcal{A}_{i})$, $gn + 1 \leq i \leq gn + 2g$, is a GDD$(2, 3, 2gn)$ of type $(2g)^{n}$.

\begin{lemma}\label{PGDD}
A PGDD$_{2}((2g)^{n+1})$ exists if and only if $n \geq 3$ and $gn(n-1)\equiv 0 \pmod 3$.
\end{lemma}

\begin{proof}
The necessary conditions for the existence of a PGDD$_{2}((2g)^{n+1})$ are $n \geq 3$ and $gn(n-1)\equiv 0 \pmod 3$. For $n \geq 3$ and $gn(n-1)\equiv 0 \pmod 3$ with $(n,g)\neq(4,1)$, there exists a PGDD$_{2}((2g)^{n+1})$, which can be produced by the so-called  generalized frame F$(3,3,(n+1)\{2g\})$ in \cite[Corollary 3.3]{PGDD1}. A PGDD$_{2}(2^{5})$ exists by \cite[Lemma 6.4]{wbt}. Then we obtain the desired result.
\end{proof}

Let $(X, S, \mathcal{G}, \mathcal{A})$ be a CS$(3,K,v)$ of type $(g^{n_{1}}_{1} g^{n_{2}}_{2} \cdots g^{n_{u}}_{u} : s)$ with  $S =\{\infty_{1},\ldots,\infty_{s}\}$. For $1 \leq i \leq s$, let $\mathcal{A}_{i} = \{A\backslash\{\infty _{i}\} : A\in \mathcal{A},~\infty_{i} \in A\}$ and $\mathcal{A}_{T} = \{A \in\mathcal{A} : A \cap S = \emptyset\}$. Then the $(s + 3)$-tuple $(X\backslash S, \mathcal{G}, \mathcal{A}_{1}, \mathcal{A}_{2},\ldots, \mathcal{A}_{s}, \mathcal{A}_{T} )$ is called an $s$-\emph{fan design} (as in \cite{fd}). If the block sizes of $\mathcal{A}_{i}$ and $\mathcal{A}_{T}$ are from $K_{i}~(1 \leq i \leq s)$ and $K_{T}$, respectively, then the $s$-fan design is denoted by $s$-FG$(3, (K_{1},\ldots,K_{s},K_{T})$, $\sum ^{u}_{i=1} n_{i}g_{i})$ of type $g^{n_{1}}_{1} g^{n_{2}}_{2} \cdots g^{n_{u}}_{u}$.

\begin{construction}\label{constrution}
Suppose that there exists an $h$-FG$(3,(K_{1},\ldots$, $K_{h}$,$ K_{T}), gn)$ of type $g^{n}$, where  $mg(n-1) \equiv 0 \pmod{3}$
and $m(k-1) \equiv 0 \pmod{3}$ for each $k \in K_{1}$.
Further suppose that there exists
\begin{enumerate}
\item[(i)] an $m$-PCS$^{r}_{2}((2am)^{k} : l)$ for each $k \in K_{1}$ and each $0\leq r \leq am(k-1)$ with $r\equiv 0 \pmod{3a}$;

\item[(ii)] a PGDD$_{2}((2am)^{k+1})$ for each $k\in K_{j}~(2\leq j\leq h)$; and

\item[(iii)] a PGDD$_{2}((2am)^{k})$ for each $k \in K_{T}$.
\end{enumerate}
Then there exists an $mg$-PCS$^{e}_{2}((2amg)^{n}: 2(h-1)am+l)$ for any $0\leq e \leq amg(n-1)$ and $e \equiv 0 \pmod{3a}$.
\end{construction}

\proof Let $(X, \mathcal{G}, \mathcal{A}_{1},\mathcal{A}_{2},\ldots,\mathcal{A}_{h}, \mathcal{T})$ be the given $h$-FG$(3,(K_{1},K_{2},\ldots,K_{h}, K_{T}), gn)$ of type $g^{n}$. Let $S = \{\infty\}\times Z_{s}$, where $s = 2(h-1)am + l$. We shall construct the desired
design on $X'= (X \times Z_{m}\times Z_{2a})\cup S$ with the group set $\mathcal{G}'=\{G'= G \times Z_{m}\times Z_{2a} : G\in\mathcal{G}\}$ and the stem $S$, where $(X\times Z_{m}\times Z_{2a})\cap S=\emptyset$. We shall describe its block set $\mathcal{F}$ below.

Denote $G_{x}=\{x\}\times Z_{m}\times Z_{2a}$ for $x\in X$ and $S = S_{1}\cup S_{2}\cup \cdots\cup S_{h}$, where $S_{1}=\{\infty\}\times Z_{l}, ~S_{j}=\{(\infty,l +2(j-2)am),\ldots,(\infty, l+2(j-1)am-1)\}$ for $2\leq j\leq h$. Fix a point $x_{0}\in X$ where $x_{0}\in G_{0}\in \mathcal{G}$. Let $\mathcal{A}_{1,0} = \{A\in \mathcal{A}_{1} : x_{0}\in A\}$. We know that $\{A\backslash\{x_{0}\} : A\in\mathcal{A}_{1,0}\}$ is a partition of $X\backslash G_{0}$.
For any $e$ with $0\leq e \leq amg(n-1)$ and $e \equiv 0 \pmod{3a}$, we may take $e_{A}$ with $0\leq e_{A} \leq am(|A|-1)$ and $e_{A} \equiv 0 \pmod{3a}$ for $A\in \mathcal{A}_{1}$ such that $e=\sum_{A\in \mathcal{A}_{1,0}}e_{A}$.

For each block $A \in \mathcal{A}_{1}$, construct a PCS$^{e_{A}}_{2}((2am)^{|A|}: l)$ on $(A\times Z_{m} \times Z_{2a}) \cup S_{1}$ having $\{G_{x}:x\in A\}$ as its group set, $S_{1}$ as its stem. Such a design exists by assumption. Denote its block set by $\mathcal{D}^{*}_{A}$ which contains a subset $\mathcal{D}_{A}$ and $\mathcal{D}_{A}$ can be partitioned into $m|A|$ pairwise disjoint block sets $\mathcal{D}_{A}(x,i)$, $(x, i) \in A \times Z_{m}$, satisfying that

(D1) each $\mathcal{D}_{A}(x, i)$ is the block set of a GDD$_{2}(2, 3$, $2am|A| + l)$ of type $1^{2am(|A|-1)}$ $(2am +l)^{1}$ and with long group $G_{x} \cup S_{1}$;

(D2) if $x_{0}\in A$, then there is a subset $\mathcal{D}'_{A}(x_{0}, 0)$ of $\mathcal{D}_{A}(x_{0}, 0)$ such that $\mathcal{D}'_{A}(x_{0}, 0)$ is the block set of an NGDD$(2, 3, 2am|A| + l)$ of type $2^{(e_{A},am(|A|-1)-e_{A})}(2am + l)^{1}$.

For any $2\leq j\leq h$ and any block $A\in \mathcal{A}_{j}$, construct a PGDD$_{2}((2am)^{|A|+1})$ on $(A\times Z_{m}\times Z_{2a})\cup S_{j}$ having $\Gamma_{A} =\{G_{x} : x\in A\}\cup\{S_{j}\}$ as its group set. Such a design exists by assumption. Denote its block set by $\mathcal{C}^{*j}_{A}$ which contains a subset $\mathcal{C}^{j}_{A}$. Two disjoint GDD$(2, 3, 2am|A|)$s with the same group set can be seen as a GDD$_{2}(2, 3, 2am|A|)$ containing a sub-GDD$(2, 3, 2am|A|)$.
So $\mathcal{C}^{j}_{A}$ can be partitioned into $m|A|$ pairwise disjoint block sets $\mathcal{C}^{j}_{A}(x, i)$, $(x, i) \in A \times Z_{m}$, satisfying that

(N1) for any $(x, i)\in A \times Z_{m}$, each $\mathcal{C}^{j}_{A}(x, i)$ is the block set of a GDD$_{2}(2, 3, 2am|A|)$ of type $(2am)^{|A|}$ with the group set $\Gamma _{A}\backslash \{G_{x}\}$;

(N2) if $x_{0} \in A$, then there is a subset $\mathcal{C}'^{j}_{A}(x_{0}, 0)$ of $\mathcal{C}^{j}_{A}(x_{0}, 0)$ so that $\mathcal{C}'^{j}_{A}(x_{0}, 0)$ is the block set of a GDD$(2, 3, 2am|A|)$ with the group set $\Gamma _{A}\backslash \{G_{x_{0}}\}$.

For each $A \in \mathcal{T}$, construct a PGDD$_{2}((2am)^{|A|})$ on $A \times Z_{m} \times Z_{2a}$ having $\Gamma_{A}' =\{G_{x} : x\in A\}$ as its group set. Such a design exists by assumption. Denote its block set by $\mathcal{B}^{*}_{A}$ which contains a subset $\mathcal{B}_{A}$;
the block set $\mathcal{B}_{A}$ can be partitioned into $m|A|$ pairwise disjoint block sets $\mathcal{B}_{A}(x, i)$, $(x, i) \in A \times Z_{m}$, satisfying that

(B1) each $\mathcal{B}_{A}(x, i)$ is the block set of a GDD$_{2}(2, 3, 2am(|A|- 1))$ with group set $\Gamma _{A}'\backslash \{G_{x}\}$;

(B2) if $x_{0} \in A$, then there is a subset $\mathcal{B}'_{A}(x_{0}, 0)$ of $\mathcal{B}_{A}(x_{0}, 0)$ so that $\mathcal{B}'_{A}(x_{0}, 0)$ is the block set of a GDD$(2, 3, 2am(|A|- 1))$ with the group set $\Gamma _{A}'\backslash \{G_{x_{0}}\}$.

For each $(x, i)\in X \times Z_{m}$, define
$$\mathcal{F}(x, i) =\big(\bigcup _{x\in A,A\in \mathcal{A}_{1}}\mathcal{D}_{A}(x, i)\big)\bigcup \big(\bigcup_{2\leq j \leq h}~\bigcup _{x\in A,A\in \mathcal{A}_{j}}\mathcal{C}^{j}_{A}(x, i)\big)\bigcup\big(\bigcup _{x\in A,A\in \mathcal{T}}\mathcal{B}_{A}(x, i)\big).$$

Define $$\mathcal{F}'(x_{0}, 0) =\big(\bigcup _{x_{0}\in A, A\in \mathcal{A}_{1}}\mathcal{D}'_{A}(x_{0}, 0)\big)\bigcup \big(\bigcup_{2\leq j \leq h}~\bigcup _{x_{0}\in A,A\in \mathcal{A}_{j}}\mathcal{C}'^{j}_{A}(x_{0}, 0)\big)\bigcup\big(\bigcup _{x_{0}\in A,A\in \mathcal{T}}\mathcal{B}'_{A}(x_{0},0)\big),$$

Finally, define $$\mathcal{F} = \big(\bigcup _{A\in \mathcal{A}_{1}}\mathcal{D}^{*}_{A}\big)\bigcup \big(\bigcup_{2\leq j \leq h}~\bigcup_{A\in \mathcal{A}_{j}}\mathcal{C}^{*j}_{A}\big)\bigcup\big(\bigcup _{A\in \mathcal{T}}\mathcal{B}^{*}_{A}\big).$$

It is readily checked that $(X', S, \mathcal{G}', \mathcal{F})$ is a $mg$-PCS$^{e}_{2}((2amg)^{n} : 2(h-1)am+l)$; each $\mathcal{F}(x, i)$ is the block set of a GDD$_{2}(2, 3, 2amgn + s)$ of type $1^{2amg(n-1)}(2amg + s)^{1}$ with long group $(G \times Z_{m} \times Z_{2a}) \cup S$; $\mathcal{F}'(x_{0}, 0)\subset \mathcal{F}(x_{0}, 0)$ is the block set of an NGDD$(2, 3, 2amgn + s)$ of type $2^{(e,amg(n-1)-e)}(2amg + s)^{1}$.
\qed

We also employ a variant of Construction \ref{constrution}.

\begin{construction}\label{constrution1}
Suppose that there exists a $1$-FG$(3,(K_{1}, K_{T}), gn)$ of type $g^{n}$
where $mg(n-1) \equiv 0 \pmod{3}$ and $m(k-1) \equiv 0 \pmod{3}$ for each $k \in K_{1}$.
Further suppose that there exists an $m^{+}$-PCS$^{r}_{2}((4m)^{k} : l)$ for each $k \in K_{1}$ and each $0\leq r \leq 2m(k-1)$ with $r\equiv 0 \pmod{6}$, and
a PGDD$_{2}((4m)^{k})$ for each $k \in K_{T}$,
then there exists an $(mg)^{+}$-PCS$^{e}_{2}((4mg)^{n}:l)$ for any $0\leq e \leq 2mg(n-1)$ and $e \equiv 0 \pmod{6}$.
\end{construction}

\begin{theorem}\label{s46}
A $1$-FG$(3, (\{3, 5\},\{4, 6\}), n)$ of type $1^{n}$ exists for any $n \equiv 1 \pmod{2}$.
\end{theorem}
\proof An S$(3,\{4, 6\}, n+1)$ exists if and only if $n \equiv 1$ (mod 2) \cite{46}. Deleting a point yields the desired 1-fan design.\qed

\begin{lemma}\label{tee11}
Let $m\equiv n\equiv0\pmod{3}$,  $u\equiv2,4\pmod{6}$, and $v=2m+2n+u$. Further let $c<{u \choose 3}$ and $(u,c)$ satisfy (C2) associating with $(\lambda,\delta)$.  Suppose that
 \begin{enumerate}
 \item[(i)] $(V,{\cal A})$ is  a simple NGDD$(2,3,v)$ of type $2^{(m,n)}u^{1}$  with $G_i,1\le i\le m+n,$ being  groups of size $2$  and $U$ the group of size $u$;
  \item[(ii)]   $(V,{\cal B})$ is
a simple GDD$_{\lambda-1}(2,3,v)$ of type $1^{2m+2n}u^1$ with long group $U$; and
 \item[(iii)] $(U,\mathcal{C})$ is an NWBTS$(u; c)$.
   \end{enumerate}
   If the block sets ${\cal A},{\cal B}$ and ${\cal C}$ are mutually disjoint, then  ${\cal F}:={\cal A}\cup{\cal B}\cup{\cal C}$ forms the block set of an NWBTS$(v; |{\cal F}|)$, where $3|{\cal F}|=\lambda{v\choose 2}+m-n+\delta$.
\end{lemma}

\proof %Denote $|{\cal A}|=a$ and $|{\cal B}|=b$.
It is obvious that $3|{\cal A}|={v\choose 2}-{u\choose 2}+m-n$ and $3|{\cal B}|=\big(\lambda-1)({v\choose 2}-{u\choose 2}\big)$. And using  $3c=\lambda {u\choose 2}+\delta$ yields
$3|{\cal F}|=3(|{\cal A}|+|{\cal B}|+c)=\lambda{v\choose 2}+\varepsilon$ where $\varepsilon=m-n+\delta$. Clearly $-\frac{v}{2}<\varepsilon<\frac{v}{2}$ %and ${v\over 2}-\varepsilon\equiv {u\over 2}-\delta$ (mod 2)
  as $v=2m+2n+u$ and $-\frac{u}{2}<\delta<\frac{u}{2}$.  Let  $(U,\mathcal{C})$ have a defect graph $D=D_1\cup D_{-1}$, isomorphic to $H_{u,\delta}^0$ or one of $H_{u,\delta}^1$ and $H_{u,\delta}^2$, corresponding to the case ${u\over 2}-\delta$ being even or odd.  By the construction we have that every pair of $V$ is contained in $\lambda$ blocks of ${\cal F}$ except that each pair in $D_1^{\cal F}:=P_1\cup D_1$ is contained in $\lambda+1$ blocks and each  pair in $D_{-1}^{\cal F}:=P_{-1}\cup D_{-1}$ is contained in $\lambda-1$ blocks, where $P_1=\{G_i:1\le i\le m\}$ and  $P_{-1}=\{G_i:m+1\le i\le m+n\}$ represent  the groups  contained in two blocks and no blocks of ${\cal A}$, respectively. Thus we determined the defect graph $D^{\cal F}=D_1^{\cal F}\cup D_{-1}^{\cal F}$ of ${\cal F}$.
In the case of ${u\over 2}\equiv\delta$ (mod 2), we have ${v\over 2}\equiv\varepsilon$ (mod 2), $({v+2\varepsilon\over 4},{v-2\varepsilon\over 4})=(m+|D_1|,n+|D_{-1}|)$,  $D^{\cal F}$ is isomorphic to $H_{v,\varepsilon}^0$, and hence ${\cal F}$ is nearly 2-balanced.  Similar arguments yield that $D^{\cal F}$ is isomorphic to $H_{v,\varepsilon}^i$ $(i=1,2)$ if $D$ is isomorphic to $H_{u,\delta}^i$ in the case ${u\over 2}\not\equiv\delta$ (mod 2). Finally noticing ${\cal F}$ is a simple family yields that it is nearly well-balanced. \qed

\begin{construction}\label{tbbb}
Let $g\equiv 0 \pmod{3}$, $s\equiv2,4\pmod 6$, and $\lambda\in\{1,3,5\}$. Suppose that there exists a $g$-PCS$^{e}_{2}((2ag)^{n}:s)$ for any $0\leq e \leq ag(n-1)$ and $e\equiv 0 \pmod{3a}$. Further let $u=2ag+s$ and suppose that there exists
\begin{enumerate}
\item[(i)] a simple S$_{6}(2,3,u)$ defined over a $u$-set $X$, each of whose blocks is not strictly contained in an $s$-subset $Y\subset X$; and

\item[(ii)] an NWBTS$(u;c)$  for any $c=(\lambda{u\choose 2}+\delta)/3$ where  $-{u\over 2} <\delta<{u\over 2}$. % $(2ag+s,c)$ associates with $(\lambda,\delta)$ and $\lambda\in\{1,3,5\}$.
\end{enumerate}
 Let $v=2agn+s$.  Then there exists an NWBTS$(v;b)$ for any $b=qv(v-1)+b'$, where %$0\leq q\leq n-1$ and $3b'=\lambda\binom{v}{2}-ag(n-1)+2e+\delta$. In particular, in either case of $(g,s)=(3,2)$ or $s\equiv  4\pmod 6$,  there exists an NWBTS$(v;qv(v-1)+b')$ for any
$0\leq q\leq n-1$ and $3b'=\lambda\binom{v}{2}+\varepsilon$ with  $-{v\over 2} <\varepsilon<{v\over 2}$.
\end{construction}

\proof Let $V = (Z_{n} \times Z_{g}\times Z_{2a})\cup S$, $|S|=s$,  and $\mathcal{G} = \{G_{0}, G_{1},\ldots,G_{n-1}\}$ where $G_{i} = \{i\}\times Z_{g}\times Z_{2a}$. For any fixed $e$ with $0\leq e \leq ag(n-1)$ and $e\equiv 0 \pmod{3a}$, let $(V, S, \mathcal{G}, \mathcal{A})$ be an assumed $g$-PCS$_{2}^{e}((2ag)^{n}:s)$. Then the block set $\mathcal{A}$ contains a subset $\mathcal{A}'$ and $\mathcal{A}'$ can be partitioned into $gn$ pairwise disjoint parts $\mathcal{A}(i, j)$, $0 \leq i \leq n- 1$, $0 \leq j \leq g-1$, so that each $\mathcal{A}(i, j)$ is the block set of a GDD$_{2}(2, 3, v)$ of type $1^{2ag(n-1)}u^{1}$ with long group $G_{i}\cup S$, and $\mathcal{A}(0, 0)$ contains a subset $\mathcal{A}'(0, 0)$ that is the block set of an NGDD$(2,3,v)$ of type $2^{(e,ag(n-1)-e)}u^{1}$. By assumption, there exists an NWBTS$(u; c)$ on $G_{0}\cup S$ with block set $\mathcal{C}_{\lambda}$ where $3c=\lambda\binom{u}{2}+\delta$.
Construct a family $\mathcal{F}$ containing all blocks of $\mathcal{A}'(0,0)$, $(\lambda-1)/2$ sets in $\{\mathcal{A}(0,j):j=1,2\}$, and all blocks of $\mathcal{C}_{\lambda}$. Then by Lemma \ref{tee11}, $\mathcal{F}$ is an NWBTS$(v;b')$, where $3b'=\lambda\binom{v}{2}-ag(n-1)+2e+\delta$.

For each $1\leq i \leq n- 1$, construct a simple S$_{6}(2,3,u)$ on $G_i\cup S$ with block set $\mathcal{B}(i)\subset\binom{G_{i}\cup S}{3}\backslash \binom{S}{3}$ which exists by assumption.
Let $\mathcal{F}(i) = (\cup^{2}_{j=0}\mathcal{A}(i, j )) \cup \mathcal{B}(i)$. Then each $\mathcal{F}(i)$ is the block set of a simple S$_{6}(2,3,v)$ and all $\mathcal{F}(i)$s are pairwise disjoint. Construct a family $\mathcal{F}'$ containing $q$ $\mathcal{F}(i)$s where $0\leq q\leq n-1$. Thus $\mathcal{F} \cup \mathcal{F}'$ is an NWBTS$(v; qv(v-1)+b')$ by Lemma \ref{tee}, where $3b'=\lambda\binom{v}{2}+\varepsilon$ with
$\varepsilon=-{v-u\over 2}+2e+\delta$.

Now let $s\equiv  2\pmod 6$. We know that $(u,c)$ associates with $(\lambda,\delta)$. %$3b'=\lambda\binom{v}{2}+\varepsilon$ and any $-{v\over 2} <\varepsilon<{v\over 2}$.
When $e,\delta$ take all admissible integers, it is readily  checked that the value of $\varepsilon$ covers all admissible integers in the intervals $I_e:=[-{v\over 2}+2e+\tau_1,-{v\over 2}+u+2e+\tau_2]$ with $0\leq e \leq ag(n-1)$ and $e\equiv 0 \pmod{3a}$, where
 $$(\tau_1,\tau_2)=\left\{\begin{array}{ll}
                          { (3,-2)}, &                  {\rm if\ } \lambda=1,\\
                          { (1,-1)}, &                   {\rm if\ } \lambda=3,\\
                          { (2,-3)}, &                   {\rm if\ } \lambda=5.\\
 \end{array}\right.$$
It is not difficult to compare  the right endpoint $y_j$ of $I_j$ with the left endpoint $x_{j+1}$ of $I_{j+1}$ to have  that $y_j+3\ge x_{j+1}$. Thus these intervals cover all admissible integers $\varepsilon$ with $-{v\over 2}<\varepsilon<{v\over 2}$. As a consequence, an NWBTS$(v;qv(v-1)+b')$ exists for any $0\leq q\leq n-1$ and $3b'=\lambda\binom{v}{2}+\varepsilon$ with  $-{v\over 2} <\varepsilon<{v\over 2}$. Similar arguments also apply for the case $s\equiv  4\pmod 6$; in this case we always have $(\tau_1,\tau_2)=(2,-3)$. The details are omitted.
\qed

\subsection{$v \equiv 2\pmod  6$}
In this subsection, we construct $3$-PCS$^{e}_{2}((6a)^{n} : 2)$s for $a=1,2$ and $0\leq e\leq 3a(n-1)$ with $e\equiv 0\pmod{3a}$, and then show the existence of an NWBTS$(v;b)$ if $v \equiv 2 \pmod{6}$, $v\neq50,74$, and $(v,  b)$ satisfies (C2).

\begin{lemma}\label{3632}
There is a $3$-PCS$^{r}_{2}(6^{3}:2)$ for any $r=0, 3, 6$.
\end{lemma}
\proof For $r=0, 6$, a $3$-PCS$^{r}_{2}(6^{3}:2)$ exists by \cite[Lemma 6.6]{wbt}. Next, we construct a $3$-PCS$^{3}_{2}(6^{3}:2)$ on $X = Z_{18} \cup \{a, b\}$ with stem $S = \{a, b\}$ and groups $G_{i} = \{i, i + 3, i + 6,\ldots,i + 15\}, i = 0, 1, 2$. We first construct a GDD$_{2}(2, 3, 20)$ of
type $1^{12}8^{1}$ with long group $G_{0} \cup S$, whose blocks are generated from the following
blocks by adding 9 (mod 18); the underlined base blocks and their developments form the
blocks of an NGDD$(2,3,20)$ of type $2^{(3,3)}8^{1}$, where $\lambda_{1,10}=\lambda_{4,13}=\lambda_{7,16}=2$ and $\lambda_{2,11}=\lambda_{5,14}=\lambda_{8,17}=0$.

{\small\begin{tabular}{llllllll}
$\underline{\{a, 1, 2\}}$ & $\underline{\{a, 4, 8\}}$ & $\underline{\{a, 7, 14\}}$ & $\underline{\{b, 4, 5\}}$ & $\underline{\{b, 7, 11\}}$ & $\underline{\{b, 1, 8\}}$ & $\underline{\{2, 6, 17\}}$ & $\underline{\{5, 6, 11\}}$\\
$\underline{\{1, 3, 10\}}$ & $\underline{\{3, 4, 13\}}$ & $\underline{\{3, 7, 16\}}$ & $\underline{\{0, 1, 7\}}$ & $\underline{\{6, 7, 8\}}$ & $\underline{\{5, 7, 10\}}$ & $\underline{\{8, 13, 16\}}$ & $\underline{\{3, 11, 17\}}$\\
$\underline{\{2, 10, 13\}}$ & $\underline{\{1, 6, 13\}}$ & $\underline{\{6, 10, 14\}}$ & $\underline{\{0, 8, 10\}}$ & $\underline{\{0, 2, 4\}}$ & $\underline{\{4, 6, 16\}}$ & $\underline{\{2, 3, 5\}}$ & $\underline{\{3, 8, 14\}}$\\
$\underline{\{4, 9, 14\}}$ & $\underline{\{2, 7, 9\}}$ & $\underline{\{0, 14, 17\}}$ & $\{2, 3, 11\}$ & $\{3, 5, 14\}$ & $\{3, 8, 17\}$ & $\{b, 2, 4\}$ & $\{b, 7, 17\}$\\
$\{3, 10, 13\}$ & $\{1, 3, 7\}$ & $\{3, 4, 16\}$ & $\{0, 2, 14\}$ & $\{0, 7, 8\}$ & $\{1, 6, 16\}$ & $\{a, 4, 17\}$ & $\{b, 1, 14\}$\\
$\{0, 1, 5\}$ & $\{0, 13, 16\}$ & $\{0, 11, 17\}$ & $\{0, 4, 10\}$ & $\{6, 10, 17\}$ & $\{4, 6, 11\}$ & $\{a, 1, 11\}$ & $\{a, 5, 7\}$\\
$\{1, 2, 17\}$ & $\{7, 11, 14\}$ & $\{2, 6, 7\}$ & $\{4, 5, 8\}$ & $\{5, 6, 13\}$ & $\{6, 8, 14\}$\\
\end{tabular}}

\noindent Nine GDD$_{2}(2, 3, 20)$s are obtained by developing modulo 18, forming a $3$-PCS$^{3}_{2}(6^{3}:2)$.\qed

\begin{lemma}\label{6n21}
Let $n \equiv 1 \pmod{2}$ and $n\geq 3$. Then there is a $3$-PCS$_{2}^{e}(6^{n}:2)$ for any $0\leq e \leq 3(n-1)$ and $e\equiv0\pmod{3}$.
\end{lemma}
\proof For $n=3$, see Lemma \ref{3632}. For $n=5$, there exists a 1-FG$(3, (4,\{4, 5\}), 15)$ of type $3^{5}$, which is obtained by deleting two points from the known S$(3, 5, 17)$ (see \cite{652}).  Apply Construction \ref{constrution} with $h= a=m= 1$ and $l = 2$. The input designs are $1$-PCS$^{i}_{2}(2^{4}: 2)$s $(i=0, 3)$ in \cite[Lemma 6.1]{wbt} and PGDD$_{2}(2^{j})$s $(j=4,5)$ in Lemma \ref{PGDD}. Then we get a $3$-PCS$^{r}_{2}(6^{5}: 2)$ for any $r\in\{0,3,6,9,12\}$.

For $n\equiv 1$ (mod 2) and $n\geq7$, start with a 1-FG$(3, (\{3, 5\},\{4, 6\}), n)$ of type $1^{n}$ which exists by Lemma \ref{s46}. A $3$-PCS$^{r}_{2}(6^{k}: 2)$ exists by  the above proofs for $k\in\{3,5\}$ and $r\le 3(k-1)$ with $r\equiv 0\pmod{3}$.   A PGDD$_{2}(6^{k})$ for $k\in\{4,6\}$ exists by Lemma \ref{PGDD}. Then we obtain the desired design by applying Construction \ref{constrution} with $h=a = 1$, $l = 2$, and $m = 3$.\qed

\begin{lemma}\label{v812}
Let $v \equiv 8 \pmod{12}$ and $(v,  b)$ satisfy (C2). There exists an NWBTS$(v;b)$.
\end{lemma}

\proof For $v = 8$, see Lemma \ref{81} in Appendix. For $v \geq 20$, we can assume that $b\leq \frac{v(v-1)(v-2)}{6}$ by Lemma \ref{bb}. Write $v = 6n +2$ where $n$ is odd; write $b = qv(v-1) + b'$ where $0\leq q\leq \frac{v-2}{6}-1 = n- 1$ and $0\leq b' \leq v(v-1)$. According to Lemma \ref{bb}, we can assume that $0 \leq b' \leq \frac{v(v-1)}{2}$ and $\frac{\lambda v(v-1)}{6}-\frac{v}{6} < b' < \frac{\lambda v(v-1)}{6}+\frac{v}{6}$ with $\lambda=1,3$. Thus $(v,b')$ associates with $(\lambda,\varepsilon)$ where $\lambda=1,3$ and $-\frac{v}{2}< \varepsilon < \frac{v}{2}$.  Apply Construction \ref{tbbb} with $a=1$ and $(g,s)=(3,2)$.
There exists a $3$-PCS$_{2}^{e}(6^{n}:2)$ for any $0\leq e \leq 3(n-1)$ and $e\equiv0\pmod{3}$ from Lemma \ref{6n21}. An NWBTS$(8; c)$ exists for all $c$ with $(8,c)$ satisfying (C2)
 from Lemma \ref{81} in Appendix.  A simple S$_{6}(2,3,8)$ exists trivially. Then there exists an NWBTS$(v;b)$. %, where $0\leq q\leq n-1$, $3b'=\lambda\binom{v}{2}-3(n-1)+2e+\delta$ and $(8,c)$ associates $\lambda$ and $\delta$.
\qed

\begin{lemma}\label{H1232}\label{H1232}
There is a $3$-PCS$^{r}_{2}(12^{3} : 2)$ for $r = 0, 6, 12$.
\end{lemma}

\proof For $r = 0, 12$, a $3$-PCS$^{r}_{2}(12^{3} : 2)$ exists by \cite[Lemma 6.12]{wbt}. For $r = 6$,
we construct a $3$-PCS$^{6}_{2}(12^{3} : 2)$ on $X = (Z_{12} \times Z_{3}) \cup\{a, b\}$ with stem $S = \{a, b\}$ and groups $G_{j} = \{(0, j), (1, j), \ldots, (11, j)\},~j = 0, 1, 2$. We first construct an initial GDD$_{2}(2, 3, 38)$ of type $1^{24}14^{1}$ with long group $G_0\cup S$. The blocks are generated from the base blocks in the following by adding $3$ (mod 12) to the first coordinate; the underlined base blocks and their developments form the blocks of an NGDD$(2, 3, 38)$ of type $2^{(6,6)}14^{1}$, where $\lambda_{(0,1),(6,1)}=\lambda_{(1,1),(7,1)}=\lambda_{(2,1),(8,1)}=\lambda_{(3,1),(9,1)}=\lambda_{(4,1),(10,1)}=\lambda_{(5,1),(11,1)}=2$ and $\lambda_{(0,2),(6,2)}=\lambda_{(1,2),(7,2)}=\lambda_{(2,2),(8,2)}=\lambda_{(3,2),(9,2)}=\lambda_{(4,2),(10,2)}=\lambda_{(5,2),(11,2)}=0$.

{\small\begin{longtable}{llllll}
$\underline{\{a, (0, 1), (1, 2)\}}$ & $\underline{\{a, (1, 1), (3, 2)\}}$ & $\underline{\{a, (2, 1), (2, 2)\}}$ & $\underline{\{b, (0, 1), (2, 2)\}}$\\
$\underline{\{b, (1, 1), (1, 2)\}}$ & $\underline{\{b, (2, 1), (9, 2)\}}$ & $\underline{\{(0, 1), (6, 1), (0, 2)\}}$ & $\underline{\{(5, 0), (1, 1), (7, 1)\}}$\\
$\underline{\{(5, 0), (2, 1), (8, 1)\}}$ & $\underline{\{(0, 0), (0, 1), (2, 1)\}}$ & $\underline{\{(1, 0), (1, 1), (4, 1)\}}$ & $\underline{\{(10, 0), (0, 1), (5, 1)\}}$\\
$\underline{\{(2, 0), (1, 1), (2, 1)\}}$ & $\underline{\{(2, 0), (9, 1), (1, 2)\}}$ & $\underline{\{(11, 0), (0, 1), (3, 1)\}}$ & $\underline{\{(0, 0), (1, 1), (0, 2)\}}$\\
$\underline{\{(4, 1), (0, 2), (5, 2)\}}$ & $\underline{\{(1, 0), (2, 1), (7, 1)\}}$ & $\underline{\{(4, 0), (0, 1), (1, 1)\}}$ & $\underline{\{(6, 0), (1, 1), (3, 1)\}}$\\
$\underline{\{(3, 0), (2, 1), (6, 1)\}}$ & $\underline{\{(0, 0), (6, 1), (5, 2)\}}$ & $\underline{\{(6, 0), (1, 2), (2, 2)\}}$ & $\underline{\{(3, 0], (0, 2), (2, 2)\}}$\\
$\underline{\{(5, 0), (2, 2), (7, 2)\}}$ & $\underline{\{(1, 0), (1, 2), (3, 2)\}}$ & $\underline{\{(2, 0), (7, 1), (2, 2)\}}$ & $\underline{\{(1, 1), (6, 1), (11, 2)\}}$\\
$\underline{\{(3, 0), (7, 1), (1, 2)\}}$ & $\underline{\{(3, 0), (1, 1), (5, 2)\}}$ & $\underline{\{(6, 0), (2, 1), (0, 2)\}}$ & $\underline{\{(0, 0), (5, 1), (1, 2)\}}$\\
$\underline{\{(9, 0), (0, 2), (1, 2)\}}$ & $\underline{\{(8, 0), (1, 2), (4, 2)\}}$ & $\underline{\{(2, 0), (8, 1), (0, 2)\}}$ & $\underline{\{(11, 0), (0, 2), (3, 2)\}}$\\
$\underline{\{(1, 1), (5, 1), (6, 2)\}}$ & $\underline{\{(2, 1), (4, 1), (1, 2)\}}$ & $\underline{\{(0, 1), (4, 1), (7, 2)\}}$ & $\underline{\{(10, 0), (2, 2), (3, 2)\}}$\\
$\underline{\{(7, 0), (1, 2), (5, 2)\}}$ & $\underline{\{(1, 0), (0, 2), (4, 2)\}}$ & $\underline{\{(4, 0), (3, 1), (0, 2)\}}$ & $\underline{\{(7, 0), (2, 2), (4, 2)\}}$\\
$\underline{\{(10, 0), (2, 1), (3, 1)\}}$ & $\underline{\{(4, 0), (2, 1), (5, 2)\}}$ & $\underline{\{(8, 0), (6, 1), (2, 2)\}}$ & $\underline{\{(8, 1), (2, 2), (5, 2)\}}$\\
$\underline{\{(11, 0), (2, 2), (6, 2)\}}$ & $\underline{\{(3, 1), (1, 2), (6, 2)\}}$ & $\underline{\{(2, 1), (5, 1), (7, 2)\}}$ & $\{a, (0, 1), (4, 2)\}$\\
$\{a, (2, 1), (9, 2)\}$ &$\{a, (1, 1), (11, 2)\}$ &$\{(2, 1), (0, 2), (6, 2)\}$ & $\{b, (0, 1), (10, 2)\}$\\
$\{b, (1, 1), (9, 2)\}$ & $\{b, (2, 1), (5, 2)\}$ & $\{(3, 0), (1, 2), (7, 2)\}$ & $\{(5, 0), (2, 2), (8, 2)\}$\\
$\{(4, 0), (0, 2), (1, 2)\}$ & $\{(2, 1), (7, 1), (10, 2)\}$ & $\{(9, 0), (0, 2), (5, 2)\}$ & $\{(0, 0), (1, 2), (2, 2)\}$\\
$\{(1, 1), (3, 1), (5, 2)\}$ & $\{(5, 0), (1, 1), (1, 2)\}$ & $\{(8, 0), (1, 2), (3, 2)\}$ & $\{(7, 0), (3, 1), (0, 2)\}$\\
$\{(9, 0), (0, 1), (5, 1)\}$ & $\{(4, 0), (2, 2), (5, 2)\}$ & $\{(6, 1), (2, 2), (6, 2)\}$ & $\{(3, 0), (1, 1), (8, 2)\}$\\
$\{(1, 0), (11, 1), (1, 2)\}$ & $\{(6, 0), (0, 2), (3, 2)\}$ & $\{(2, 0), (4, 1), (3, 2)\}$ & $\{(8, 0), (0, 2), (2, 2)\}$\\
$\{(1, 0), (0, 1), (3, 2)\}$ & $\{(5, 0), (2, 1), (3, 2)\}$ & $\{(2, 0), (1, 2), (4, 2)\}$ & $\{(2, 0), (2, 1), (6, 1)\}$\\
$\{(9, 0), (1, 1), (6, 1)\}$ & $\{(4, 0), (1, 1), (4, 1)\}$ & $\{(0, 0), (0, 1), (7, 2)\}$ & $\{(7, 0), (1, 2), (6, 2)\}$\\
$\{(3, 0), (2, 2), (3, 2)\}$ & $\{(0, 1), (1, 1), (6, 2)\}$ & $\{(9, 0), (2, 1), (3, 1)\}$ & $\{(6, 0), (1, 1), (5, 1)\}$\\
$\{(0, 0), (1, 1), (2, 1)\}$ & $\{(1, 0), (2, 1), (4, 1)\}$ & $\{(2, 0), (1, 1), (2, 2)\}$ & $\{(1, 0), (3, 1), (8, 2)\}$\\
$\{(10, 0), (2, 1), (2, 2)\}$ & $\{(3, 1), (2, 2), (4, 2)\}$ & $\{(7, 0), (0, 1), (2, 1)\}$ & $\{(1, 0), (7, 1), (4, 2)\}$\\
$\{(11, 0), (2, 1), (5, 1)\}$ & $\{(8, 1), (2, 2), (7, 2)\}$ & $\{(11, 0), (0, 1), (4, 1)\}$ & $\{(5, 0), (0, 1), (3, 1)\}$\\
$\{(10, 1), (0, 2), (4, 2)\}$ & $\{(8, 1), (1, 2), (5, 2)\}$\\
\end{longtable}}

\noindent Nine pairwise disjoint GDD$_{2}(2, 3, 38)$s result by developing over both coordinates.
The remaining blocks are forced by the definition of a $3$-PCS$_{2}^{6}(12^{3} : 2)$.\qed

\begin{lemma}\label{H12n2}
Let $n \equiv 1 \pmod{2}$ and $n \geq 3$. Then there exists a $3$-PCS$^{e}_{2}(12^{n} : 2)$ for any $0 \leq e \leq 6(n-1)$ and $e\equiv0\pmod{6}$.
\end{lemma}
\proof For $n=3$, see Lemma \ref{H1232}. For $n=5$, there exists a 2-FG$(3, (4, 4, 4), 30)$ of type $6^{5}$ from \cite[Lemma 3.5]{45}. It forms a 1-FG$(3,(4,4),30)$ of the same type. Apply Construction \ref{constrution} with $h= a=m= 1$ and $l=2$; the input designs are $1$-PCS$^{i}_{2}(2^{4} : 2)$s $(i=0, 3)$ in \cite[Lemma 6.1]{wbt} and a PGDD$_{2}(2^{4})$ in Lemma \ref{PGDD}. Then we obtain a $6$-PCS$^{r}_{2}(12^{5}:2)$, also a $3$-PCS$^{r}_{2}(12^{5}:2)$, where $r \in\{ 0, 6, 12, 18, 24\}$.

For $n \equiv 1 \pmod{2}$ and $n\geq7$, start with a 1-FG$(3, (\{3, 5\},\{4, 6\}), n)$ of type $1^{n}$ which exists by Lemma \ref{s46} and use the above $3$-PCS$^{r}_{2}(12^{k} : 2)$s for $k\in\{3,5\}$ and $r\le 6(k-1)$ with $r\equiv 0\pmod{6}$. A PGDD$_{2}(12^{k})$ for $k\in\{4,6\}$ exists by Lemma \ref{PGDD}. Apply Construction \ref{constrution} with $h=1$, $a=l=2$, and $m = 3$ to obtain the required $3$-PCS$^{e}_{2}(12^{n} : 2)$ for any $0 \leq e \leq 6(n-1)$ and $e\equiv0\pmod{6}$.\qed

\begin{lemma}\label{v1424}
Let $v \equiv 14 \pmod{24}$ and $(v,  b)$ satisfy (C2). There exists an NWBTS$(v;b)$.
\end{lemma}

\proof For $v = 14$, see Lemma \ref{14} in Appendix. For $v \geq 38$, we only need to consider $b\leq \frac{v(v-1)(v-2)}{12}$ by Lemma \ref{bb}. Let $v = 12n +2$ where $n$ is odd.  Write $b = qv(v-1) + b'$ where $0\leq q\leq \frac{v-2}{12}-1 = n- 1$ and $\frac{\lambda v(v-1)}{6}-\frac{v}{6} < b' < \frac{\lambda v(v-1)}{6}+\frac{v}{6}$ with $\lambda=1,3,5$. Thus $(v,b')$ associates with $(\lambda,\varepsilon)$ where $\lambda=1,3,5$ and $-\frac{v}{2}< \varepsilon < \frac{v}{2}$.
There exists a $3$-PCS$^{e}_{2}(12^{n}: 2)$ where $0 \leq e \leq 6(n-1)$ and $e\equiv0\pmod{6}$ from Lemma \ref{H12n2}. An NWBTS$(14; c)$ exists for all $c$ with $(14,c)$ satisfying (C2) by Lemma \ref{14} in Appendix. A simple S$_{6}(2,3,14)$ exists by Theorem \ref{NF}. Then there exists an NWBTS$(v; b)$  by Construction \ref{tbbb}.\qed

It remains to treat the case when $v\equiv2\pmod{24}$. We define a generalization of $g$-PCS$^{r}_{2}((2g)^{n}:s)$.

Let $(X, S, \mathcal{G}, \mathcal{A})$ be a CS$((6g)^{n}:s)$  with $g$ being a positive integer. It is denoted by $g$-PCS$^{r}_{6}((6g)^{n} : s)$ if its block set $\mathcal{A}$ contains a subset $\mathcal{A}'$ and $\mathcal{A}'$ can be partitioned into  $gn+2$ pairwise disjoint parts $\mathcal{A}_{i}$, $1 \leq i \leq gn+2$, when it satisfies that

(i) for a fixed group $G_{0}\in \mathcal{G}$, there are exactly $g+2$ $\mathcal{A}_{i}$'s satisfying that (a) each $\mathcal{A}_{i}~(1\leq i \leq 3)$ is the block set of a GDD$_{2}(2, 3, 6gn + s)$ of type $1^{6g(n-1)}(6g + s)^{1}$ with long group $G_0\cup S$ and at least one of $\mathcal{A}_{i}$'s contains a subset $\mathcal{A}'$ such that $\mathcal{A}'$ is the block set of an NGDD$(2, 3, 6gn + s)$ of type $2^{(r,3g(n-1)-r)}(6g + s)^{1}$, and (b) each $\mathcal{A}_{i}~(4\leq i \leq g+2)$ is the block set of a GDD$_{6}(2, 3, 6gn + s)$ of type $1^{6g(n-1)}(6g + s)^{1}$ with long group $G_0\cup S$,  and

(ii) for each group $G\in \mathcal{G}\backslash\{G_{0}\}$, there are exactly $g$ $\mathcal{A}_{i}$'s such that each $\mathcal{A}_{i}$ is the block set of a GDD$_{6}(2, 3, 6gn + s)$ of type $1^{6g(n-1)}(6g + s)^{1}$ with long group $G\cup S$.

Clearly, a $3g$-PCS$^{r}_{2}((6g)^{n} : s)$ gives a $g$-PCS$^{r}_{6}((6g)^{n} : s)$. We will make use of the following two immediate
variants of Constructions \ref{constrution} and \ref{tbbb}.

\begin{construction}\label{constrution14}
Suppose that there exists a $1$-FG$(3,(K_{1}, K_{T}), gn)$ of type $g^{n}$.
If there exists an $m$-PCS$^{r}_{6}((6m)^{k} : l)$ for each $k \in K_{1}$ and each $0\leq r \leq 3m(k-1)$ with $r\equiv 0 \pmod{3}$, and
a PGDD$_{2}((6m)^{k})$ for each $k \in K_{T}$,
then there exists an $mg$-PCS$^{e}_{6}((6mg)^{n}:l)$ for any $0\leq e \leq 3mg(n-1)$ and $e \equiv 0 \pmod{3}$.
\end{construction}

\begin{construction}\label{tbbb1}
Let $s\equiv2,4\pmod 6$ and $\lambda\in\{1,3,5\}$. Suppose that there exists a $g$-PCS$^{e}_{6}((6g)^{n}:s)$ for any $0\leq e \leq 3g(n-1)$ and $e\equiv 0 \pmod{3}$. Further let $u=6g+s$ and suppose that there exists
\begin{enumerate}
\item[(i)] a simple S$_{6}(2,3,u)$ defined over a $u$-set $X$, each of whose blocks is not strictly contained in an $s$-subset $Y\subset X$; and

\item[(ii)] an NWBTS$(u;c)$  for any $c=(\lambda{u\choose 2}+\delta)/3$ where  $-{u\over 2} <\delta<{u\over 2}$. % $(2ag+s,c)$ associates with $(\lambda,\delta)$ and $\lambda\in\{1,3,5\}$.
\end{enumerate}
 Let $v=6gn+s$.  Then there exists an NWBTS$(v;b)$ for any $b=qv(v-1)+b'$, where %$0\leq q\leq n-1$ and $3b'=\lambda\binom{v}{2}-ag(n-1)+2e+\delta$. In particular, in either case of $(g,s)=(3,2)$ or $s\equiv  4\pmod 6$,  there exists an NWBTS$(v;qv(v-1)+b')$ for any
$0\leq q\leq n-1$ and $3b'=\lambda\binom{v}{2}+\varepsilon$ with  $-{v\over 2} <\varepsilon<{v\over 2}$.
\end{construction}

\begin{lemma}\label{wpcs}
Let $n \equiv 0 \pmod{4}$ and $n\neq8,12$. Then there exists a $1$-PCS$^{e}_{6}(6^{n} : 2)$ for any $0\leq e \leq 3(n-1)$ and $e \equiv 0 \pmod{3}$.
\end{lemma}
\proof
For $n \equiv 0 \pmod{4}$ and $n\neq8,12$, there exists a 1-FG$(3, (\{3, 4, 5\},\{4, 5, 6\}), n)$ of type $1^{n}$ which can be obtained by deleting a point from an S$(3,\{4,5,6\},n+1)$ in \cite{456}. A $3$-PCS$^{r}_{2}(6^{4}:2)$ for any $r\in\{0,9\}$ exists by Lemma \ref{6420} in Appendix. There exists a $1$-PCS$^{r}_{6}(6^{4}:2)$ for any $r\in\{3,6\}$ by Lemma \ref{wp} in Appendix.
Since a $3$-PCS$^{r}_{2}(6^{k}:2)$ exists for $k\in\{3,5\}$ and $r\leq3(k-1)$ with $r\equiv0\pmod{3}$ by Lemma \ref{6n21}, and a PGDD$_{2}(6^{k})~(k\in\{4,5,6\})$ also exists by Lemma \ref{PGDD}, we can obtain a $1$-PCS$_{6}^{e}(6^{n}:2)$ for any $0\leq e \leq 3(n-1)$ and $e\equiv0\pmod{3}$ by using Construction \ref{constrution14}.
\qed

\begin{lemma}\label{v224}
Let $v \equiv 2 \pmod{24}$, $v\neq50,74$, and $(v, b)$ satisfy (C2). There exists an NWBTS$(v;b)$.
\end{lemma}

\proof
We can assume that $b\leq \frac{v(v-1)(v-2)}{6}$ by Lemma \ref{bb}. Write $v = 6n +2$ where $n \equiv 0 \pmod{4}$ and $n\neq8,12$; write $b = qv(v-1) + b'$ where $0\leq q\leq \frac{v-2}{6}-1 = n- 1$ and $0\leq b' \leq v(v-1)$. According to Lemma \ref{bb}, we can assume that $0 \leq b' \leq \frac{v(v-1)}{2}$ and $\frac{\lambda v(v-1)}{6}-\frac{v}{6} < b' < \frac{\lambda v(v-1)}{6}+\frac{v}{6}$ with $\lambda=1,3$. Thus $(v,b')$ associates with $(\lambda,\varepsilon)$ where $\lambda=1,3$ and $-\frac{v}{2}< \varepsilon < \frac{v}{2}$.  Apply Construction \ref{tbbb1} with $s=2$.
There exists a $1$-PCS$_{6}^{e}(6^{n}:2)$ for any $0\leq e \leq 3(n-1)$ and $e\equiv0\pmod{3}$ by Lemma \ref{wpcs}. An NWBTS$(8; c)$ exists for all $c$ with $(8,c)$ satisfying (C2) from Lemma \ref{81} in Appendix.  A simple S$_{6}(2,3,8)$ exists trivially. Then there exists an NWBTS$(v;b)$ for $v \equiv 2 \pmod{24}$ and $v\neq50,74$. \qed

Combining Lemmas \ref{v812}, \ref{v1424}, and \ref{v224} yields the main result of this subsection.

\begin{theorem}\label{v2}
Let $v \equiv 2\pmod{6}$, $v\neq50,74$, and $(v,  b)$ satisfy (C2). There exists an NWBTS$(v;b)$.
\end{theorem}

Analogous to the proofs of Lemmas \ref{wpcs} and \ref{v224}, if there exists a $3$-PCS$^{3}_{2}(6^{4}:2)$, then we could also prove Theorem \ref{v2} by applying Constructions \ref{constrution} and \ref{tbbb}. However, we did not find out a $3$-PCS$^{3}_{2}(6^{4}:2)$ and had to turn to its generalization $1$-PCS$^{3}_{6}(6^{4}:2)$. This design we constructed (Lemma \ref{wp} in Appendix) does not impose nice automorphism groups for all the component designs, because otherwise no positive result appeared after a long time of computer search.

\subsection{$v \equiv 4\pmod  6$}
In this subsection, we show the existence of an NWBTS$(v;b)$ if $v \equiv 10,16,22 \pmod{24}$ and $(v,  b)$ satisfies (C2).

\begin{lemma}\label{634}
There is a $3$-PCS$^{r}_{2}(6^{3}:4)$ for any $r=0, 3, 6$.
\end{lemma}

\proof We construct the desired design on $X = Z_{18} \cup \{a, b, c, d\}$ with stem $S = \{a, b, c, d\}$ and groups
$G_{i} = \{i, i + 3, i + 6,\ldots,i + 15\}, i = 0, 1, 2$. Let $\mathcal{H}$ be an automorphism group generated by the following two permutations:
$$\alpha = (0,6,12)(1,7,13)(2,8,14)(3,9,15)(4,10,16)(5,11,17)(a)(b)(c)(d)(e);$$
$$\beta = (0,2,4)(1,15,17)(6,8,10)(3,5,7)(12,14,16)(9,11,13)(a)(b)(c)(d)(e).$$

For $r=0, 6$, we first construct a GDD$_{2}(2, 3, 22)$ of type $1^{12}10^{1}$ with long group $G_{0} \cup S$ in the following; the underlined blocks form the blocks of an NGDD$(2,3,22)$ of type $2^{(0,6)}10^{1}$, where $\lambda_{2,11}=\lambda_{4,13}=\lambda_{8,17}=\lambda_{1,10}=\lambda_{5,14}=\lambda_{7,16}=0$.

{\small\begin{longtable}{lllllll}
$\{\underline{a, 1, 2}\}$ & $\{\underline{a, 7, 11}\}$ & $\{\underline{a, 5, 13}\}$ & $\{\underline{a, 4, 14}\}$ & $\{\underline{a, 8, 10}\}$ & $\{\underline{a, 16, 17}\}$ & $\{\underline{b, 1, 5}\}$\\
$\{\underline{b, 10, 17}\}$ & $\{\underline{b, 4, 8}\}$ & $\{\underline{b, 2, 7}\}$ & $\{\underline{b, 11, 13}\}$ & $\{\underline{b, 14, 16}\}$ & $\{\underline{c, 1, 8}\}$ & $\{\underline{c, 7, 17}\}$\\
$\{\underline{c, 13, 14}\}$ & $\{\underline{c, 2, 16}\}$ & $\{\underline{c, 4, 11}\}$ & $\{\underline{c, 5, 10}\}$ & $\{\underline{d, 1, 11}\}$ & $\{\underline{d, 10, 14}\}$ & $\{\underline{d, 2, 4}\}$\\
$\{\underline{d, 13, 17}\}$ & $\{\underline{d, 5, 16}\}$ & $\{\underline{d, 7, 8}\}$ & $\{\underline{0, 2, 5}\}$ & $\{\underline{1, 12, 14}\}$ & $\{\underline{6, 8, 16}\}$ & $\{\underline{6, 11, 17}\}$\\
$\{\underline{5, 11, 12}\}$ & $\{\underline{0, 8, 11}\}$ & $\{\underline{0, 14, 17}\}$ & $\{\underline{4, 12, 17}\}$ & $\{\underline{5, 15, 17}\}$ & $\{\underline{2, 8, 12}\}$ & $\{\underline{2, 6, 10}\}$\\
$\{\underline{2, 13, 15}\}$ & $\{\underline{3, 11, 14}\}$ & $\{\underline{2, 3, 17}\}$ & $\{\underline{1, 9, 17}\}$ & $\{\underline{2, 9, 14}\}$ & $\{\underline{3, 5, 8}\}$ & $\{\underline{6, 7, 14}\}$\\
$\{\underline{4, 5, 6}\}$ & $\{\underline{1, 6, 13}\}$ & $\{\underline{5, 7, 9}\}$ & $\{\underline{8, 9, 13}\}$ & $\{\underline{8, 14, 15}\}$ & $\{\underline{9, 11, 16}\}$ & $\{\underline{4, 9, 10}\}$\\
$\{\underline{10, 11, 15}\}$ & $\{\underline{10, 12, 16}\}$ & $\{\underline{4, 15, 16}\}$ & $\{\underline{1, 7, 15}\}$ & $\{\underline{7, 12, 13}\}$ & $\{\underline{0, 1, 4}\}$ & $\{\underline{0, 7, 10}\}$\\
$\{\underline{0, 13, 16}\}$ & $\{\underline{1, 3, 16}\}$ & $\{\underline{3, 4, 7}\}$ & $\{\underline{3, 10, 13}\}$ &
$\{a, 1, 8\}$ & $\{a, 11, 13\}$ & $\{10, 12, 13\}$\\
$\{a, 4, 17\}$ & $\{a, 5, 16\}$ & $\{a, 2, 7\}$ & $\{a, 10, 14\}$ & $\{0, 1, 10\}$ & $\{0, 5, 14\}$ & $\{1, 10, 15\}$\\
$\{8, 15, 17\}$ & $\{5, 14, 15\}$ & $\{2, 11, 15\}$ & $\{4, 13, 15\}$ & $\{7, 15, 16\}$ & $\{0, 2, 11\}$ & $\{0, 4, 13\}$\\
$\{0, 7, 16\}$ & $\{0, 8, 17\}$ & $\{b, 1, 11\}$ & $\{b, 16, 17\}$ & $\{b, 4, 14\}$ & $\{b, 2, 13\}$ & $\{b, 5, 10\}$\\
$\{b, 7, 8\}$ & $\{c, 5, 7\}$ & $\{c, 1, 14\}$ & $\{c, 13, 17\}$ & $\{c, 2, 4\}$ & $\{c, 8, 16\}$ & $\{c, 10, 11\}$\\
$\{d, 4, 5\}$ & $\{d, 7, 14\}$ & $\{d, 8, 13\}$ & $\{d, 11, 16\}$ & $\{d, 1, 17\}$ & $\{d, 2, 10\}$ & $\{4, 8, 10\}$\\
$\{2, 3, 8\}$ & $\{6, 8, 14\}$ & $\{2, 14, 16\}$ & $\{5, 8, 12\}$ & $\{8, 9, 11\}$ & $\{1, 2, 6\}$ & $\{2, 5, 9\}$\\
$\{2, 12, 17\}$ & $\{3, 4, 11\}$ & $\{3, 13, 14\}$ & $\{5, 6, 11\}$ & $\{1, 5, 13\}$ & $\{3, 5, 17\}$ & $\{7, 11, 17\}$\\
$\{6, 10, 17\}$ & $\{9, 14, 17\}$ & $\{11, 12, 14\}$ & $\{1, 3, 7\}$ & $\{3, 10, 16\}$ & $\{4, 6, 16\}$ & $\{6, 7, 13\}$\\
$\{1, 4, 9\}$ & $\{1, 12, 16\}$ & $\{4, 7, 12\}$ & $\{7, 9, 10\}$ & $\{9, 13, 16\}$\\
\end{longtable}}
\noindent Nine GDD$_{2}(2, 3, 20)$s are obtained by developing the above 124 base blocks under the action of the automorphism group $\mathcal{H}$ and the remaining blocks are forced by the definition of the $3$-PCS$^{0}_{2}(6^{3}:4)$.

For a $3$-PCS$^{3}_{2}(6^{3}:4)$, we first construct a GDD$_{2}(2, 3, 22)$ of type $1^{12}10^{1}$ with long group $G_{0} \cup S$ in the following; the underlined blocks form the blocks of an NGDD$(2,3,22)$ of type $2^{(3,3)}10^{1}$, where $\lambda_{2,11}=\lambda_{4,13}=\lambda_{8,17}=2$ and $\lambda_{1,10}=\lambda_{5,14}=\lambda_{7,16}=0$.

{\small\begin{longtable}{lllllll}
$\{\underline{a, 1, 2}\}$ & $\{\underline{a, 7, 14}\}$ & $\{\underline{a, 4, 8}\}$ & $\{\underline{a, 5, 13}\}$ & $\{\underline{a, 10, 17}\}$ & $\{\underline{a, 11, 16}\}$ & $\{\underline{7, 15, 16}\}$\\
$\{\underline{1, 10, 15}\}$ & $\{\underline{4, 11, 15}\}$ & $\{\underline{8, 13, 15}\}$ & $\{\underline{2, 15, 17}\}$ & $\{\underline{5, 14, 15}\}$ & $\{\underline{b, 2, 4}\}$ & $\{\underline{b, 7, 8}\}$\\
$\{\underline{b, 13, 17}\}$ & $\{\underline{b, 5, 16}\}$ & $\{\underline{b, 1, 14}\}$ & $\{\underline{b, 10, 11}\}$ & $\{\underline{c, 4, 5}\}$ & $\{\underline{c, 1, 17}\}$ & $\{\underline{c, 7, 11}\}$\\
$\{\underline{c, 8, 10}\}$ & $\{\underline{c, 13, 14}\}$ & $\{\underline{c, 2, 16}\}$ & $\{\underline{d, 4, 14}\}$ & $\{\underline{d, 1, 8}\}$ & $\{\underline{d, 2, 7}\}$ & $\{\underline{d, 5, 10}\}$\\
$\{\underline{d, 11, 13}\}$ & $\{\underline{d, 16, 17}\}$ & $\{\underline{0, 14, 17}\}$ & $\{\underline{0, 8, 11}\}$ & $\{\underline{0, 2, 5}\}$ & $\{\underline{1, 9, 11}\}$ & $\{\underline{2, 8, 9}\}$\\
$\{\underline{2, 3, 13}\}$ & $\{\underline{2, 12, 14}\}$ & $\{\underline{2, 6, 10}\}$ & $\{\underline{5, 8, 12}\}$ & $\{\underline{6, 8, 16}\}$ & $\{\underline{3, 8, 14}\}$ & $\{\underline{11, 12, 17}\}$\\
$\{\underline{3, 5, 11}\}$ & $\{\underline{1, 5, 7}\}$ & $\{\underline{3, 7, 17}\}$ & $\{\underline{6, 11, 14}\}$ & $\{\underline{5, 9, 14}\}$ & $\{\underline{5, 6, 17}\}$ & $\{\underline{4, 9, 17}\}$\\
$\{\underline{10, 14, 16}\}$ & $\{\underline{1, 4, 6}\}$ & $\{\underline{6, 7, 13}\}$ & $\{\underline{0, 1, 13}\}$ & $\{\underline{1, 10, 12}\}$ & $\{\underline{0, 4, 16}\}$ & $\{\underline{0, 7, 10}\}$\\
$\{\underline{1, 3, 16}\}$ & $\{\underline{3, 4, 10}\}$ & $\{\underline{4, 7, 12}\}$ & $\{\underline{7, 9, 16}\}$ & $\{\underline{9, 10, 13}\}$ & $\{\underline{12, 13, 16}\}$ & $\{a, 1, 17\}$\\
$\{a, 4, 5\}$ & $\{a, 7, 11\}$ & $\{a, 2, 10\}$ & $\{a, 8, 13\}$ & $\{a, 14, 16\}$ & $\{b, 1, 11\}$ & $\{b, 2, 13\}$\\
$\{b, 5, 7\}$ & $\{b, 10, 14\}$ & $\{b, 4, 17\}$ & $\{b, 8, 16\}$ & $\{c, 2, 7\}$ & $\{c, 1, 8\}$ & $\{c, 5, 13\}$\\
$\{c, 10, 17\}$ & $\{c, 4, 14\}$ & $\{c, 11, 16\}$ & $\{d, 8, 10\}$ & $\{d, 4, 11\}$ & $\{d, 1, 5\}$ & $\{d, 2, 16\}$\\
$\{d, 7, 17\}$ & $\{d, 13, 14\}$ & $\{8, 12, 17\}$ & $\{4, 12, 13\}$ & $\{0, 4, 13\}$ & $\{0, 7, 8\}$ & $\{0, 1, 14\}$\\
$\{0, 5, 11\}$ & $\{0, 2, 17\}$ & $\{0, 10, 16\}$ & $\{2, 11, 12\}$ & $\{2, 9, 11\}$ & $\{6, 8, 17\}$ & $\{4, 6, 7\}$\\
$\{5, 6, 16\}$ & $\{1, 2, 15\}$ & $\{1, 6, 13\}$ & $\{1, 12, 16\}$ & $\{2, 3, 5\}$ & $\{2, 4, 8\}$ & $\{2, 6, 14\}$\\
$\{3, 8, 11\}$ & $\{3, 14, 17\}$ & $\{4, 10, 15\}$ & $\{1, 4, 9\}$ & $\{1, 3, 7\}$ & $\{3, 4, 16\}$ & $\{3, 10, 13\}$\\
$\{5, 10, 12\}$ & $\{5, 8, 15\}$ & $\{5, 9, 17\}$ & $\{6, 10, 11\}$ & $\{7, 9, 10\}$ & $\{7, 12, 14\}$ & $\{7, 13, 15\}$\\
$\{8, 9, 14\}$ & $\{9, 13, 16\}$ & $\{11, 13, 17\}$ & $\{11, 14, 15\}$ & $\{15, 16, 17\}$\\
\end{longtable}}

\noindent Nine GDD$_{2}(2, 3, 20)$s are obtained by developing the above 124 base blocks under the action of the automorphism group $\mathcal{H}$; the remaining blocks are forced by the definition of the $3$-PCS$^{3}_{2}(6^{3}:4)$.\qed

\begin{lemma}\label{6n41}
For $n \equiv 1 \pmod{2}$ and $n\geq 3$. There exists a $3$-PCS$_{2}^{e}(6^{n}:4)$ for any $0\leq e \leq 3(n-1)$ and $e\equiv0\pmod{3}$.
\end{lemma}
\proof For $n=3$, see Lemma \ref{634}. For $n=5$, there exists a 2-FG$(3,(4,3,4),15)$ of type $3^{5}$, which can be obtained by deleting one point from an S$(3,4,16)$ containing a subdesign S$(2,4,16)$ (see \cite{35}). Apply Construction \ref{constrution} with $a=m=1$ and $h=l=2$. The input designs are $1$-PCS$^{i}_{2}(2^{4}:2)$s $(i=0,3)$ in \cite[Lemma 6.1]{wbt} and a PGDD$_{2}(2^{4})$ in Lemma \ref{PGDD}. Then we get a $3$-PCS$^{r}_{2}(6^{5}:4)$ for $r\in\{0,3,6,9,12\}$.

For $n\equiv 1$ (mod 2) and $n\geq7$, start with a 1-FG$(3, (\{3, 5\},\{4, 6\}), n)$ of type $1^{n}$ which exists by Lemma \ref{s46}. A PGDD$_{2}(6^{k})$ for $k\in\{4,6\}$ exists by Lemma \ref{PGDD} and use the above $3$-PCS$^{r}_{2}(6^{k}:4)$s for $k\in\{3,5\}$ and $r\le 3(k-1)$ with $r\equiv 0\pmod{3}$. Then we obtain the desired design by applying Construction \ref{constrution} with $h=a = 1$, $m = 3$, and $l = 4$.\qed

\begin{lemma}\label{v1012}
Let $v \equiv 10 \pmod{12}$ and $(v,  b)$ satisfy (C2). There exists an NWBTS$(v;b)$.
\end{lemma}

\proof For $v = 10$, see Lemma \ref{10} in Appendix. For $v \geq 22$, we only need to consider $b\leq \frac{v(v-1)(v-2)}{12}$ by Lemma \ref{bb}. Write $v = 6n +4$ where $n$ is odd, $b = qv(v-1) + b'$ where $0\leq q\leq \frac{v-2}{12}-1 < n-1$ and $\frac{\lambda v(v-1)}{6}-\frac{v}{6} < b' < \frac{\lambda v(v-1)}{6}+\frac{v}{6}$ with $\lambda=1,3,5$. Thus $(v,b')$ associates with $(\lambda,\varepsilon)$ where $\lambda=1,3,5$ and $-\frac{v}{2}< \varepsilon < \frac{v}{2}$.
  There exists a $3$-PCS$^{e}_{2}(6^{n}: 4)$ for any $0\leq e \leq 3(n-1)$ and $e\equiv0\pmod{3}$ from Lemma \ref{6n41}. An NWBTS$(10; c)$ exists for all $c$ with $(10,c)$ satisfying (C2) by Lemma \ref{10} in Appendix. There exists a simple S$_{6}(2,3,10)$ defined over a 10-set with each block not contained in a fixed 4-subset from Lemma \ref{S610} in Appendix. Then there exists an NWBTS$(v; b=qv(v-1)+b')$  by applying Construction \ref{tbbb}.\qed

\begin{lemma}\label{H12n4}
Let $n \equiv 1 \pmod{2}$ and $n \geq 3$. Then there is a $3^{+}$-PCS$^{e}_{2}(12^{n} : 4)$ for any $0 \leq e \leq 6(n-1)$ and $e\equiv0\pmod{6}$.
\end{lemma}
\proof For $n=3$, see Lemma \ref{H12340} in Appendix. For $n=5$, there exists a 2-FG$(3, (4, 4, 4), 30)$ of type $6^{5}$ from \cite[Lemma 3.5]{45}. Apply Construction \ref{constrution} with $a=m = 1$ and $h=l=2$; the input designs are $1$-PCS$^{i}_{2}(2^{4} : 2)$s $(i = 0, 3)$ in \cite[Lemma 6.1]{wbt} and PGDD$_{2}(2^{j})$s $(j = 4, 5)$ in Lemma \ref{PGDD}. Then we obtain a $6$-PCS$^{r}_{2}(12^{5}:4)$, also a $3^{+}$-PCS$^{r}_{2}(12^{5}:4)$ for $r \in\{0, 6, 12, 18, 24\}$.

For $n \equiv 1 \pmod{2}$ and $n \geq 7$, start with a 1-FG$(3, (\{3, 5\},\{4, 6\}), n)$ of type $1^{n}$ which exists by Lemma \ref{s46}. A PGDD$_{2}(12^{k})$ for $k\in\{4,6\}$ exists by Lemma \ref{PGDD} and use the above $3^{+}$-PCS$^{r}_{2}(12^{k} : 4)$ for $k\in\{3,5\}$ and $r\le 6(k-1)$ with $r\equiv 0\pmod{6}$. Apply Construction \ref{constrution1} with $m=3$ and $l=4$. So we obtain a $3^{+}$-PCS$^{e}_{2}(12^{n} : 4)$ where $0 \leq e \leq 6(n-1)$ and $e\equiv0\pmod{6}$.\qed

\begin{lemma}\label{S616}
Let $X$ be a set of $16$ points, $Y$ be a $4$-subset of $X$, and let $b\in[38,42]$. Then there exist two disjoint  designs defined over $X$, an S$_{6}(2,3,16)$ and an NWBTS$(16,b)$, whose blocks all belong to $\binom{X}{3}\backslash \binom{Y}{3}$.
\end{lemma}
\proof See Lemma \ref{ddf} in Appendix.\qed

\begin{lemma}\label{v1624}
Let $v \equiv 16 \pmod{24}$ and $(v,  b)$ satisfy (C2). There exists an NWBTS$(v;b)$.
\end{lemma}

\proof For $v = 16$, see Lemma \ref{N16} in Appendix. For $v \geq 40$, we only need to consider $b\leq \frac{v(v-1)(v-2)}{12}$ from Lemma \ref{bb}. Write $v = 12n +4$ where $n$ is odd. Further write $b = qv(v-1) + b'$ where $(v,b')$ associates with $(\lambda,\varepsilon)$, $-\frac{v}{2}< \varepsilon < \frac{v}{2}$, and if $0\leq q \leq \left \lfloor\frac{v-2}{12}\right \rfloor-1 = n-1$, then  $\lambda=1,3,5$, and that if  $q=n$, then $\lambda=1$.

Let $V = (Z_{n}\times Z_{3} \times Z_{4}) \cup S$, $|S|=4$, and $\mathcal{G} = \{G_{0}, G_{1},\ldots,G_{n-1}\}$ where $G_{i} = \{i\} \times Z_{3} \times Z_{4}$. Let $(V, S, \mathcal{G}, \mathcal{A})$ be a $3^{+}$-PCS$^{e}_{2}(12^{n}: 4)$ where $0 \leq e \leq 6(n-1)$ and $e\equiv0\pmod{6}$ from Lemma \ref{H12n4}. Then the block set $\mathcal{A}$ contains a subset $\mathcal{A}'$ and $\mathcal{A}'$ can be partitioned into $(3n+1)$ pairwise disjoint parts $\mathcal{A}(i, j)$, $0 \leq i \leq n-1$, $0\leq j\leq 2$, plus ${\cal A}(0,3)$, so that each $\mathcal{A}(i, j)$ is the block set of a GDD$_{2}(2, 3, 12n + 4)$ of type $1^{12(n-1)}16^{1}$ with long group $G_{i}\cup S$, and $\mathcal{A}(0, 3)$ contains a subset $\mathcal{A}'(0, 3)$ which is the block set of an NGDD$(2,3,12n+6)$ of type $2^{(e,6(n-1)-e)}16^{1}$. Applying Construction \ref{tbbb} with $\lambda =1,3,5$
yields the existence of an NWBTS$(v;qv(v-1) + b')$ for any $q\le n-1$, because a simple S$_{6}(2,3,16)$ exists by Lemma \ref{S616} with all blocks in $\binom{G_{i}\cup S}{3}\backslash \binom{S}{3}$ and an NWBTS$(16; c)$ exists for all $c$ with $(16,c)$ satisfying (C2).

 %there exists an NWBTS$(16; b_{\lambda})$ on $G_{0}\cup S$ with $\mathcal{B}_{\lambda}$ where $b_{1}\in[38,42]$, $b_{3}\in[118,122]$, and $b_{5}\in[198,202]$ from Lemma \ref{N16}, and $3b_{\lambda}=\lambda{16\choose 2}+\varepsilon_{\lambda}$.
%Construct a family $\mathcal{F}$ containing all blocks of $\mathcal{A}'(0,3)$, $(\lambda-1)/2$ set in $\{\mathcal{A}(0,i):i=1,2\}$ and all blocks of $\mathcal{B}_{\lambda}$. Then by Lemma \ref{tee11}, $\mathcal{F}$ is an NWBTS$(v; b')$ where $b'=\frac{\lambda v(v-1)}{6}-\frac{v}{6}+\frac{8+\varepsilon_{\lambda}+2e}{3}$;
%this covers the desirable block sizes of $b'$ when $q=0$, i.e., $\frac{\lambda v(v-1)}{6}-\frac{v}{6} < b' < \frac{\lambda v(v-1)}{6}+\frac{v}{6}$ where $\lambda=1,3,5$.

Now we consider that $q=n$. Thus $b = nv(v-1) + b'$ and $(v,b')$ associates with $(1,\varepsilon)$. By Lemma \ref{S616}, we may construct on $G_{i}\cup S$ ($0 \leq i \leq n-1$)  two disjoint  designs, an S$_{6}(2,3,16)$ with block set $\mathcal{B}(i)$  and an NWBTS$(16,c)$ ($c\in[38,42]$) with block set ${\cal C}(i)$, whose blocks all belong to
 $\binom{G_{i}\cup S}{3}\backslash \binom{S}{3}$. Let $\mathcal{F}(i) = (\bigcup^{2}_{j=0} \mathcal{A}(i, j))\cup \mathcal{B}(i)$. Each $\mathcal{F}(i)$ is the block set of a simple S$_{6}(2,3,12n + 4)$ and all $\mathcal{F}(i)$s are pairwise disjoint. Construct a family $\mathcal{F}=\mathcal{A}'(0,3)\cup{\cal C}(0)\cup(\bigcup_{i=0}^{n-1}\mathcal{F}(i))$.  Then by Lemmas \ref{tee11} and \ref{tee}, $\mathcal{F}$ is an NWBTS$(v; nv(v-1) + b')$ where $3b' ={v\choose 2}+\varepsilon $.\qed

 Combining Lemmas \ref{v1012} and  \ref{v1624} yields the main result of this subsection.
 \begin{theorem}\label{v4}
Let $v \equiv 10,16,22 \pmod{24}$ and $(v,  b)$ satisfy (C2). There exists an NWBTS$(v;b)$.
\end{theorem}

\section{Conclusion}

In this section we summarize our main results  and propose a promising approach to resolving entirely the existence problem of NWBTSs and  optimal data placements for triple replication.

\begin{theorem}\label{final}
Let $v\geq3$ and $b\geq0$ be integers. An NWBTS$(v;b)$   exists for any $v,b$ satisfying (C1) or (C2),  possibly except that $v\equiv4 \pmod{24}$ or $v=50,74$, and $\frac{\lambda v(v-1)}{6}-\frac{v}{6} < b < \frac{\lambda v(v-1)}{6}+\frac{v}{6}$ for an odd integer $\lambda$.
\end{theorem}

\proof Combine Theorems \ref{v26}, \ref{v06}, \ref{v2} and \ref{v4} to prove that  an NWBTS$(v;b)$ exists if $(v,b)$ satisfies (C1) or (C2) possibly except that $v\equiv 4 \pmod{24}$ or $v=50,74$, and $(v,b)$ satisfies (C2). \qed

\begin{theorem}\label{odpTS}
 An optimal data placement TS$(v;b)$ exists for any positive integers $v,b$ with $v\ge 3$,  possibly except that $v\equiv4 \pmod{24}$ or $v=50,74$, and $\frac{\lambda v(v-1)}{6}-\frac{v}{6} < b < \frac{\lambda v(v-1)}{6}+\frac{v}{6}$ for an odd integer $\lambda$.
\end{theorem}

\proof The conclusion follows by combining Corollary \ref{opt} with Theorems \ref{kk} and \ref{final}.\qed

For the possible exceptions in Theorems \ref{final} and \ref{odpTS}, we put forward a promising working approach.
Let $v \equiv 4 \pmod{24}$ and $(v,  b)$ satisfy (C2).
For $n\equiv 0$ (mod 4) and $n\neq8,12$,  there exists a 1-FG$(3, (\{3, 4, 5\},\{4, 5, 6\}), n)$ of type $1^{n}$.
A $3$-PCS$^{r}_{2}(6^{4}:4)$ (also a $1$-PCS$^{r}_{6}(6^{4}:4)$) for any $r\in\{0,9\}$ can be found by computer research.
Suppose that there exists a $1$-PCS$^{r}_{6}(6^{4}:4)$ for any $r\in\{3,6\}$.
Since a $3$-PCS$^{r}_{2}(6^{k}:4)$ exists for $k\in\{3,5\}$ and $r\leq3(k-1)$ with $r\equiv0\pmod{3}$ by Lemma \ref{6n41}, and a PGDD$_{2}(6^{k})~(k\in\{4,5,6\})$ also exists by Lemma \ref{PGDD}, we can obtain a $1$-PCS$_{6}^{e}(6^{n}:4)$ for any $0\leq e \leq 3(n-1)$ and $e\equiv0\pmod{3}$ by using Construction \ref{constrution14}. Then applying Construction \ref{tbbb1} we get an NWBTS$(v; b)$  ($v\not\in\{50,52,74,76\}$), also an optimal data placement for triple replication.

To sum up, we could reduce an almost solution of the existence of NWBTSs (and optimal data placements) to one small design $1$-PCS$^{3}_{6}(6^{4}:4)$. Certainly the final complete solution also depends on the small cases of $v\in\{50,52,74,76\}$. We conjecture that NWBTS$(v;b)$s always exist without no exceptions, although techniques of direct constructions and effective computational methods are two important concerns.

\bigskip\bigskip
%\newpage

%\section{Appendix: Supporting Information}
\newpage \appendix
{\centering\title{\bf\Large{ Appendix}}}
\renewcommand{\appendixname}{~\Alph{section}}
\section{NWBTSs}

\begin{lemma}\label{t1v242}
Let $v \in\{50,74\}$ and   $(v,  b)$ satisfy (C1). There exists an NWBTS$(v; b)$.
\end{lemma}

\proof By Lemma \ref{bb}, we assume that $b\leq \frac{v(v-1)(v-2)}{12}$. We first produce a nearly 2-balanced simple triple system.
Write $w = v/2$. For $w\in \{25, 37\}$, we partition $\frac{(w-1)^{2}}{24}$ inequivalent simple difference triples modulo $w$ into $\frac{w-1}{4}$ classes (the rows) each having $\frac{w-1}{6}$ difference triples $(a,b,c)$ with $a+b+c\equiv0\pmod{w}$; in each class, every nonzero element $(\pm)$ appears exactly once.

{\small\begin{longtable}{lllll}
$D_{00}$ & $d_{0}=(1,2,-3)$ & $f_{0}=(5,6,-11)$ & $(4,9,12)$ &$(7,8,10)$\\
$D_{01}$ & $(1,5,-6)$ & $(2,8,-10)$ & $(3,9,-12)$ &$(4,7,-11)$\\
$D_{02}$ & $(1,8,-9)$ & $(2,3,-5)$ & $(4,10,11)$ &$(6,7,12)$\\
\end{longtable}}

{\small\begin{longtable}{lllll}
$D_{10}$ & $d_{1}=(4,8,-12)$ & $(3,6,-9)$ & $(1,10,-11)$ &$(2,5,-7)$\\
$D_{11}$ & $(1,4,-5)$ & $(2,7,-9)$ & $(3,10,12)$ &$(6,8,11)$\\
$D_{12}$ & $(1,11,-12)$ & $(4,5,-9)$ & $(2,6,-8)$ &$(3,7,-10)$\\
\end{longtable}}

{\small\begin{longtable}{lllllll}
$D_{00}$ & $d_{0}=(2,8,-10)$ & $f_{0}=(7,11,-18)$ & $(1,12,-13)$ &$(3,14,-17)$ &$(4,5,-9)$ & $(6,15,16)$\\
$D_{01}$ & $(1,7,-8)$ & $(2,9,-11)$ & $(3,13,-16)$ &$(4,10,-14)$ &$(5,15,17)$ & $(6,12,-18)$\\
$D_{02}$ & $(1,8,-9)$ & $(2,13,-15)$ & $(3,7,-10)$ &$(4,12,-16)$ &$(5,14,18)$ & $(6,11,-17)$\\
$D_{10}$ & $d_{1}=(4,13,-17)$ & $(1,6,-7)$ & $(2,14,-16)$ &$(3,9,-12)$ &$(5,10,-15)$ & $(8,11,18)$\\
$D_{11}$ & $(1,2,-3)$ & $(4,15,18)$ & $(5,6,-11)$ &$(7,9,-16)$ &$(8,12,17)$ & $(10,13,14)$\\
$D_{12}$ & $(1,3,-4)$ & $(2,7,-9)$ & $(5,12,-17)$ &$(6,13,18)$ &$(8,14,15)$ & $(10,11,16)$\\
$D_{20}$ & $d_{2}=(3,6,-9)$ & $(1,10,-11)$ & $(2,12,-14)$ &$(4,16,17)$ &$(5,13,-18)$ & $(7,8,-15)$\\
$D_{21}$ & $(1,4,-5)$ & $(2,15,-17)$ & $(3,11,-14)$ &$(6,7,-13)$ &$(8,10,-18)$ & $(9,12,16)$\\
$D_{22}$ & $(1,5,-6)$ & $(2,17,18)$ & $(3,8,-11)$ &$(4,9,-13)$ &$(7,14,16)$ & $(10,12,15)$\\
\end{longtable}}

We treat difference triples as \emph{ordered}, using the ordering given. For each $0 \leq \alpha < \frac{w-1}{12}$, $0\leq \beta<3$, let $S_{\alpha,\beta} = \bigcup_{(a,b,c)\in D_{\alpha,\beta}}\bigcup_{i\in Z_{w}}\{\{i, a + i, a + b + i\}\}$, and $T_{\alpha,\beta}=\bigcup_{(a,b,c)\in D_{\alpha,\beta}}\bigcup_{i\in Z_{w}}\{\{i, b + i, a + b + i\}\}$, arithmetic in $Z_{w}$.
Then $|S_{\alpha,\beta}|=|T_{\alpha,\beta}|=\frac{w(w-1)}{6}$.
For each $\alpha$, let $R_{\alpha,0}=\bigcup_{(a,b,c)\in D_{\alpha,0}\setminus\{d_{\alpha}\}}$ $\bigcup_{i\in Z_{w}}\{\{i, a + i, a + b + i\}\}$.
Then $S_{\alpha,\beta}$ and $T_{\alpha,\beta}$ are S$(2,3,w)$s on $Z_{w}$. They share no triples with $S_{\alpha',\beta'}$ or $T_{\alpha',\beta'}$ unless $(\alpha, \beta) = (\alpha', \beta')$, because no (unordered) difference triple of $D_{\alpha, \beta}$ is equivalent to one in $D_{\alpha', \beta'}$. Moreover, $S_{\alpha,\beta}$ and $T_{\alpha,\beta}$ are disjoint, because $\{0, a, a + b\}\neq\{j, b + j, a + b + j \}$ when the difference triple is simple.

For $w = 25$, let $E = \{(0, p, q)\}=\{(0, 6, 20)\}$ and $C$ be a parallel class on $Z_{w}\backslash \{0\}$:
$\{(1, 7, 21)$, $(2, 8, 22)$, $(3, 4, 6), (5, 19, 24), (9, 14, 20), (10, 11, 13), (12, 17, 23), (15, 16, 18)\}$.

For $w = 37$, let $E = \{(0, p, q)\}=\{(0, 11, 30)\}$ and $C$ be a parallel class on $Z_{w}\backslash \{0\}$:
$\{(1, 3, 11)$, $(2, 13, 32)$, $(4, 23, 30), (5, 24, 31), (6, 17, 36), (7, 14, 25), (8, 27, 34), (9, 28, 35), (10, 12, 20), (15, 22, 33)$, $(16, 18, 26)$, $(19, 21, 29)\}$.

(All triples of $C$ are translates of $d_{0}$ or $f_{0}$ i.e., for given difference triple $(a,b,c)\in \{d_{0},f_{0}\}$, we have $\{i, a + i, a + b + i\}\in C$ for some $i\in Z_{w}$, and all triples in $E\cup C$ are ordered in the same way.) Let $C'$ be obtained from $C$ by adding $p$ to each entry (arithmetic in $Z_{w}$). Because $C$ contains all elements of $Z_{w}$ except for 0, $C'$ contains all elements except for $p$. Each triple in $C\cup C'$
is in $S_{0,0}$, and $C$ and $C'$ share no triples.

We form a nearly 2-balanced simple triple system on $Z_{w}\times \{0,1\}$, representing $(i, j)$ as $i_{j}$. For a 3-subset $B=\{a, b, c\}$ of $Z_{w}$ and $\sigma \in \{0, 1\}$, we set $l_{\sigma}(B) = \{\{a_{l}, b_{j} , c_{k}\} : l + j + k \equiv \sigma$ (mod 2)\}. When $d_{\alpha}=(a, b, c)$, let

$M_{\alpha} = \{\{i_{0}, i_{1}, (i + x)_{j} \} : i\in Z_{w}$, $x \in \{a, b, -(a + b)\},~j\in \{0, 1\}\}$,

$N_{\alpha} = \{\{i_{0}, i_{1}, (i + x)_{j} \} : i\in Z_{w}$, $x \in \{-a, -b, a + b\},~j\in \{0, 1\}\}$,

$P_{\alpha} = M_{\alpha}\cup l_{0}(R_{\alpha,0})\cup l_{1}(T_{\alpha,0})\cup l_{0}(S_{\alpha,1})\cup l_{1}(T_{\alpha,1})\cup l_{0}(S_{\alpha,2})\cup l_{1}(T_{\alpha,2})$,

$Q_{\alpha} = N_{\alpha}\cup l_{1}(R_{\alpha,0})\cup l_{0}(T_{\alpha,0})\cup l_{1}(S_{\alpha,1})\cup l_{0}(T_{\alpha,1})\cup l_{1}(S_{\alpha,2})\cup l_{0}(T_{\alpha,2})$.

\noindent For each $\alpha$, each of $P_{\alpha}$ and $Q_{\alpha}$ is an S$_{6}(2,3,v)$, the triple systems so produced are disjoint.

Write $b = (2q + q')v(v-1) + b'$ with $0\leq b'<v(v-1)$, $q'\in \{0, 1\}$, and $0 \leq q < \frac{w-1}{12}$. Place in $\mathcal{F}$ all triples in $P_{\alpha}$ and $Q_{\alpha}$ for $1\leq \alpha\leq q$, and also the triples of $Q_{0}$ when $q' = 1$. When $2q + q'=0$, $b=b'$. We treat cases for $b'$. Let $(v,b')$ associate with $(\lambda',\varepsilon)$ where $\lambda'=2,4$.

For $b'\in\{\frac{1}{6}(8w^{2}-4w-4), \frac{1}{6}(8w^{2}-4w+2)\}$. We have that $\lambda' = 2$ and $\varepsilon = -2, 1$. $\mathcal{F}'$ include
(1) all triples in $l_{0}(S_{0,0} \backslash C)$;
(2) $\{\{x_{0}, x_{1}, y_{j}\},\{y_{0}, y_{1}, z_{j}\},\{z_{0}, z_{1}, x_{j}\} : j \in\{0, 1\},(x, y,z)\in C\}$ (none is in $N_{0}$);
(3) $\{0_{0}, 0_{1}, p_{0}\}$;
(4) all triples in $l_{1}(T_{0,0})$ except for those that contain $\{0_{0}, p_{0}\}$ or $\{0_{1}, p_{0}\}$, denoted by $U = \{\{0_{0}, p_{0}, f_{1}\}, \{0_{1}, p_{0}, f_{0}\}\}$; and
(5) any $b'-\frac{1}{6}(8w^{2}-4w-10)$ triples from $U$.
The defect graph of $(Z_{w}\times \{0,1\},\mathcal{F}')$ is isomorphic to the graph $G_{-2}$ and $G_{1}$ in Figure 1, respectively.
So $\mathcal{F}'$ is a nearly 2-balanced simple triple system, also an NWBTS$(v;b')$ by Lemma \ref{NBTS}.

For $b'\in\{\frac{1}{6}(16w^{2}-8w-2), \frac{1}{6}(16w^{2}-8w+4)\}$. We have that $\lambda' = 4$ and $\varepsilon = -1, 2$. $\mathcal{F}'$ include
(1) all triples in $l_{0}(S_{0,0} \backslash (C\cup C'))$ except for two triples $\{0_{0}, p_{0}, q_{0}\}$ and $\{0_{1}, p_{1}, q_{0}\}$ in $l_{0}(E)$;
(2) all triples in $l_{1}(T_{0,0})$, $l_{0}(S_{0,1})$, and $l_{1}(T_{0,1})$;
(3) $\{\{x_{0}, x_{1}, y_{j}\},\{y_{0}, y_{1}, z_{j}\},\{z_{0}, z_{1}, x_{j}\} : j \in\{0, 1\},(x, y,z)\in C\cup C'\}$;
(4) $\{0_{0}, 0_{1}, q_{0}\}$ and $\{p_{0}, p_{_{1}}, q_{0}\}$; and
(5) any $b'-\frac{1}{6}(16w^{2}-8w-8)$ triples from $\{\{0_{0}, 0_{1}, p_{0}\}, \{0_{1}, p_{1}, f_{0}\}\}$, where $\{0_{1}, p_{1}, f_{0}\}$ which is contained in $l_{0}(S_{0,2})$, respectively.
The defect graph of $(Z_{w}\times \{0,1\},\mathcal{F}')$ is isomorphic to the graph $G_{-1}$ and $G_{2}$ in Figure 1, respectively.
So $\mathcal{F}'$ is a nearly 2-balanced simple triple system, also an NWBTS$(v;b')$ by Lemma \ref{NBTS}.

When $2q + q'\geq1$, $\mathcal{F} \cup \mathcal{F}'$ is an NWBTS$(v;b)$ by Lemma \ref{tee}.
\qed

\begin{lemma}\label{11}
There exists an NWBTS$(11; b)$, where $b\equiv18,19,36,37\pmod{55}$.
\end{lemma}
\proof We only need to consider that $b\in\{18,19,36,37,73,74\}$ by Lemma \ref{bb}. An NWBTS$(11; 18)$ can be constructed on $Z_{11}$, where $\lambda_{0,1}=2$ and $\lambda_{0,10}=\lambda_{1,10}=0$.
{\small\begin{longtable}{lllllllll}
$\{2, 6, 10\}$ & $\{3, 4, 10\}$ & $\{7, 9, 10\}$ & $\{5, 8, 10\}$ & $\{0, 2, 4\}$ & $\{0, 1, 7\}$ & $\{1, 3, 5\}$ & $\{1, 4, 9\}$ & $\{3, 6, 7\}$\\
$\{0, 3, 8\}$ & $\{0, 1, 6\}$ & $\{0, 5, 9\}$ & $\{2, 5, 7\}$ & $\{1, 2, 8\}$ & $\{2, 3, 9\}$ & $\{4, 5, 6\}$ & $\{4, 7, 8\}$ & $\{6, 8, 9\}$\\

\end{longtable}}

\noindent We can get an NWBTS$(11; 19)$ by adding a block $\{1,2,10\}$ on the above NWBTS$(11; 18)$, where $\lambda_{0,10}=0$ and $\lambda_{0,1}=\lambda_{2,10}=\lambda_{1,2}=2$.

We get a simple S$_{3}(2,3,11)$, whose blocks are generated by the base blocks under the action of the group $Z_{11}$, which are different from all blocks of the above NWBTS$(11; 19)$.

{\small\begin{longtable}{llllllll}
\{0, 1, 3\} & \{0, 1, 4\} & \{0, 1, 5\} & \{0, 2, 5\} & \{0, 2, 6\}\\
\end{longtable}}
\noindent Then we can obtain an NWBTS$(11; l)$ $(l = 73,74)$ by combining all blocks of the S$_{3}(2,3,11)$ and the NWBTS$(11; r)$ $(r=18,19)$, respectively.

An NWBTS$(11; 36)$ can be constructed on $Z_{11}$, where $\lambda_{0,1}=3$ and $\lambda_{0,9}=\lambda_{1,10}=\lambda_{9,10}=1$.

{\small\begin{longtable}{lllllllll}
$\{0, 1, 2\}$ & $\{0, 1, 3\}$ & $\{0, 1, 4\}$ & $\{0, 2, 9\}$ & $\{0, 3, 4\}$ & $\{0, 5, 6\}$ & $\{0, 5, 7\}$ & $\{0, 6, 8\}$ & $\{2, 6, 7\}$\\
$\{0, 7, 10\}$ & $\{0, 8, 10\}$ & $\{1, 2, 10\}$ & $\{1, 3, 4\}$ & $\{1, 5, 6\}$ & $\{1, 5, 7\}$ & $\{1, 6, 8\}$ & $\{1, 7, 9\}$ & $\{2, 6, 9\}$\\
$\{1, 8, 9\}$ & $\{2, 3, 5\}$ & $\{3, 5, 8\}$ & $\{2, 4, 5\}$ & $\{4, 5, 9\}$ & $\{5, 8, 10\}$ & $\{5, 9, 10\}$ & $\{3, 8, 9\}$ & $\{3, 6, 7\}$\\
$\{2, 4, 8\}$ & $\{2, 7, 8\}$ & $\{4, 7, 8\}$ & $\{4, 6, 9\}$ & $\{4, 6, 10\}$ & $\{4, 7, 10\}$ & $\{2, 3, 10\}$ & $\{3, 6, 10\}$ & $\{3, 7, 9\}$\\

\end{longtable}}

\noindent We can get an NWBTS$(11; 37)$ by adding a block $\{0,9,10\}$ on the NWBTS$(11; 36)$, where $\lambda_{1,10}=1$ and $\lambda_{0,1}=\lambda_{0,10}=3$.
\qed

\begin{lemma}\label{81}
There exists an NWBTS$(8; b)$, where $b\equiv 9,10,18,19,27,28,29,37,38,46,47\pmod{56}$.
\end{lemma}

\proof We only need to consider that $b\in\{9,10,18,19,27,28\}$ by Lemma \ref{bb}.
We first construct an NWBTS$(8; 9)$ on $Z_{8}$, where $\lambda_{0,1}=\lambda_{2,3}=2$ and $\lambda_{3,5}=\lambda_{3,6}=\lambda_{4,7}=0$.
{\small\begin{longtable}{lllllllll}
$\{0, 1, 2\}$ & $\{0, 1, 3\}$ & $\{0, 4, 5\}$ & $\{0, 6, 7\}$ & $\{1, 4, 6\}$ & $\{2, 3, 4\}$ & $\{2, 3, 7\}$ & $\{2, 5, 6\}$ & $\{1, 5, 7\}$\\
\end{longtable}}
\noindent We can get an NWBTS$(8; 10)$ by adding a block $\{3,5,6\}$ on the above NWBTS$(8; 9)$, where $\lambda_{4,7}=0$ and $\lambda_{0,1}=\lambda_{2,3}=\lambda_{5,6}=2$.

An NWBTS$(8; 18)$ can be constructed on $Z_{8}$, where $\lambda_{0,1}=3$ and $\lambda_{0,7}=\lambda_{1,6}=\lambda_{6,7}=1$.
{\small\begin{longtable}{llllllllll}
$\{0, 1, 2\}$ & $\{0, 1, 3\}$ & $\{0, 1, 4\}$ & $\{0, 2, 6\}$ & $\{0, 3, 7\}$ & $\{0, 4, 5\}$ & $\{2, 3, 7\}$ & $\{2, 4, 6\}$ & $\{2, 4, 7\}$\\
$\{0, 5, 6\}$ & $\{1, 2, 5\}$ & $\{1, 3, 6\}$ & $\{1, 4, 7\}$ & $\{1, 5, 7\}$ & $\{2, 3, 5\}$ & $\{3, 4, 5\}$ & $\{3, 4, 6\}$ & $\{5, 6, 7\}$\\
\end{longtable}}
\noindent We can get an NWBTS$(8; 19)$ by adding a block $\{0,6,7\}$ on the above  NWBTS$(8; 18)$, where $\lambda_{1,6}=1$ and $\lambda_{0,1}=\lambda_{0,6}=3$.

Next we construct an NWBTS$(8; 27)$ on $Z_{8}$, where $\lambda_{2,3}=4$ and $\lambda_{4,5}=\lambda_{6,7}=\lambda_{1,2}=\lambda_{0,2}=2$.
{\small\begin{longtable}{lllllllll}
$\{3, 5, 7\}$ & $\{0, 1, 3\}$ & $\{0, 1, 4\}$ & $\{0, 1, 5\}$ & $\{0, 2, 3\}$ & $\{0, 2, 4\}$ & $\{0, 3, 6\}$ & $\{1, 4, 6\}$ & $\{3, 4, 6\}$\\
$\{0, 4, 7\}$ & $\{0, 5, 6\}$ & $\{0, 5, 7\}$ & $\{0, 6, 7\}$ & $\{1, 2, 3\}$ & $\{2, 3, 5\}$ & $\{2, 3, 6\}$ & $\{3, 4, 7\}$ & $\{1, 6, 7\}$\\
$\{1, 2, 7\}$ & $\{2, 4, 6\}$ & $\{2, 4, 7\}$ & $\{2, 5, 6\}$ & $\{2, 5, 7\}$ & $\{1, 4, 5\}$ & $\{3, 4, 5\}$ & $\{1, 5, 6\}$ & $\{1, 3, 7\}$\\
\end{longtable}}
\noindent We can get an NWBTS$(8; 28)$ by adding a block $\{0,1,2\}$ on the above  NWBTS$(8; 27)$, where $\lambda_{0,1}=\lambda_{2,3}=4$ and $\lambda_{4,5}=\lambda_{6,7}=2$.
\qed

\begin{lemma}\label{14}
There exists an NWBTS$(14; b)$, where $b\equiv 60,61,121,122 \pmod{182}$ for (C1) and $b\equiv29,30,31,32,89,90,91,92,93,150,151,152,153\pmod{182}$ for (C2).
\end{lemma}
\proof We only need to consider that $b\in[60,61]\cup[121,122]$ for (C1) and $b\in[29,32]\cup[89,93]\cup[150,153]$ for (C2) by Lemma \ref{bb}. We first construct an NWBTS$(14; 60)$ on $Z_{14}$, where $\lambda_{0,1}=3$ and $\lambda_{0,12}=\lambda_{1,13}=\lambda_{12,13}=1$.
{\small\begin{longtable}{llllllll}
$\{0, 1, 2\}$ & $\{0, 1, 3\}$ & $\{0, 1, 4\}$ & $\{0, 2, 12\}$ & $\{0, 3, 4\}$ & $\{0, 5, 6\}$ & $\{6, 8, 12\}$ & $\{4, 6, 9\}$\\
$\{0, 5, 7\}$ & $\{0, 6, 7\}$ & $\{0, 8, 9\}$ & $\{0, 8, 10\}$ & $\{0, 9, 11\}$ & $\{0, 10, 13\}$ & $\{4, 6, 13\}$ & $\{6, 9, 13\}$\\
$\{0, 11, 13\}$ & $\{1, 2, 13\}$ & $\{1, 3, 4\}$ & $\{1, 5, 6\}$ & $\{1, 5, 7\}$ & $\{1, 6, 7\}$ & $\{2, 7, 9\}$ & $\{7, 9, 12\}$\\
$\{1, 8, 9\}$ & $\{1, 8, 10\}$ & $\{1, 9, 11\}$ & $\{1, 10, 12\}$ & $\{1, 11, 12\}$ & $\{2, 5, 8\}$ & $\{2, 4, 12\}$, & $\{4, 8, 12\}$\\
$\{3, 5, 8\}$ & $\{2, 3, 5\}$ & $\{4, 5, 9\}$ & $\{5, 9, 10\}$ & $\{4, 5, 10\}$ & $\{5, 11, 12\}$ & $\{2, 4, 7\}$ & $\{2, 8, 13\}$\\
$\{5, 11, 13\}$ & $\{5, 12, 13\}$ & $\{2, 3, 11\}$ & $\{3, 6, 11\}$ & $\{3, 6, 8\}$ & $\{3, 7, 12\}$ & $\{4, 7, 13\}$ & $\{7, 8, 13\}$\\
$\{3, 9, 12\}$ & $\{3, 7, 10\}$ & $\{3, 9, 13\}$ & $\{3, 10, 13\}$ & $\{2, 6, 11\}$ & $\{4, 8, 11\}$ & $\{2, 9, 10\}$ & $\{6, 10, 12\}$\\
$\{7, 8, 11\}$ & $\{4, 10, 11\}$ & $\{7, 10, 11\}$ & $\{2, 6, 10\}$\\
\end{longtable}}

\noindent We can get an NWBTS$(14; 61)$ by adding a block $\{1,12,13\}$ on the above NWBTS$(14; 60)$, where $\lambda_{0,12}=1$ and $\lambda_{0,1}=\lambda_{1,12}=3$.

We construct an NWBTS$(14; 121)$ on $Z_{14}$, where $\lambda_{0,1}=5$ and $\lambda_{0,13}=\lambda_{1,13}=3$.
{\small\begin{longtable}{llllllll}
$\{0, 1, 2\}$ & $\{0, 1, 3\}$ & $\{0, 1, 4\}$ & $\{0, 1, 5\}$ & $\{0, 1, 6\}$ & $\{0, 2, 13\}$ & $\{0, 3, 13\}$ & $\{0, 4, 13\}$\\
$\{0, 2, 3\}$ & $\{0, 2, 4\}$ & $\{0, 3, 4\}$ & $\{0, 5, 6\}$ & $\{0, 5, 7\}$ & $\{0, 5, 8\}$ & $\{0, 6, 7\}$ & $\{0, 6, 8\}$\\
$\{0, 7, 9\}$ & $\{0, 7, 10\}$ & $\{0, 8, 11\}$ & $\{0, 8, 12\}$ & $\{0, 9, 10\}$ & $\{0, 9, 11\}$ & $\{0, 9, 12\}$ & $\{0, 10, 11\}$\\
$\{0, 10, 12\}$ & $\{0, 11, 12\}$ & $\{1, 2, 13\}$ & $\{1, 3, 13\}$ & $\{1, 4, 13\}$ & $\{1, 2, 3\}$ & $\{1, 2, 4\}$ & $\{1, 3, 4\}$\\
$\{1, 5, 6\}$ & $\{1, 5, 7\}$ & $\{1, 5, 8\}$ & $\{1, 6, 7\}$ & $\{1, 6, 8\}$ & $\{1, 7, 9\}$ & $\{1, 7, 10\}$ & $\{1, 8, 11\}$\\
$\{1, 8, 12\}$ & $\{1, 9, 10\}$ & $\{1, 9, 11\}$ & $\{1, 9, 12\}$ & $\{1, 10, 11\}$ & $\{1, 10, 12\}$ & $\{1, 11, 12\}$ & $\{2, 3, 4\}$\\
$\{2, 3, 5\}$ & $\{2, 4, 5\}$ & $\{2, 5, 6\}$ & $\{2, 5, 7\}$ & $\{2, 6, 7\}$ & $\{2, 6, 8\}$ & $\{2, 6, 9\}$ & $\{2, 7, 8\}$\\
$\{2, 7, 9\}$ & $\{2, 8, 10\}$ & $\{2, 8, 11\}$ & $\{2, 9, 10\}$ & $\{2, 9, 12\}$ & $\{2, 10, 11\}$ & $\{2, 10, 12\}$ & $\{2, 11, 12\}$\\
$\{2, 11, 13\}$ & $\{2, 12, 13\}$ & $\{3, 4, 5\}$ & $\{3, 5, 7\}$ & $\{3, 5, 9\}$ & $\{3, 6, 9\}$ & $\{3, 6, 10\}$ & $\{3, 6, 11\}$\\
$\{3, 6, 12\}$ & $\{3, 7, 10\}$ & $\{3, 7, 11\}$ & $\{3, 7, 12\}$ & $\{3, 8, 9\}$ & $\{3, 8, 10\}$ & $\{3, 8, 11\}$ & $\{3, 8, 12\}$\\
$\{3, 9, 11\}$ & $\{3, 10, 13\}$ & $\{3, 12, 13\}$ & $\{5, 8, 9\}$ & $\{5, 8, 10\}$ & $\{4, 5, 11\}$ & $\{4, 5, 12\}$ & $\{5, 6, 12\}$\\
$\{5, 9, 10\}$ & $\{5, 9, 13\}$ & $\{5, 10, 11\}$ & $\{5, 10, 13\}$ & $\{5, 11, 12\}$ & $\{5, 11, 13\}$ & $\{5, 12, 13\}$ & $\{4, 6, 11\}$\\
$\{4, 7, 11\}$ & $\{4, 9, 11\}$ & $\{6, 7, 11\}$ & $\{6, 11, 13\}$ & $\{7, 11, 13\}$ & $\{4, 6, 9\}$ & $\{6, 9, 13\}$ & $\{4, 6, 10\}$\\
$\{4, 6, 12\}$ & $\{6, 8, 13\}$ & $\{6, 10, 12\}$ & $\{6, 10, 13\}$ & $\{4, 7, 10\}$ & $\{4, 8, 10\}$ & $\{4, 10, 13\}$ & $\{4, 7, 8\}$\\
$\{4, 7, 12\}$ & $\{4, 8, 9\}$ & $\{4, 8, 13\}$ & $\{4, 9, 12\}$ & $\{7, 9, 13\}$ & $\{8, 9, 13\}$ & $\{7, 8, 12\}$ & $\{7, 8, 13\}$\\
$\{7, 12, 13\}$\\
\end{longtable}}

\noindent We can obtain an NWBTS$(14; 122)$ by adding a block $\{0,5,13\}$ on the above NWBTS$(14; 121)$, where $\lambda_{1,13}=3$ and $\lambda_{0,1}=\lambda_{5,13}=\lambda_{0,5}=5$.

Next we construct an NWBTS$(14; 29)$ on $Z_{14}$, where $\lambda_{2,3}=\lambda_{4,5}=2$ and $\lambda_{0,5}=\lambda_{1,5}=\lambda_{6,7}=\lambda_{8,9}=\lambda_{10,11}=\lambda_{12,13}=0$.
{\small\begin{longtable}{llllllll}
$\{0, 6, 8\}$ & $\{1, 6, 10\}$ & $\{2, 6, 12\}$ & $\{3, 4, 6\}$ & $\{5, 6, 9\}$ & $\{6, 11, 13\}$ & $\{1, 7, 9\}$ & $\{5, 7, 11\}$\\
$\{2, 8, 10\}$ & $\{1, 8, 13\}$ & $\{3, 5, 8\}$ & $\{4, 7, 8\}$ & $\{8, 11, 12\}$ & $\{0, 3, 10\}$ & $\{4, 9, 11\}$ & $\{9, 10, 12\}$\\
$\{1, 3, 12\}$ & $\{2, 3, 11\}$ & $\{2, 3, 7\}$ & $\{3, 9, 13\}$ & $\{0, 2, 9\}$ & $\{1, 2, 4\}$ & $\{4, 5, 10\}$ & $\{0, 1, 11\}$\\
$\{2, 5, 13\}$ & $\{0, 4, 13\}$ & $\{7, 10, 13\}$ & $\{0, 7, 12\}$ & $\{4, 5, 12\}$\\

\end{longtable}}

\noindent We can get an NWBTS$(14; 30)$ by adding a block $\{0, 1, 5\}$ on the above NWBTS$(14; 29)$, where $\lambda_{0,1}=\lambda_{2,3}=\lambda_{4,5}=2$ and $\lambda_{6,7}=\lambda_{8,9}=\lambda_{10,11}=\lambda_{12,13}=0$.
We can obtain an NWBTS$(14; 31)$ by adding two blocks $\{0, 1, 5\}$ and $\{6, 7, 8\}$ on the above NWBTS$(14; 29)$, where $\lambda_{0,1}=\lambda_{2,3}=\lambda_{4,5}=\lambda_{6,8}=\lambda_{7,8}=2$ and $\lambda_{8,9}=\lambda_{10,11}=\lambda_{12,13}=0$.

We construct an NWBTS$(14; 32)$ on $Z_{14}$, where $\lambda_{0,1}=0$ and $\lambda_{2,3}=\lambda_{4,5}=\lambda_{6,7}=\lambda_{8,9}=\lambda_{10,11}=\lambda_{12,13}=2$.

{\small\begin{longtable}{llllllll}
$\{0, 2, 3\}$ & $\{0, 4, 5\}$ & $\{0, 6, 7\}$ & $\{0, 8, 9\}$ & $\{0, 10, 11\}$ & $\{0, 12, 13\}$ & $\{7, 8, 11\}$ & $\{3, 7, 12\}$\\
$\{1, 2, 3\}$ & $\{1, 4, 5\}$ & $\{1, 6, 7\}$ & $\{1, 8, 9\}$ & $\{1, 10, 11\}$ & $\{1, 12, 13\}$ & $\{6, 8, 12\}$ & $\{3, 8, 13\}$\\
$\{2, 4, 6\}$ & $\{2, 5, 8\}$ & $\{2, 7, 10\}$ & $\{2, 9, 12\}$ & $\{2, 11, 13\}$ & $\{3, 4, 9\}$ & $\{3, 6, 10\}$ & $\{5, 6, 13\}$\\
$\{4, 8, 10\}$ & $\{4, 7, 13\}$ & $\{4, 11, 12\}$ & $\{5, 7, 9\}$ & $\{6, 9, 11\}$ & $\{3, 5, 11\}$ & $\{5, 10, 12\}$ & $\{9, 10, 13\}$\\
\end{longtable}}

An NWBTS$(14; 89)$ can be constructed on $Z_{14}$, where $\lambda_{2,3}=4$ and $\lambda_{0,2}=\lambda_{1,2}=\lambda_{4,5}=\lambda_{6,7}=\lambda_{8,9}=\lambda_{10,11}=\lambda_{12,13}=2$.

{\small\begin{longtable}{llllllll}
$\{4, 5, 11\}$ & $\{0, 1, 3\}$ & $\{0, 1, 4\}$ & $\{0, 1, 5\}$ & $\{0, 2, 3\}$ & $\{0, 2, 4\}$ & $\{0, 3, 4\}$& $\{5, 8, 10\}$\\
$\{0, 5, 6\}$ & $\{0, 5, 7\}$ & $\{0, 6, 7\}$ & $\{0, 6, 8\}$ & $\{0, 7, 8\}$ & $\{0, 8, 9\}$ & $\{0, 9, 10\}$& $\{4, 5, 10\}$\\
$\{0, 9, 11\}$ & $\{0, 10, 12\}$ & $\{0, 10, 13\}$ & $\{0, 11, 12\}$ & $\{0, 11, 13\}$ & $\{0, 12, 13\}$ & $\{4, 8, 11\}$& $\{3, 8, 11\}$\\
$\{1, 2, 3\}$ & $\{2, 3, 4\}$ & $\{2, 3, 5\}$ & $\{1, 2, 4\}$ & $\{2, 5, 6\}$ & $\{2, 5, 7\}$ & $\{2, 6, 7\}$& $\{3, 5, 11\}$\\
$\{2, 6, 8\}$ & $\{2, 7, 8\}$ & $\{2, 8, 9\}$ & $\{2, 9, 10\}$ & $\{2, 9, 11\}$ & $\{2, 10, 12\}$ & $\{2, 10, 13\}$&$\{3, 8, 10\}$\\
$\{2, 11, 12\}$ & $\{2, 11, 13\}$ & $\{2, 12, 13\}$ & $\{1, 6, 8\}$ & $\{1, 6, 12\}$ & $\{3, 6, 12\}$ & $\{4, 6, 12\}$& $\{3, 5, 10\}$\\
$\{1, 3, 6\}$ & $\{3, 6, 9\}$ & $\{4, 6, 9\}$ & $\{4, 6, 10\}$ & $\{5, 6, 11\}$ & $\{6, 9, 13\}$ & $\{6, 10, 11\}$& $\{7, 10, 11\}$\\
$\{6, 10, 13\}$ & $\{6, 11, 13\}$ & $\{1, 4, 13\}$ & $\{1, 5, 13\}$ & $\{1, 7, 13\}$ & $\{3, 8, 13\}$ & $\{4, 8, 13\}$& $\{4, 7, 11\}$\\
$\{5, 8, 13\}$ & $\{3, 4, 13\}$ & $\{3, 7, 13\}$ & $\{5, 9, 13\}$ & $\{7, 9, 13\}$ & $\{1, 8, 10\}$ & $\{1, 8, 11\}$& $\{4, 7, 10\}$\\
$\{1, 5, 9\}$ & $\{1, 7, 9\}$ & $\{1, 7, 10\}$ & $\{1, 9, 11\}$ & $\{1, 10, 12\}$ &$\{1, 11, 12\}$ & $\{3, 9, 10\}$& $\{3, 7, 11\}$\\
$\{3, 9, 12\}$ & $\{3, 7, 12\}$ & $\{4, 7, 9\}$ &$\{4, 9, 12\}$ & $\{5, 9, 12\}$ & $\{4, 8, 12\}$ & $\{5, 7, 12\}$& $\{7, 8, 12\}$\\
$\{5, 8, 12\}$\\
\end{longtable}}

\noindent We can obtain an NWBTS$(14; 90)$ by adding a block $\{0, 1, 2\}$ on the above NWBTS$(14; 89)$, where $\lambda_{0,1}=\lambda_{2,3}=4$ and $\lambda_{4,5}=\lambda_{6,7}=\lambda_{8,9}=\lambda_{10,11}=\lambda_{12,13}=2$.

We construct an NWBTS$(14; 91)$ on $Z_{14}$, where $\lambda_{0,1}=\lambda_{4,5}=\lambda_{6,7}=\lambda_{8,9}=4$ and $\lambda_{0,2}=\lambda_{0,3}=\lambda_{10,11}=\lambda_{12,13}=2$.

{\small\begin{longtable}{llllllll}
$\{0, 1, 2\}$ & $\{0, 2, 4\}$ & $\{0, 1, 3\}$ & $\{0, 1, 4\}$ & $\{0, 1, 5\}$ & $\{0, 3, 4\}$ & $\{0, 5, 6\}$ & $\{3, 5, 9\}$\\
$\{0, 5, 7\}$ & $\{0, 6, 7\}$ & $\{0, 6, 8\}$ & $\{0, 7, 8\}$ & $\{0, 8, 9\}$ & $\{0, 9, 10\}$ & $\{0, 9, 11\}$ & $\{4, 5, 9\}$\\
$\{0, 10, 12\}$ & $\{0, 10, 13\}$ & $\{0, 11, 12\}$ & $\{0, 11, 13\}$ & $\{0, 12, 13\}$ & $\{1, 10, 11\}$ & $\{7, 9, 12\}$ & $\{4, 7, 9\}$\\
$\{2, 10, 11\}$ & $\{1, 2, 10\}$ & $\{1, 3, 10\}$ & $\{2, 3, 10\}$ & $\{3, 4, 10\}$ & $\{4, 5, 10\}$ & $\{4, 6, 10\}$ & $\{4, 9, 12\}$\\
$\{5, 6, 10\}$ & $\{5, 7, 10\}$ & $\{6, 7, 10\}$ & $\{7, 8, 10\}$ & $\{8, 9, 10\}$ & $\{8, 10, 12\}$ & $\{9, 10, 13\}$ & $\{4, 7, 11\}$\\
$\{10, 12, 13\}$ & $\{1, 2, 12\}$ & $\{1, 3, 12\}$ & $\{1, 4, 12\}$ & $\{1, 4, 5\}$ & $\{1, 5, 6\}$ & $\{1, 6, 7\}$ & $\{5, 7, 12\}$\\
$\{1, 6, 8\}$ & $\{1, 7, 8\}$ & $\{1, 7, 9\}$ & $\{1, 8, 13\}$ & $\{1, 9, 11\}$ & $\{1, 9, 13\}$ & $\{1, 11, 13\}$ & $\{5, 11, 12\}$\\
$\{2, 6, 7\}$ & $\{2, 6, 12\}$ & $\{3, 6, 12\}$ & $\{4, 6, 12\}$ & $\{2, 3, 6\}$ & $\{3, 6, 9\}$ & $\{4, 6, 11\}$ & $\{5, 8, 12\}$\\
$\{6, 8, 13\}$ & $\{6, 9, 11\}$ & $\{6, 9, 13\}$ & $\{6, 11, 13\}$ & $\{2, 7, 13\}$ & $\{3, 7, 13\}$ & $\{4, 7, 13\}$ & $\{4, 8, 11\}$\\
$\{2, 3, 13\}$ & $\{2, 4, 13\}$ & $\{3, 5, 13\}$ & $\{4, 5, 13\}$ & $\{5, 8, 13\}$ & $\{2, 8, 9\}$ & $\{3, 8, 9\}$ & $\{3, 7, 12\}$\\
$\{2, 4, 8\}$ & $\{2, 5, 8\}$ & $\{2, 5, 9\}$ & $\{2, 9, 12\}$ & $\{2, 5, 11\}$ & $\{2, 7, 11\}$ & $\{3, 4, 8\}$ & $\{3, 7, 11\}$\\
$\{3, 8, 11\}$ & $\{8, 11, 12\}$ & $\{3, 5, 11\}$\\

      \\
\end{longtable}}
\noindent We can get an NWBTS$(14; 92)$ by adding a block $\{0, 2, 3\}$ on the above NWBTS$(14; 91)$, where $\lambda_{10,11}=\lambda_{12,13}=2$ and $\lambda_{0,1}=\lambda_{2,3}=\lambda_{4,5}=\lambda_{6,7}=\lambda_{8,9}=4$.
We can obtain an NWBTS$(14; 93)$ by adding two blocks $\{0, 2, 3\}$ and $\{10, 11, 12\}$ on the above NWBTS$(14; 91)$, where $\lambda_{0,1}=\lambda_{2,3}=\lambda_{4,5}=\lambda_{6,7}=\lambda_{8,9}=\lambda_{10,12}=\lambda_{11,12}=4$ and $\lambda_{12,13}=2$.

We construct an NWBTS$(14; 150)$ on $Z_{14}$, whose blocks are generated by the following 75 base blocks modulo 7,
where $\lambda_{0,7}=6$ and $\lambda_{1,8}=\lambda_{2,9}=\lambda_{3,10}=\lambda_{4,11}=\lambda_{5,12}=\lambda_{6,13}=4$.

{\small\begin{longtable}{llllllll}
$\{0, 1, 7\}$ & $\{0, 2, 7\}$ & $\{0, 3, 7\}$ & $\{0, 1, 8\}$ & $\{1, 2, 8\}$ & $\{0, 2, 9\}$ & $\{1, 2, 9\}$ & $\{4, 12, 13\}$\\
$\{0, 3, 10\}$ & $\{1, 3, 10\}$ & $\{0, 4, 11\}$ & $\{1, 4, 11\}$ & $\{0, 5, 12\}$ & $\{1, 5, 12\}$ & $\{0, 6, 13\}$ & $\{4, 6, 12\}$\\
$\{1, 6, 13\}$ & $\{0, 1, 2\}$ & $\{0, 1, 3\}$ & $\{0, 1, 4\}$ & $\{0, 2, 3\}$ & $\{0, 2, 4\}$ & $\{0, 3, 4\}$ & $\{3, 12, 13\}$\\
$\{0, 4, 5\}$ & $\{0, 5, 6\}$ & $\{0, 5, 8\}$ & $\{0, 5, 9\}$ & $\{0, 6, 8\}$ & $\{0, 6, 9\}$ & $\{0, 6, 11\}$ & $\{3, 11, 13\}$\\
$\{0, 8, 12\}$ & $\{0, 9, 13\}$ & $\{0, 10, 11\}$ & $\{0, 10, 12\}$ & $\{0, 10, 13\}$ & $\{0, 11, 12\}$ & $\{0, 11, 13\}$ & $\{3, 11, 12\}$\\
$\{0, 12, 13\}$ & $\{1, 2, 3\}$ & $\{1, 2, 4\}$ & $\{1, 3, 4\}$ & $\{1, 3, 5\}$ & $\{1, 4, 5\}$ & $\{1, 5, 6\}$ & $\{3, 6, 12\}$\\
$\{1, 6, 10\}$ & $\{1, 6, 11\}$ & $\{1, 6, 12\}$ & $\{1, 9, 11\}$ & $\{1, 9, 12\}$ & $\{1, 9, 13\}$ & $\{1, 10, 11\}$ & $\{3, 6, 11\}$\\
$\{1, 10, 12\}$ & $\{1, 10, 13\}$ & $\{1, 11, 13\}$ & $\{2, 3, 5\}$ & $\{2, 3, 11\}$ & $\{2, 3, 13\}$ & $\{2, 4, 10\}$ & $\{3, 4, 12\}$\\
$\{2, 4, 12\}$ & $\{2, 5, 10\}$ & $\{2, 5, 11\}$ & $\{2, 5, 13\}$ & $\{2, 6, 10\}$ & $\{2, 6, 11\}$ & $\{2, 6, 12\}$ & $\{2, 11, 13\}$\\
$\{2, 10, 12\}$ & $\{2, 10, 13\}$ & $\{2, 11, 12\}$\\
\end{longtable}}
\noindent We can obtain an NWBTS$(14; 151)$ by adding a block $\{1, 3, 8\}$ on the above NWBTS$(14; 150)$, where $\lambda_{0,7}=\lambda_{1,3}=\lambda_{3,8}=6$ and $\lambda_{2,9}=\lambda_{3,10}=\lambda_{4,11}=\lambda_{5,12}=\lambda_{6,13}=4$.

We construct an NWBTS$(14; 152)$ on $Z_{14}$,
where $\lambda_{0,1}=\lambda_{2,3}=\lambda_{4,5}=\lambda_{6,7}=6$ and $\lambda_{8,9}=\lambda_{10,11}=\lambda_{12,13}=4$.

{\small\begin{longtable}{lllllll}
$\{0, 1, 8\}$ & $\{0, 1, 9\}$ & $\{0, 1, 10\}$ & $\{0, 1, 11\}$ & $\{0, 1, 12\}$ & $\{0, 1, 13\}$ & $\{0, 2, 10\}$\\
$\{0, 2, 11\}$ & $\{0, 2, 12\}$ & $\{0, 2, 13\}$ & $\{0, 3, 10\}$ & $\{0, 3, 11\}$ & $\{0, 3, 12\}$ & $\{0, 3, 13\}$\\
$\{0, 4, 9\}$ & $\{0, 4, 10\}$ & $\{0, 4, 11\}$ & $\{0, 4, 12\}$ & $\{0, 4, 13\}$ & $\{0, 5, 6\}$ & $\{0, 5, 7\}$\\
$\{0, 5, 8\}$ & $\{0, 5, 12\}$ & $\{0, 5, 13\}$ & $\{0, 6, 7\}$ & $\{0, 6, 8\}$ & $\{0, 6, 9\}$ & $\{0, 6, 11\}$\\
$\{0, 7, 8\}$ & $\{0, 7, 9\}$ & $\{0, 7, 10\}$ & $\{0, 8, 9\}$ & $\{1, 2, 3\}$ & $\{1, 2, 4\}$ & $\{1, 2, 5\}$\\
$\{1, 2, 6\}$ & $\{1, 2, 7\}$ & $\{1, 3, 4\}$ & $\{1, 3, 5\}$ & $\{1, 3, 6\}$ & $\{1, 3, 7\}$ & $\{1, 4, 7\}$\\
$\{1, 4, 9\}$ & $\{ 1, 4, 13\}$ & $\{1, 5, 6\}$ & $\{1, 5, 7\}$ & $\{1, 5, 11\}$ & $\{1, 6, 11\}$ & $\{1, 6, 13\}$\\
$\{1, 7, 9\}$ & $\{1, 8, 10\}$ & $\{1, 8, 11\}$ & $\{1, 8, 12\}$ & $\{1, 8, 13\}$ & $\{1, 9, 10\}$ & $\{1, 9, 12\}$\\
$\{1, 10, 12\}$ & $\{1, 10, 13\}$ & $\{1, 11, 12\}$ & $\{2, 3, 4\}$ & $\{2, 3, 5\}$ & $\{2, 3, 6\}$ & $\{2, 3, 7\}$\\
$\{2, 4, 6\}$ & $\{2, 4, 7\}$ & $\{2, 4, 9\}$ & $\{2, 5, 9\}$ & $\{2, 5, 11\}$ & $\{2, 5, 13\}$ & $\{2, 6, 9\}$\\
$\{2, 6, 11\}$ & $\{2, 7, 8\}$ & $\{2, 7, 13\}$ & $\{2, 8, 10\}$ & $\{2, 8, 11\}$ & $\{2, 8, 12\}$ & $\{2, 8, 13\}$\\
$\{2, 9, 10\}$ & $\{2, 9, 12\}$ & $\{2, 10, 12\}$ & $\{2, 10, 13\}$ & $\{2, 11, 12\}$ & $\{3, 4, 6\}$ & $\{3, 4, 9\}$\\
$\{3, 4, 11\}$ & $\{3, 5, 9\}$ & $\{3, 5, 11\}$ & $\{3, 5, 13\}$ & $\{3, 6, 9\}$ & $\{3, 6, 13\}$ & $\{3, 7, 8\}$\\
$\{3, 7, 9\}$ & $\{3, 7, 12\}$ & $\{3, 8, 10\}$ & $\{3, 8, 11\}$ & $\{3, 8, 12\}$ & $\{3, 8, 13\}$ & $\{3, 9, 10\}$\\
$\{3, 10, 12\}$ & $\{3, 10, 13\}$ & $\{3, 11, 12\}$ & $\{4, 5, 8\}$ & $\{4, 5, 9\}$ & $\{4, 5, 10\}$ & $\{4, 5, 11\}$\\
$\{4, 5, 12\}$ & $\{4, 5, 13\}$ & $\{4, 6, 8\}$ & $\{4, 6, 10\}$ & $\{4, 6, 12\}$ & $\{4, 7, 10\}$ & $\{4, 7, 11\}$\\
$\{4, 7, 12\}$ & $\{4, 8, 11\}$ & $\{4, 8, 12\}$ & $\{4, 8, 13\}$ & $\{4, 10, 13\}$ & $\{5, 6, 8\}$ & $\{5, 6, 10\}$\\
$\{5, 6, 12\}$ & $\{5, 7, 8\}$ & $\{5, 7, 10\}$ & $\{5, 7, 12\}$ & $\{5, 8, 10\}$ & $\{5, 9, 11\}$ & $\{5, 9, 13\}$\\
$\{5, 10, 12\}$ & $\{6, 7, 8\}$ & $\{6, 7, 10\}$ & $\{6, 7, 11\}$ & $\{6, 7, 12\}$ & $\{6, 7, 13\}$ & $\{6, 8, 10\}$\\
$\{6, 9, 10\}$ & $\{6, 9, 12\}$ & $\{6, 11, 13\}$ & $\{6, 12, 13\}$ & $\{7, 9, 11\}$ & $\{7, 9, 13\}$ & $\{7, 10, 11\}$\\
$\{7, 11, 13\}$ & $\{7, 12, 13\}$ & $\{8, 9, 11\}$ & $\{8, 9, 12\}$ & $\{8, 9, 13\}$ & $\{9, 10, 11\}$ & $\{9, 11, 13\}$\\
$\{9, 12, 13\}$ & $\{10, 11, 12\}$ & $\{10, 11, 13\}$ & $\{0, 2, 3\}$ & $\{11, 12, 13\}$\\
\end{longtable}}
\noindent We can get an NWBTS$(14; 153)$ by adding a block $\{8, 9, 10\}$ on the above NWBTS$(14; 152)$, where $\lambda_{0,1}=\lambda_{2,3}=\lambda_{4,5}=\lambda_{6,7}=\lambda_{8,10}=\lambda_{9,10}=6$ and $\lambda_{10,11}=\lambda_{12,13}=4$.
\qed

\begin{lemma}\label{10}
There exists an NWBTS$(10; b)$, where $b\equiv14,15,16\pmod{30}$.
\end{lemma}
\proof  We only need to consider that $b\in[14,16]\cup[44,46]$ by Lemma \ref{bb}.
We first construct an NWBTS$(10; 14)$ on $Z_{10}$, where $\lambda_{0,1}=2$ and $\lambda_{2,8}=\lambda_{3,6}=\lambda_{4,9}=\lambda_{5,7}=0$.

{\small\begin{longtable}{lllllllll}
$\{0, 1, 2\}$ & $\{0, 1, 3\}$ & $\{0, 4, 5\}$ & $\{0, 6, 7\}$ & $\{0, 8, 9\}$ & $\{1, 4, 7\}$ & $\{1, 5, 8\}$\\
$\{1, 6, 9\}$ & $\{2, 3, 4\}$ & $\{2, 5, 6\}$ & $\{2, 7, 9\}$ & $\{3, 5, 9\}$ & $\{3, 7, 8\}$ & $\{4, 6, 8\}$\\
\end{longtable}}

\noindent We can get an NWBTS$(10; 15)$ by adding a block $\{2,8,9\}$ on the above NWBTS$(10; 14)$,  where $\lambda_{0,1}=\lambda_{2,9}=\lambda_{8,9}=2$ and $\lambda_{3,6}=\lambda_{4,9}=\lambda_{5,7}=0$.

We construct an NWBTS$(10; 16)$ on $Z_{10}$, where $\lambda_{4,8}=0$ and $\lambda_{0,1}=\lambda_{2,3}=\lambda_{5,7}=\lambda_{6,9}=2$.

{\small\begin{longtable}{lllllllll}
$\{0, 1, 2\}$ & $\{0, 1, 3\}$ & $\{0, 4, 5\}$ & $\{0, 6, 7\}$ & $\{0, 8, 9\}$ & $\{1, 4, 7\}$ & $\{3, 5, 7\}$ & $\{3, 6, 8\}$\\
$\{1, 5, 8\}$ & $\{1, 6, 9\}$ & $\{2, 3, 4\}$ & $\{2, 3, 9\}$ & $\{2, 5, 6\}$ & $\{2, 7, 8\}$ & $\{4, 6, 9\}$ & $\{5, 7, 9\}$\\
\\
\end{longtable}}

We get a simple S$_{2}(2,3,10)$, whose blocks are different from all blocks of the above NWBTS$(10;$ $r)(r\in[14,16])$.

{\small\begin{longtable}{llllllllll}
$\{0, 2, 9\}$ & $\{1, 2, 9\}$ & $\{0, 5, 9\}$ & $\{1, 5, 9\}$ & $\{3, 4, 9\}$ & $\{4, 7, 9\}$ & $\{3, 6, 9\}$ & $\{6, 8, 9\}$ & $\{7, 8, 9\}$ &$\{0, 1, 4\}$\\
$\{0, 1, 5\}$ & $\{0, 2, 6\}$ & $\{0, 3, 6\}$ & $\{0, 3, 7\}$ & $\{0, 4, 8\}$ & $\{0, 7, 8\}$ & $\{1, 2, 7\}$ & $\{1, 6, 7\}$ & $\{1, 3, 4\}$ & $\{1, 3, 8\}$\\
$\{1, 6, 8\}$ & $\{2, 3, 7\}$ & $\{2, 4, 6\}$ & $\{2, 3, 5\}$ & $\{3, 5, 8\}$ & $\{2, 4, 8\}$ & $\{2, 5, 8\}$ & $\{4, 5, 6\}$ & $\{4, 5, 7\}$ & $\{5, 6, 7\}$\\
\end{longtable}}
\noindent Then we can obtain an NWBTS$(10; b)$ $(b \in[44,46])$ by combining all blocks of the S$_{2}(2,3,10)$ and the NWBTS$(10$; $r)$ $(r\in[14,16])$, respectively.\qed

\begin{lemma}\label{ddf}
Let $X$ be a set of $16$ points, $Y$ be a $4$-subset of $X$, and let $b\in[38,42]$. Then there exist two disjoint  designs defined over $X$, an S$_{6}(2,3,16)$ and an NWBTS$(16,b)$, whose blocks all belong to $\binom{X}{3}\backslash \binom{Y}{3}$.
\end{lemma}
\proof Let $X=Z_{16}$ and $Y=Z_{4}$. The block set $\mathcal{A}$ of an S$_{6}(2,3,16)$ is generated by the following base blocks modulo 16 from the triples of $\binom{X}{3}\backslash \binom{Y}{3}$.
{\small\begin{longtable}{llllllll}
$\{0, 2, 8\}$& $\{0, 2, 10\}$& $\{0, 3, 11\}$& $\{0, 2, 6\}$& $\{0, 2, 7\}$& $\{0, 2, 11\}$& $\{0, 2, 13\}$& $\{0, 1, 4\}$\\
$\{0, 1, 13\}$& $\{0, 3, 10\}$& $\{0, 3, 12\}$& $\{0, 1, 5\}$& $\{0, 1, 12\}$& $\{0, 1, 7\}$& $\{0, 1, 10\}$\\
\end{longtable}}

Next we construct an NWBTS$(16;38)$ with the block set $\mathcal{B}_{1}$ contained in $\binom{X}{3}\backslash \big(\binom{Y}{3}\cup \mathcal{A}\big)$, where $\lambda_{14,15}=2$ and $\lambda_{0,1}=\lambda_{2,3}=\lambda_{4,5}=\lambda_{6,7}=\lambda_{8,9}=\lambda_{10,11}=\lambda_{12,13}=0$.
{\small\begin{longtable}{llllllll}
$\{0, 2, 4\}$ &$\{0, 3, 9\}$ &$\{0, 7, 8\}$ &$\{0, 5, 11\}$ &$\{0, 6, 12\}$ &$\{0, 10, 15\}$ &$\{0, 13, 14\}$ &$\{3, 10, 13\}$\\
$\{5, 9, 15\}$ &$\{2, 5, 10\}$ &$\{3, 5, 7\}$ &$\{1, 3, 6\}$ &$\{6, 8, 10\}$ &$\{1, 4, 10\}$ &$\{4, 6, 9\}$ &$\{5, 6, 13\}$\\
$\{1, 5, 14\}$ &$\{5, 8, 12\}$ &$\{3, 8, 15\}$ &$\{1, 8, 11\}$ &$\{1, 2, 15\}$ &$\{11, 13, 15\}$ &$\{2, 6, 11\}$ &$\{6, 14, 15\}$\\
$\{7, 14, 15\}$ &$\{4, 12, 15\}$ &$\{3, 11, 12\}$ &$\{3, 4, 14\}$ &$\{4, 8, 13\}$ &$\{2, 8, 14\}$ &$\{4, 7, 11\}$ &$\{7, 9, 10\}$\\
$\{9, 11, 14\}$ &$\{10, 12, 14\}$ &$\{1, 7, 13\}$ &$\{1, 9, 12\}$ &$\{2, 7, 12\}$ &$\{2, 9, 13\}$\\
\end{longtable}}

\noindent We can get an NWBTS$(16; 39)$ with the block set $\mathcal{B}_{1}\cup\{\{2, 3, 4\}\}$ contained in $\binom{X}{3}\backslash \big(\binom{Y}{3}\cup \mathcal{A}\big)$, where $\lambda_{0,1}=\lambda_{2,4}=\lambda_{3,4}=2$ and $\lambda_{4,5}=\lambda_{6,7}=\lambda_{8,9}=\lambda_{10,11}=\lambda_{12,13}=\lambda_{14,15}=0$.

We construct an NWBTS$(16; 40)$ with the block set $\mathcal{B}_{2}$ contained in $\binom{X}{3}\backslash \big(\binom{Y}{3}\cup \mathcal{A}\big)$, where $\lambda_{0,1}=\lambda_{2,3}=\lambda_{4,5}=\lambda_{6,7}=2$ and $\lambda_{8,9}=\lambda_{10,11}=\lambda_{12,13}=\lambda_{14,15}=0$.
{\small\begin{longtable}{llllllll}
$\{0, 2, 4\}$& $\{2, 3, 5\}$& $\{2, 3, 8\}$& $\{1, 3, 4\}$& $\{0, 1, 15\}$& $\{2, 13, 15\}$& $\{0, 1, 8\}$& $\{8, 10, 12\}$\\
$\{0, 3, 9\}$& $\{0, 12, 14\}$& $\{0, 7, 11\}$& $\{0, 5, 10\}$& $\{0, 6, 13\}$& $\{2, 7, 12\}$& $\{1, 2, 9\}$& $\{1, 10, 13\}$\\
$\{1, 7, 14\}$& $\{1, 5, 11\}$& $\{1, 6, 12\}$& $\{2, 10, 14\}$& $\{2, 6, 11\}$& $\{3, 6, 10\}$& $\{3, 7, 13\}$& $\{3, 11, 14\}$\\
$\{3, 12, 15\}$& $\{4, 5, 12\}$& $\{8, 11, 15\}$& $\{4, 11, 13\}$& $\{9, 11, 12\}$& $\{5, 9, 13\}$& $\{4, 5, 15\}$& $\{8, 13, 14\}$\\
$\{4, 6, 8\}$& $\{4, 7, 10\}$& $\{4, 9, 14\}$& $\{5, 7, 8\}$& $\{5, 6, 14\}$& $\{6, 7, 9\}$& $\{6, 7, 15\}$& $\{9, 10, 15\}$\\
\end{longtable}}

\noindent We can get an NWBTS$(16; 41)$ with the block set $\mathcal{B}_{2}\cup\{\{10, 11, 12\}\}$  contained in $\binom{X}{3}\backslash \big(\binom{Y}{3}\cup \mathcal{A}\big)$, where $\lambda_{0,1}=\lambda_{2,3}=\lambda_{4,5}=\lambda_{6,7}=\lambda_{8,10}=\lambda_{9,10}=2$ and $\lambda_{10,11}=\lambda_{12,13}=\lambda_{14,15}=0$.

We construct an NWBTS$(16; 42)$ from the triples of $\binom{X}{3}\backslash \big(\binom{Y}{3}\cup \mathcal{A}\big)$, where $\lambda_{0,1}=\lambda_{2,3}=\lambda_{4,5}=\lambda_{6,7}=\lambda_{8,9}=\lambda_{10,11}=\lambda_{12,13}=2$ and $\lambda_{14,15}=0$.
{\small\begin{longtable}{lllllll}
$\{0, 2, 4\}$& $\{2, 3, 5\}$& $\{2, 3, 8\}$ & $\{1, 3, 4\}$ & $\{0, 1, 15\}$& $\{1, 2, 7\}$& $\{2, 6, 15\}$\\
$\{2, 12, 13\}$& $\{12, 13, 15\}$& $\{2, 9, 11\}$ & $\{2, 10, 14\}$ & $\{4, 9, 14\}$& $\{4, 5, 10\}$& $\{4, 8, 15\}$\\
$\{5, 9, 15\}$& $\{3, 10, 15\}$& $\{7, 11, 15\}$ & $\{4, 11, 13\}$ & $\{4, 6, 7\}$& $\{4, 5, 12\}$& $\{0, 10, 11\}$\\
$\{10, 11, 12\}$& $\{0, 5, 13\}$& $\{1, 5, 11\}$ & $\{3, 11, 14\}$ & $\{6, 8, 11\}$& $\{1, 6, 10\}$& $\{0, 3, 6\}$\\
$\{3, 9, 13\}$& $\{1, 8, 13\}$& $\{3, 7, 12\}$ & $\{6, 9, 12\}$ & $\{0, 1, 9\}$& $\{1, 12, 14\}$& $\{0, 8, 12\}$\\
$\{0, 7, 14\}$& $\{5, 6, 7\}$& $\{5, 8, 14\}$ & $\{6, 13, 14\}$ & $\{7, 10, 13\}$& $\{7, 8, 9\}$& $\{8, 9, 10\}$\\
\end{longtable}}\qed

\begin{lemma}\label{N16}
There exists an NWBTS$(16 ; b)$ for $b\equiv38,39,40,41,42\pmod{80}$.
\end{lemma}

\proof We only need to consider that $b\in [38,42]\cup[118,122]\cup[198,202]$ by Lemma \ref{bb}.
We first construct two disjoint simple S$_{2}(2,3,16)$s. Their blocks are generated by two sets of base blocks (each set in a line) under the action of the group $Z_{16}$, which are different from all blocks of the above NWBTS$(16;b)$ $(b\in [38,42])$ in Lemma \ref{ddf}.

{\small\begin{longtable}{llllllllll}
$\{0, 3, 11\}$ & $\{0, 2, 11\}$ & $\{0, 2, 12\}$ & $\{0, 1, 10\}$ & $\{0, 1, 13\}$\\
$\{0, 2, 10\}$ & $\{0, 2, 13\}$ & $\{0, 3, 12\}$ & $\{0, 1, 12\}$ & $\{0, 1, 7\}$\\
\end{longtable}}
\noindent We can obtain an NWBTS$(16; b)$ $(b \in [118,122])$ by combining all blocks of the S$_{2}(2,3,16)$ and the above NWBTS$(16; r)$ $(r\in[38,42])$, respectively. We can get an NWBTS$(16; b)$ $(b \in [198,202])$ by combining all blocks of the two disjoint S$_{2}(2,3,16)$s and the above NWBTS$(16; r)$ $(r\in[38,42])$, respectively.
\qed

\section{$3^{+}$-PCS$^{u}_{2}(12^{3} : 4)$, $3$-PCS$^{0}_{2}(6^{4} : 2)$, $1$-PCS$^{3}_{6}(6^{4} : 2)$}
\begin{lemma}\label{H12340}
There is a $3^{+}$-PCS$^{u}_{2}(12^{3} : 4)$ for $u = 0, 6, 12$.
\end{lemma}

\proof The design can be constructed on $X = (Z_{12} \times Z_{3}) \cup S$ with stem $S = \{a, b, c, d\}$ and groups $G_{j} = \{(0, j), (1, j), \ldots, (11, j)\},~j = 0, 1, 2$. We first construct a GDD$_{2}(2, 3, 40)$ of type $1^{24}16^{1}$. The blocks are generated from the base blocks in the following by adding $3$ (mod 12) to the first coordinate; the underlined base blocks and their developments form the blocks of an NGDD$(2, 3, 40)$ of type $2^{(0,12)}16^{1}$ where $\lambda_{(0,1),(6,1)}=\lambda_{(1,1),(7,1)}=\lambda_{(2,1),(8,1)}=\lambda_{(3,1),(9,1)}=\lambda_{(4,1),(10,1)}=\lambda_{(5,1),(11,1)}= \lambda_{(0,2),(6,2)}=\lambda_{(1,2),(7,2)}=\lambda_{(2,2),(8,2)}=\lambda_{(3,2),(9,2)}=\lambda_{(4,2),(10,2)}=\lambda_{(5,2),(11,2)}=0$.
{\small\begin{longtable}{llll}
$\underline{\{a, (0, 1), (6, 2)\}}$ & $\underline{\{a, (1, 1), (7, 2)\}}$ & $\underline{\{a, (2, 1), (8, 2)\}}$ & $\underline{\{b, (0, 1), (7, 2)\}}$\\
$\underline{\{b, (1, 1), (8, 2)\}}$ & $\underline{\{b, (2, 1), (9, 2)\}}$ & $\underline{\{c, (0, 1), (8, 2)\}}$ & $\underline{\{c, (1, 1), (9, 2)\}}$\\
$\underline{\{c, (2, 1), (10, 2)\}}$ & $\underline{\{d, (0, 1), (9, 2)\}}$ & $\underline{\{d, (1, 1), (10, 2)\}}$ & $\underline{\{d, (2, 1), (11, 2)\}}$\\
$\underline{\{(6, 0), (2, 1), (3, 1)\}}$ & $\underline{\{(2, 1), (6, 1), (5, 2)\}}$ & $\underline{\{(3, 0), (2, 1), (4, 1)\}}$ & $\underline{\{(9, 0), (0, 1), (2, 1)\}}$\\
$\underline{\{(9, 0), (0, 2), (1, 2)\}}$ & $\underline{\{(6, 0), (1, 1), (6, 1)\}}$ & $\underline{\{(6, 0), (0, 2), (4, 2)\}}$ & $\underline{\{(3, 0), (1, 1), (5, 1)\}}$\\
$\underline{\{(0, 0), (1, 2), (2, 2)\}}$ & $\underline{\{(3, 0), (0, 2), (2, 2)\}}$ & $\underline{\{(6, 0), (2, 2), (6, 2)\}}$ & $\underline{\{(9, 0), (2, 2), (4, 2)\}}$\\
$\underline{\{(9, 0), (1, 1), (3, 1)\}}$ & $\underline{\{(10, 0), (2, 1), (2, 2)\}}$ & $\underline{\{(1, 1), (2, 1), (1, 2)\}}$ & $\underline{\{(2, 0), (11, 1), (1, 2)\}}$\\
$\underline{\{(10, 1), (1, 2), (3, 2)\}}$ & $\underline{\{(2, 0), (1, 1), (3, 2)\}}$ & $\underline{\{(1, 0), (2, 1), (7, 1)\}}$ & $\underline{\{(5, 0), (0, 1), (5, 1)\}}$,\\
$\underline{\{(4, 0), (2, 1), (0, 2)\}}$ & $\underline{\{(1, 0), (8, 1), (1, 2)\}}$ & $\underline{\{(11, 0), (2, 1), (5, 1)\}}$ & $\underline{\{(11, 1), (0, 2), (3, 2)\}}$\\
$\underline{\{(1, 0), (0, 1), (0, 2)\}}$ & $\underline{\{(10, 0), (0, 1), (1, 1)\}}$ & $\underline{\{(7, 0), (0, 2), (5, 2)\}}$ & $\underline{\{(8, 0), (1, 2), (6, 2)\}}$\\
$\underline{\{(5, 0), (1, 1), (0, 2)\}}$ & $\underline{\{(1, 0), (2, 2), (3, 2)\}}$ & $\underline{\{(4, 0), (0, 1), (4, 1)\}}$ & $\underline{\{(7, 0), (1, 2), (4, 2)\}}$\\
$\underline{\{(1, 0), (6, 1), (4, 2)\}}$ & $\underline{\{(7, 0), (4, 1), (2, 2)\}}$ & $\underline{\{(2, 0), (3, 1), (6, 2)\}}$ & $\underline{\{(8, 0), (0, 1), (2, 2)\}}$\\
$\underline{\{(5, 0), (2, 2), (7, 2)\}}$ & $\underline{\{(0, 1), (3, 1), (4, 2)\}}$ & $\underline{\{(2, 0), (0, 1), (5, 2)\}}$ & $\underline{\{(1, 1), (2, 2), (5, 2)\}}$\\
$\underline{\{(11, 0), (1, 1), (4, 1)\}}$ & $\underline{\{(5, 0), (1, 2), (5, 2)\}}$ & $\{(11, 0), (1, 1), (2, 1)\}$ & $\{(5, 0), (2, 1), (4, 1)\}$\\
$\{a, (0, 1), (7, 2)\}$ & $\{a, (1, 1), (8, 2)\}$ & $\{a, (2, 1), (9, 2)\}$ & $\{c, (0, 1), (6, 2)\}$\\
$\{c, (1, 1), (7, 2)\}$ & $\{c, (2, 1), (8, 2)\}$ & $\{d, (0, 1), (8, 2)\}$ & $\{d, (2, 1), (10, 2)\}$\\
$\{d, (1, 1), (9, 2)\}$ & $\{b, (0, 1), (9, 2)\}$ & $\{b, (1, 1), (10, 2)\}$ & $\{b, (2, 1), (11, 2)\}$,\\
$\{(2, 0), (0, 1), (6, 1)\}$ & $\{(3, 0), (1, 1), (7, 1)\}$ & $\{(4, 0), (2, 1), (8, 1)\}$ & $\{(0, 0), (0, 2), (6, 2)\}$\\
$\{(1, 0), (1, 2), (7, 2)\}$ & $\{(2, 0), (2, 2), (8, 2)\}$ & $\{(0, 1), (3, 1), (1, 2)\}$ & $\{(2, 1), (6, 1), (6, 2)\}$\\
$\{(2, 1), (5, 1), (3, 2)\}$ & $\{(7, 0), (2, 2), (3, 2)\}$ & $\{(2, 0), (0, 2), (5, 2)\}$ & $\{(5, 0), (1, 1), (5, 1)\}$\\
$\{(3, 0), (0, 2), (4, 2)\}$ & $\{(3, 0), (1, 2), (6, 2)\}$ & $\{(5, 0), (0, 2), (1, 2)\}$ & $\{(6, 0), (2, 1), (1, 2)\}$\\
$\{(4, 1), (2, 2), (7, 2)\}$ & $\{(5, 0), (2, 2), (6, 2)\}$ & $\{(1, 1), (1, 2), (3, 2)\}$ & $\{(10, 0), (1, 2), (2, 2)\}$\\
$\{(4, 0), (1, 2), (5, 2)\}$ & $\{(9, 1), (0, 2), (2, 2)\}$ & $\{(1, 0), (7, 1), (0, 2)\}$ & $\{(10, 0), (0, 2), (3, 2)\}$\\
$\{(2, 0), (7, 1), (6, 2)\}$ & $\{(1, 1), (4, 1), (5, 2)\}$ & $\{(0, 0), (2, 1), (7, 1)\}$ & $\{(3, 0), (2, 1), (7, 2)\}$\\
$\{(3, 0), (0, 1), (2, 2)\}$ & $\{(9, 0), (2, 2), (5, 2)\}$ & $\{(0, 0), (3, 1), (2, 2)\}$ & $\{(2, 1), (2, 2), (4, 2)\}$\\
$\{(11, 0), (1, 2), (4, 2)\}$ & $\{(4, 0), (11, 1), (2, 2)\}$ & $\{(1, 0), (1, 1), (3, 1)\}$ & $\{(7, 0), (0, 1), (4, 1)\}$\\
$\{(0, 0), (0, 1), (1, 1)\}$ & $\{(9, 0), (2, 1), (3, 1)\}$ & $\{(11, 0), (0, 1), (5, 1)\}$ & $\{(1, 0), (0, 1), (2, 1)\}$\\
$\{(10, 0), (1, 1), (6, 1)\}$ & $\{(2, 0), (9, 1), (1, 2)\}$\\
\end{longtable}}

We construct a GDD$_{2}(2, 3, 40)$ of type $1^{24}16^{1}$. The blocks are generated from the base blocks in the following by adding $3$ (mod 12) to the first coordinate; the underlined base blocks and their developments form the blocks of an NGDD$(2, 3, 40)$ of type $2^{(6,6)}16^{1}$ where $\lambda_{(0,1),(6,1)}=\lambda_{(1,1),(7,1)}=\lambda_{(2,1),(8,1)}=\lambda_{(3,1),(9,1)}=\lambda_{(4,1),(10,1)}=\lambda_{(5,1),(11,1)}=2$ and
$\lambda_{(0,2),(6,2)}=\lambda_{(1,2),(7,2)}=\lambda_{(2,2),(8,2)}=\lambda_{(3,2),(9,2)}=\lambda_{(4,2),(10,2)}=\lambda_{(5,2),(11,2)}=0$.
{\small\begin{longtable}{llll}
$\underline{\{c, (2, 1), (8, 2)\}}$ & $\underline{\{a, (0, 1), (7, 2)\}}$ & $\underline{\{a, (2, 1), (9, 2)\}}$ & $\underline{\{a, (1, 1), (8, 2)\}}$\\
$\underline{\{b, (0, 1), (8, 2)\}}$ & $\underline{\{b, (1, 1), (9, 2)\}}$ & $\underline{\{b, (2, 1), (10, 2)\}}$ & $\underline{\{c, (0, 1), (9, 2)\}}$\\
$\underline{\{d, (0, 1), (10, 2)\}}$ & $\underline{\{d, (1, 1), (11, 2)\}}$ & $\underline{\{d, (2, 1), (0, 2)\}}$ & $\underline{\{c, (1, 1), (7, 2)\}}$\\
$\underline{\{(0, 1), (6, 1), (0, 2)\}}$ & $\underline{[\{(1, 1), (7, 1), (4, 2)\}}$ & $\underline{\{(2, 1), (8, 1), (5, 2)\}}$ & $\underline{\{(2, 0), (9, 1), (0, 2)\}}$\\
$\underline{\{(6, 0), (2, 1), (3, 1)\}}$ & $\underline{\{(2, 1), (6, 1), (7, 2)\}}$ & $\underline{\{(3, 0), (2, 1), (4, 1)\}}$ & $\underline{\{(6, 0), (0, 2), (4, 2)\}}$\\
$\underline{\{(3, 0), (0, 2), (2, 2)\}}$ & $\underline{\{(6, 0), (0, 1), (1, 1)\}}$ & $\underline{\{(9, 0), (1, 1), (1, 2)\}}$ & $\underline{\{(0, 0), (0, 1), (2, 1)\}}$\\
$\underline{\{(9, 0), (2, 1), (2, 2)\}}$ & $\underline{\{(3, 0), (1, 1), (6, 1)\}}$ & $\underline{\{(0, 0), (1, 2), (2, 2)\}}$ & $\underline{\{(6, 0), (1, 2), (6, 2)\}}$\\
$\underline{\{(9, 0), (0, 2), (5, 2)\}}$ & $\underline{\{(4, 0), (1, 2), (3, 2)\}}$ & $\underline{\{(5, 0), (0, 2), (1, 2)\}}$ & $\underline{\{(0, 1), (2, 2), (4, 2)\}}$\\
$\underline{\{(7, 0), (1, 2), (5, 2)\}}$ & $\underline{\{(7, 0), (2, 2), (7, 2)\}}$ & $\underline{\{(1, 0), (5, 1), (4, 2)\}}$ & $\underline{\{(2, 0), (1, 2), (4, 2)\}}$\\
$\underline{\{(11, 0), (2, 1), (4, 2)\}}$ & $\underline{\{(7, 0), (2, 1), (5, 1)\}}$ & $\underline{\{(11, 0), (0, 1), (5, 1)\}}$ & $\underline{\{(1, 0), (2, 1), (7, 1)\}}$\\
$\underline{\{(10, 0), (2, 2), (6, 2)\}}$ & $\underline{\{(4, 0), (1, 1), (3, 1)\}}$ & $\underline{\{(1, 0), (4, 1), (3, 2)\}}$ & $\underline{\{(2, 0), (2, 2), (3, 2)\}}$\\
$\underline{\{(8, 0), (2, 2), (5, 2)\}}$ & $\underline{\{(1, 1), (4, 1), (5, 2)\}}$ & $\underline{\{(8, 0), (1, 1), (5, 1)\}}$ & $\underline{\{(11, 1), (0, 2), (3, 2)\}}$\\
$\underline{\{(2, 0), (1, 1), (2, 1)\}}$ & $\underline{\{(2, 0), (4, 1), (6, 2)\}}$ & $\underline{\{(8, 0), (0, 1), (4, 1)\}}$ & $\underline{\{(2, 0), (0, 1), (5, 2)\}}$\\
$\underline{\{(1, 0), (1, 1), (6, 2)\}}$ & $\underline{\{(7, 0), (0, 1), (3, 1)\}}$ & $\underline{\{(1, 0), (3, 1), (2, 2)\}}$ & $\{(8, 0), (1, 2), (6, 2)\}$\\
$\{a, (0, 1), (6, 2)\}$ & $\{a, (2, 1), (8, 2)\}$ & $\{a, (1, 1), (7, 2)\}$ & $\{b, (0, 1), (7, 2)\}$\\
$\{b, (2, 1), (9, 2)\}$ & $\{b, (1, 1), (8, 2)\}$ & $\{c, (0, 1), (8, 2)\}$ & $\{d, (0, 1), (9, 2)\}$\\
$\{d, (2, 1), (11, 2)\}$ & $\{d, (1, 1), (10, 2)\}$ & $\{c, (1, 1), (9, 2)\}$ & $\{c, (2, 1), (10, 2)\}$\\
$\{(0, 0), (0, 2), (6, 2)\}$ & $\{(1, 0), (1, 2), (7, 2)\}$ & $\{(11, 0), (1, 1), (1, 2)\}$ & $\{(0, 1), (4, 1), (3, 2)\}$\\
$\{(2, 0), (2, 2), (8, 2)\}$ & $\{(9, 0), (0, 2), (1, 2)\}$ & $\{(6, 0), (1, 2), (2, 2)\}$ & $\{(3, 0), (2, 2), (5, 2)\}$\\
$\{(6, 0), (1, 1), (4, 1)\}$ & $\{(0, 0), (0, 1), (1, 1)\}$ & $\{(9, 0), (1, 1), (3, 1)\}$ & $\{(9, 0), (2, 2), (7, 2)\}$\\
$\{(0, 0), (2, 1), (3, 1)\}$ & $\{(9, 0), (2, 1), (5, 1)\}$ & $\{(3, 0), (0, 1), (2, 1)\}$ & $\{(3, 0), (0, 2), (4, 2)\}$\\
$\{(11, 0), (0, 1), (3, 1)\}$ & $\{(11, 0), (2, 2), (3, 2)\}$ & $\{(5, 0), (2, 2), (6, 2)\}$ & $\{(4, 0), (1, 2), (5, 2)\}$\\
$\{(10, 1), (0, 2), (2, 2)\}$ & $\{(1, 0), (3, 1), (5, 2)\}$ & $\{(4, 0), (3, 1), (2, 2)\}$ & $\{(10, 0), (2, 1), (6, 1)\}$\\
$\{(1, 0), (2, 1), (4, 1)\}$ & $\{(1, 0), (1, 1), (6, 1)\}$ & $\{(7, 0), (0, 2), (3, 2)\}$ & $\{(4, 0), (2, 1), (7, 2)\}$\\
$\{(4, 0), (1, 1), (6, 2)\}$ & $\{(2, 1), (7, 1), (5, 2)\}$ & $\{(0, 1), (0, 2), (5, 2)\}$ & $\{(1, 0), (8, 1), (0, 2)\}$\\
$\{(7, 0), (1, 1), (2, 2)\}$ & $\{(2, 1), (2, 2), (4, 2)\}$ & $\{(8, 0), (1, 1), (2, 1)\}$ & $\{(2, 1), (1, 2), (3, 2)\}$\\
$\{(5, 0), (2, 1), (0, 2)\}$ & $\{(5, 0), (0, 1), (5, 1)\}$ & $\{(2, 0), (10, 1), (1, 2)\}$ & $\{(0, 1), (1, 2), (4, 2)\}$\\
$\{(5, 0), (3, 1), (1, 2)\}$ & $\{(2, 0), (1, 1), (5, 1)\}$\\
\end{longtable}}

We construct an initial GDD$_{2}(2, 3, 40)$ of type $1^{24}16^{1}$. The blocks are generated from the base blocks in the following by adding $3$ (mod 12) to the first coordinate.
{\small\begin{longtable}{llll}
$\{a, (0, 1), (0, 2)\}$ & $\{a, (1, 1), (2, 2)\}$ & $\{a, (2, 1), (7, 2)\}$ & $\{b, (1, 1), (1, 2)\}$\\
$\{b, (2, 1), (3, 2)\}$ & $\{b, (0, 1), (2, 2)\}$ & $\{c, (0, 1), (1, 2)\}$ & $\{c, (1, 1), (3, 2)\}$\\
$\{c, (2, 1), (5, 2)\}$ & $\{d, (2, 1), (2, 2)\}$ & $\{d, (0, 1), (4, 2)\}$ & $\{d, (1, 1), (6, 2)\}$\\
$\{a, (0, 1), (2, 2)\}$ & $\{a, (1, 1), (4, 2)\}$ & $\{a, (2, 1), (6, 2)\}$ & $\{b, (0, 1), (3, 2)\}$\\
$\{b, (1, 1), (5, 2)\}$ & $\{b, (2, 1), (7, 2)\}$ & $\{c, (0, 1), (4, 2)\}$ & $\{c, (1, 1), (6, 2)\}$\\
$\{c, (2, 1), (2, 2)\}$ & $\{d, (2, 1), (5, 2)\}$ & $\{d, (1, 1), (3, 2)\}$ & $\{d, (0, 1), (1, 2)\}$\\
$\{(2, 0), (0, 2), (1, 2)\}$ & $\{(3, 0), (1, 2), (3, 2)\}$ & $\{(7, 0), (2, 2), (5, 2)\}$ & $\{(1, 1), (2, 1), (4, 2)\}$\\
$\{(2, 0), (2, 2), (4, 2)\}$ & $\{(1, 0), (1, 2), (5, 2)\}$ & $\{(6, 0), (1, 2), (7, 2)\}$ & $\{(2, 0), (2, 1), (6, 2)\}$\\
$\{(7, 0), (2, 1), (8, 1)\}$ & $\{(0, 0), (0, 1), (6, 1)\}$ & $\{(4, 0), (2, 1), (4, 1)\}$ & $\{(2, 0), (1, 1), (7, 1)\}$\\
$\{(1, 0), (0, 2), (6, 2)\}$ & $\{(6, 0), (2, 2), (8, 2)\}$ & $\{(3, 0), (2, 2), (6, 2)\}$ & $\{(6, 0), (2, 1), (0, 2)\}$\\
$\{(6, 1), (0, 2), (3, 2)\}$ & $\{(9, 0), (1, 2), (6, 2)\}$ & $\{(9, 0), (2, 1), (7, 1)\}$ & $\{(10, 0), (2, 1), (3, 1)\}$\\
$\{(3, 0), (2, 1), (5, 1)\}$ & $\{(0, 1), (5, 1), (11, 2)\}$ & $\{(10, 1), (2, 2), (7, 2)\}$ & $\{(4, 0), (1, 1), (6, 1)\}$\\
$\{(4, 0), (0, 2), (5, 2)\}$ & $\{(9, 0), (0, 1), (1, 1)\}$ & $\{(3, 0), (0, 1), (4, 1)\}$ & $\{(9, 0), (4, 1), (2, 2)\}$\\
$\{(1, 0), (0, 1), (3, 2)\}$ & $\{(4, 0), (7, 1), (1, 2)\}$ & $\{(7, 0), (1, 1), (3, 1)\}$ & $\{(10, 0), (1, 2), (4, 2)\}$\\
$\{(11, 0), (1, 1), (5, 1)\}$ & $\{(1, 1), (4, 1), (0, 2)\}$ & $\{(5, 0), (1, 1), (8, 2)\}$ & $\{(5, 1), (0, 2), (2, 2)\}$\\
$\{(11, 0), (0, 2), (4, 2)\}$ & $\{(11, 0), (2, 1), (6, 1)\}$ & $\{(0, 1), (2, 1), (10, 2)\}$ & $\{(5, 0), (2, 1), (1, 2)\}$\\
$\{(5, 0), (9, 1), (2, 2)\}$ & $\{(2, 0), (0, 1), (3, 1)\}$ & $\{(8, 0), (2, 2), (3, 2)\}$ & $\{(6, 1), (1, 2), (2, 2)\}$\\
$\{(0, 0), (3, 1), (8, 2)\}$ & $\{(9, 0), (2, 1), (4, 2)\}$ & $\{(4, 0), (0, 2), (1, 2)\}$ & $\{(5, 0), (1, 2), (3, 2)\}$\\
$\{(5, 0), (0, 2), (2, 2)\}$ & $\{(3, 0), (0, 2), (3, 2)\}$ & $\{(10, 0), (2, 2), (5, 2)\}$ & $\{(3, 0), (1, 2), (4, 2)\}$\\
$\{(4, 0), (2, 2), (6, 2)\}$ & $\{(10, 0), (1, 1), (1, 2)\}$ & $\{(5, 1), (0, 2), (4, 2)\}$ & $\{(11, 0), (1, 2), (5, 2)\}$\\
$\{(11, 0), (2, 2), (4, 2)\}$ & $\{(3, 0), (2, 2), (7, 2)\}$ & $\{(2, 0), (6, 1), (1, 2)\}$ & $\{(0, 1), (1, 1), (8, 2)\}$\\
$\{(3, 1), (1, 2), (2, 2)\}$ & $\{(2, 0), (4, 1), (2, 2)\}$ & $\{(1, 0), (1, 2), (6, 2)\}$ & $\{(7, 0), (2, 1), (8, 2)\}$\\
$\{(4, 1), (0, 2), (5, 2)\}$ & $\{(0, 0), (2, 2), (3, 2)\}$ & $\{(9, 0), (5, 1), (2, 2)\}$ & $\{(11, 0), (0, 1), (0, 2)\}$\\
$\{(2, 1), (4, 1), (10, 2)\}$ & $\{(7, 0), (4, 1), (1, 2)\}$ & $\{(3, 0), (0, 1), (3, 1)\}$ & $\{(6, 0), (1, 1), (0, 2)\}$\\
$\{(1, 0), (3, 1), (0, 2)\}$ & $\{(0, 1), (5, 1), (6, 2)\}$ & $\{(3, 0), (2, 1), (7, 1)\}$ & $\{(2, 0), (8, 1), (6, 2)\}$\\
$\{(8, 0), (1, 1), (3, 1)\}$ & $\{(5, 0), (1, 1), (4, 1)\}$ & $\{(2, 0), (2, 1), (5, 1)\}$ & $\{(5, 0), (2, 1), (3, 1)\}$\\
$\{(0, 0), (1, 1), (2, 1)\}$ & $\{(6, 0), (0, 1), (4, 1)\}$ & $\{(7, 0), (1, 1), (6, 1)\}$ & $\{(1, 0), (2, 1), (6, 1)\}$\\
$\{(1, 0), (1, 1), (5, 1)\}$ & $\{(4, 0), (0, 1), (2, 1)\}$\\
\end{longtable}}

Nine pairwise disjoint GDD$_{2}(2, 3, 40)$s result by developing over both coordinates. The nine GDD$_{2}(2, 3, 40)$s, the
above GDD$_{2}(2, 3, 40)$ of type $1^{24}16^{1}$ which has a sub-design NGDD$(2, 3$, $40)$ of type $2^{(0,12)}16^{1}$ $($or NGDD$(2, 3, 40)$ of type $2^{(6,6)}16^{1})$, and
the remaining blocks of a CS$(12^{3}:4)$, which are pairwise disjoint, form a $3^{+}$-PCS$^{u}_{2}(12^{3} : 4)$ for $u = 0, 12$ $($or $3^{+}$-PCS$^{6}_{2}(12^{3} : 4))$.\qed

\begin{lemma}\label{6420}
There is a $3$-PCS$^{0}_{2}(6^{4} : 2)$.
\end{lemma}

\proof We construct the design on $X = Z_{24} \cup \{a, b\}$ with stem $S = \{a, b\}$ and groups $G_{i} = \{i, i + 4, i + 8,\ldots,i + 20\}, i = 0, 1, 2, 3$. We first construct a GDD$_{2}(2, 3, 26)$ of type $1^{18}8^{1}$ with long group $G_{0}\cup S$, whose blocks are generated from the following blocks by adding 8 (mod 24); the underlined blocks and their development form the blocks of a GDD$(2,3,26)$ of type $2^{9}8^{1}$.
{\small\begin{longtable}{lllllll}
$\{\underline{a, 1, 2}\}$ & $\{\underline{a, 5, 7}\}$ & $\{\underline{a, 6, 11}\}$ & $\{\underline{b, 2, 3}\}$ & $\{\underline{b, 7, 13}\}$ & $\{\underline{b, 1, 6}\}$ & $\{\underline{0, 2, 22}\}$ \\
$\{\underline{3, 7, 12}\}$ & $\{\underline{3, 5, 9}\}$ & $\{\underline{4, 10, 14}\}$ & $\{\underline{1, 7, 23}\}$ & $\{\underline{0, 5, 6}\}$ & $\{\underline{6, 7, 22}\}$ & $\{\underline{2, 7, 16}\}$ \\
$\{\underline{0, 11, 23}\}$ & $\{\underline{1, 3, 8}\}$ & $\{\underline{0, 7, 14}\}$ & $\{\underline{0, 1, 13}\}$ & $\{\underline{2, 8, 17}\}$ & $\{\underline{5, 8, 11}\}$ & $\{\underline{7, 10, 17}\}$ \\
$\{\underline{1, 4, 17}\}$ & $\{\underline{3, 6, 17}\}$ & $\{\underline{4, 6, 9}\}$ & $\{\underline{6, 12, 19}\}$ & $\{\underline{2, 11, 19}\}$ & $\{\underline{3, 4, 13}\}$ & $\{\underline{4, 5, 18}\}$ \\
$\{\underline{2, 5, 10}\}$ & $\{\underline{2, 4, 15}\}$ & $\{\underline{4, 7, 21}\}$ & $\{\underline{6, 13, 21}\}$ & $\{a, 1, 7\}$ & $\{a, 2, 11\}$ & $\{a, 5, 22\}$ \\
$\{b, 1, 10\}$ & $\{b, 6, 19\}$ & $\{b, 5, 7\}$ & $\{0, 6, 18\}$ & $\{1, 19, 23\}$ & $\{3, 17, 21\}$ & $\{4, 6, 22\}$ \\
$\{4, 5, 10\}$ & $\{3, 5, 13\}$ & $\{4, 7, 13\}$ & $\{0, 5, 15\}$ & $\{2, 4, 19\}$ & $\{4, 15, 23\}$ & $\{1, 3, 15\}$ \\
$\{3, 4, 11\}$ & $\{0, 2, 3\}$ & $\{0, 7, 11\}$ & $\{6, 7, 11\}$ & $\{3, 6, 8\}$ & $\{7, 8, 22\}$ & $\{5, 8, 18\}$ \\
$\{1, 4, 21\}$ & $\{0, 9, 13\}$ & $\{0, 1, 17\}$ & $\{1, 2, 12\}$ & $\{6, 9, 21\}$ & $\{1, 14, 18\}$ & $\{2, 7, 14\}$ \\
$\{4, 9, 14\}$ & $\{2, 5, 6\}$ & $\{7, 10, 18\}$ \\
\end{longtable}}
\noindent Eight GDD$_{2}(2, 3, 26)$s are obtained by developing modulo 24.

Then, we construct another GDD$_{2}(2, 3, 26)$ of type $1^{18}8^{1}$ with long group $G_{0}\cup S$, whose blocks are generated from the following blocks by adding 4 (mod 24):
{\small\begin{longtable}{lllllll}
$\{a, 2, 5\}$ & $\{a, 3, 14\}$ & $\{a, 3, 13\}$ & $\{b, 2, 5\}$ & $\{1, 6, 18\}$ & $\{3, 5, 15\}$ & $\{1, 4, 13\}$ \\
$\{b, 3, 10\}$ & $\{b, 3, 13\}$ & $\{2, 6, 16\}$ & $\{2, 8, 18\}$ & $\{2, 4, 18\}$ & $\{3, 9, 17\}$ & $\{0, 3, 19\}$ \\
$\{1, 4, 19\}$ & $\{0, 1, 9\}$ & $\{2, 7, 15\}$ & $\{1, 2, 21\}$ & $\{1, 5, 14\}$ & $\{3, 7, 22\}$ & $\{0, 1, 18\}$ \\
$\{1, 2, 3\}$ & $\{2, 3, 6\}$ & $\{0, 3, 7\}$ & $\{3, 4, 6\}$ & $\{3, 4, 10\}$ & $\{0, 2, 11\}$ & $\{0, 6, 17\}$ \\
$\{1, 3, 20\}$ & $\{0, 5, 11\}$ & $\{3, 5, 12\}$ & $\{3, 8, 21\}$ & $\{2, 4, 17\}$ \\
\end{longtable}}
\noindent Four GDD$_{2}(2, 3, 26)$s are obtained by developing modulo 24. The GDDs so produced form a CS$(6^{4}:2)$.
\qed

\begin{lemma}\label{wp}
There is a $1$-PCS$^{3}_{6}(6^{4} : 2)$.
\end{lemma}

\proof The design can be constructed on $X = Z_{24} \cup \{a, b\}$ with stem $S = \{a, b\}$ and groups $G_{i} = \{i, i + 4, i + 8,\ldots,i + 20\}, i = 0, 1, 2, 3$. We first construct a GDD$_{2}(2, 3, 26)$ of type $1^{18}8^{1}$ with long group $G_{0}\cup S$, whose blocks are generated from the following blocks by adding 4 (mod 24):
{\small\begin{longtable}{lllllll}
$\{a, 3, 9\}$& $\{a, 1, 10\}$& $\{a, 3, 10\}$& $\{b, 2, 7\}$& $\{b, 2, 9\}$& $\{b, 3, 9\}$& $\{1, 3, 7\}$\\
$\{0, 2, 18\}$& $\{3, 6, 18\}$& $\{1, 3, 15\}$& $\{1, 2, 13\}$& $\{2, 4, 6\}$& $\{2, 4, 18\}$& $\{3, 5, 10\}$\\
$\{3, 5, 12\}$& $\{3, 4, 11\}$& $\{0, 1, 7\}$& $\{0, 3, 17\}$& $\{2, 6, 13\}$& $\{0, 1, 5 \}$&$\{2, 3, 8\}$ \\
$\{0, 3, 6\}$& $\{3, 8, 14\}$& $\{3, 4, 19\}$& $\{0, 5, 14\}$& $\{1, 5, 16\}$& $\{1, 2, 16\}$& $\{2, 5, 21\}$\\
$\{1, 4, 17\}$& $\{2, 3, 16\}$& $\{1, 4, 15\}$& $\{2, 5, 15\}$& $\{3, 7, 22\}$\\
\end{longtable}}
\noindent Four GDD$_{2}(2, 3, 26)$s $(\mathcal{A}_{i}$, $i=0,1,2,3)$ are obtained by developing modulo 24.

Next we construct block sets $\mathcal{A}_{4}, \mathcal{A}_{5}, \cdots, \mathcal{A}_{9}$. $\mathcal{A}_{4}$ and $\mathcal{A}_{5}$ are GDD$_{2}(2, 3, 26)$s of type $1^{18}8^{1}$ with long group $G_{0}\cup S$, the underlined base blocks in $\mathcal{A}_{4}$ form the blocks of an NGDD$(2, 3, 26)$ of type $2^{(3,6)}8^{1}$ where $\lambda_{2, 14}=\lambda_{3, 15}=\lambda_{5, 17}=2$ and  $\lambda_{1, 13}=\lambda_{6, 18}=\lambda_{7, 19}=\lambda_{9, 21}=\lambda_{10, 22}=\lambda_{11, 23}=0$. $\mathcal{A}_{6}$ and $\mathcal{A}_{7}$ are GDD$_{2}(2, 3, 26)$s of type $1^{18}8^{1}$ with long group $G_{1}\cup S$. $\mathcal{A}_{8}$ and $\mathcal{A}_{9}$ are GDD$_{4}(2, 3, 26)$s of type $1^{18}8^{1}$ with long group $G_{2}\cup S$ and $G_{3}\cup S$, respectively.

To sum up, $\mathcal{A}_{0}\cup\mathcal{A}_{4}\cup\mathcal{A}_{5}$, $\mathcal{A}_{1}\cup\mathcal{A}_{6}\cup\mathcal{A}_{7}$, $\mathcal{A}_{2}\cup\mathcal{A}_{8}$, and $\mathcal{A}_{3}\cup\mathcal{A}_{9}$ are GDD$_{6}(2, 3, 26)$s of type $1^{18}8^{1}$ with long group $G_{i}\cup S$ where $i=0,1,2,3$, respectively.
Then $\bigcup_{0\leq i \leq 9}\mathcal{A}_{i}$ forms a $1$-PCS$^{3}_{6}(6^{4} : 2)$.

{\small\begin{longtable}{llllllll}
$\mathcal{A}_{4}$: &$\{\underline{a, 1, 2}\}$& $\{\underline{a, 3, 5}\}$& $\{\underline{a, 6, 7}\}$& $\{\underline{a, 9, 10}\}$& $\{\underline{a, 11, 13}\}$& $\{\underline{a, 14, 15}\}$& $\{\underline{a, 17, 18}\}$\\
&$\{\underline{a, 22, 23}\}$& $\{\underline{b, 2, 3}\}$& $\{\underline{b, 5, 6}\}$& $\{\underline{b, 10, 11}\}$& $\{\underline{b, 13, 14}\}$& $\{\underline{b, 18, 19}\}$& $\{\underline{b, 15, 17}\}$\\
&$\{\underline{b, 7, 9}\}$& $\{\underline{b, 21, 22}\}$& $\{\underline{0, 1, 21}\}$& $\{\underline{1, 8, 9}\}$& $\{\underline{5, 8, 13}\}$& $\{\underline{1, 4, 5}\}$& $\{\underline{12, 13, 21}\}$\\
&$\{\underline{13, 16, 17}\}$& $\{\underline{1, 14, 17}\}$& $\{\underline{3, 6, 14}\}$& $\{\underline{3, 4, 7}\}$& $\{\underline{5, 7, 15}\}$& $\{\underline{0, 3, 11}\}$& $\{\underline{0, 10, 23}\}$\\
&$\{\underline{7, 8, 11}\}$& $\{\underline{7, 10, 18}\}$& $\{\underline{2, 15, 18}\}$& $\{\underline{6, 20, 22}\}$& $\{\underline{1, 6, 16}\}$& $\{\underline{1, 7, 12}\}$& $\{\underline{1, 10, 19}\}$\\
&$\{\underline{1, 3, 20}\}$& $\{\underline{1, 18, 22}\}$& $\{\underline{8, 10, 21}\}$& $\{\underline{6, 8, 23}\}$& $\{\underline{2, 6, 9}\}$& $\{\underline{6, 12, 15}\}$& $\{\underline{6, 10, 13}\}$\\
&$\{\underline{4, 6, 21}\}$& $\{\underline{0, 7, 14}\}$& $\{\underline{0, 9, 18}\}$& $\{\underline{0, 5, 19}\}$& $\{\underline{0, 2, 13}\}$& $\{\underline{0, 15, 22}\}$& $\{\underline{3, 8, 15}\}$\\
&$\{\underline{2, 8, 19}\}$& $\{\underline{8, 17, 22}\}$& $\{\underline{2, 12, 22}\}$& $\{\underline{7, 13, 22}\}$& $\{\underline{3, 19, 22}\}$& $\{\underline{12, 19, 23}\}$& $\{\underline{3, 10, 12}\}$\\
&$\{\underline{11, 12, 14}\}$& $\{\underline{3, 15, 21}\}$& $\{\underline{3, 17, 23}\}$& $\{\underline{3, 9, 13}\}$& $\{\underline{3, 16, 18}\}$& $\{\underline{7, 16, 23}\}$& $\{\underline{2, 4, 23}\}$\\
&$\{\underline{11, 18, 20}\}$& $\{\underline{9, 15, 20}\}$& $\{\underline{5, 9, 23}\}$& $\{\underline{13, 15, 23}\}$& $\{\underline{13, 19, 20}\}$& $\{\underline{4, 17, 19}\}$& $\{\underline{4, 10, 15}\}$\\
&$\{\underline{4, 14, 22}\}$& $\{\underline{4, 9, 11}\}$& $\{\underline{9, 16, 22}\}$& $\{\underline{5, 11, 22}\}$& $\{\underline{15, 16, 19}\}$& $\{\underline{18, 21, 23}\}$& $\{\underline{14, 20, 23}\}$\\
&$\{\underline{5, 16, 21}\}$& $\{\underline{2, 11, 16}\}$& $\{\underline{2, 5, 14}\}$& $\{\underline{2, 14, 21}\}$& $\{\underline{11, 17, 21}\}$& $\{\underline{2, 7, 17}\}$& $\{\underline{2, 10, 20}\}$\\
&$\{\underline{5, 17, 20}\}$& $\{\underline{7, 20, 21}\}$& $\{\underline{a, 19, 21}\}$ & $\{\underline{b, 1, 23}\}$ & $\{\underline{9, 12, 17}\}$& $\{\underline{6, 11, 19}\}$& $\{\underline{1, 11, 15}\}$\\
&$\{\underline{0, 6, 17}\}$& $\{\underline{8, 14, 18}\}$& $\{\underline{5, 12, 18}\}$& $\{\underline{4, 13, 18}\}$& $\{\underline{9, 14, 19}\}$& $\{\underline{10, 14, 16}\}$& $\{\underline{5, 10, 17}\}$\\
&$\{a, 2, 3\}$&  $\{a, 5, 6\}$&  $\{a, 7, 9\}$&  $\{a, 10, 11\}$&  $\{a, 13, 14\}$&  $\{a, 15, 17\}$&  $\{a, 1, 23\}$\\
&$\{a, 18, 19\}$&  $\{a, 21, 22\}$&  $\{b, 1, 2\}$&  $\{b, 6, 7\}$&  $\{b, 9, 10\}$&  $\{b, 14, 15\}$&  $\{b, 17, 18\}$ \\
&$\{b, 3, 5\}$&  $\{b, 19, 21\}$&  $\{b, 11, 13\}$&  $\{b, 22, 23\}$&  $\{0, 1, 9\}$&  $\{5, 8, 9\}$&  $\{4, 5, 13\}$\\
&$\{8, 13, 21\}$&  $\{9, 12, 13\}$&  $\{9, 16, 17\}$&  $\{1, 12, 17\}$&  $\{1, 6, 14\}$&  $\{3, 4, 14\}$&  $\{4, 7, 15\}$\\
&$\{3, 6, 11\}$&  $\{0, 3, 23\}$&  $\{2, 10, 23\}$&  $\{8, 11, 19\}$&  $\{6, 19, 22\}$&  $\{7, 8, 18\}$&  $\{3, 8, 22\}$ \\
&$\{1, 8, 10\}$&  $\{8, 14, 23\}$&  $\{2, 8, 15\}$&  $\{6, 8, 17\}$&  $\{2, 16, 18\}$&  $\{6, 10, 16\}$&  $\{2, 6, 20\}$\\
&$\{2, 13, 17\}$&  $\{4, 6, 9\}$&  $\{17, 20, 21\}$&  $\{3, 16, 19\}$&  $\{3, 7, 17\}$&  $\{3, 12, 18\}$&  $\{1, 3, 13\}$\\
&$\{3, 9, 21\}$&  $\{3, 10, 20\}$&  $\{6, 12, 23\}$&  $\{14, 17, 19\}$&  $\{0, 10, 14\}$&  $\{0, 17, 22\}$& $\{4, 10, 17\}$ \\
&$\{11, 17, 23\}$&  $\{5, 10, 18\}$&  $\{10, 12, 21\}$&  $\{7, 10, 19\}$&  $\{10, 15, 22\}$&  $\{10, 13, 22\}$&  $\{12, 15, 19\}$\\
&$\{2, 7, 12\}$&  $\{9, 19, 23\}$&  $\{5, 19, 20\}$&  $\{2, 9, 11\}$&  $\{9, 14, 21\}$&  $\{9, 15, 18\}$&  $\{9, 20, 22\}$ \\
&$\{5, 15, 23\}$&  $\{0, 6, 18\}$&  $\{6, 13, 18\}$&  $\{6, 15, 21\}$&  $\{0, 7, 11\}$&  $\{0, 13, 15\}$&  $\{0, 2, 19\}$\\
&$\{0, 5, 21\}$&  $\{2, 5, 22\}$&  $\{2, 4, 21\}$&  $\{1, 4, 19\}$&  $\{7, 13, 19\}$&  $\{1, 13, 20\}$&  $\{1, 15, 16\}$ \\
&$\{1, 7, 22\}$&  $\{1, 5, 11\}$&  $\{1, 18, 21\}$&  $\{5, 7, 16\}$&  $\{5, 12, 14\}$&  $\{7, 14, 20\}$&  $\{7, 21, 23\}$\\
&$\{11, 12, 22\}$&  $\{4, 18, 22\}$&  $\{4, 11, 23\}$&  $\{11, 15, 20\}$&  $\{11, 14, 18\}$&  $\{11, 16, 21\}$&  $\{13, 16, 23\}$ \\
&$\{14, 16, 22\}$&  $\{18, 20, 23\}$\\
\end{longtable}}

{\small\begin{longtable}{llllllll}
$\mathcal{A}_{5}$: &$\{a, 1, 3\}$& $\{a, 1, 6\}$& $\{a, 3, 22\}$& $\{a, 6, 11\}$& $\{a, 9, 11\}$& $\{a, 9, 14\}$& $\{a, 14, 19\}$\\
&$\{a, 17, 22\}$& $\{a, 17, 19\}$& $\{a, 5, 7\}$& $\{a, 5, 10\}$& $\{a, 2, 7\}$& $\{a, 2, 21\}$& $\{a, 13, 15\}$\\
&$\{a, 10, 15\}$& $\{a, 13, 18\}$& $\{a, 18, 23\}$& $\{a, 21, 23\}$& $\{b, 1, 3\}$& $\{b, 1, 11\}$& $\{b, 3, 13\}$\\
&$\{b, 11, 14\}$& $\{b, 14, 17\}$& $\{b, 17, 19\}$& $\{b, 2, 13\}$& $\{b, 2, 15\}$& $\{b, 5, 15\}$& $\{b, 5, 18\}$\\
&$\{b, 6, 19\}$& $\{b, 6, 9\}$& $\{b, 9, 22\}$& $\{b, 7, 22\}$& $\{b, 7, 18\}$& $\{b, 10, 21\}$& $\{b, 10, 23\}$\\
&$\{b, 21, 23\}$& $\{0, 2, 3\}$& $\{0, 3, 22\}$& $\{3, 4, 23\}$& $\{3, 4, 6\}$& $\{0, 13, 14\}$& $\{0, 6, 14\}$\\
&$\{0, 1, 13\}$& $\{0, 5, 6\}$& $\{0, 5, 10\}$& $\{0, 7, 15\}$& $\{0, 7, 9\}$& $\{0, 9, 17\}$& $\{0, 2, 15\}$\\
&$\{0, 21, 22\}$& $\{0, 19, 21\}$& $\{0, 10, 18\}$& $\{0, 17, 23\}$& $\{0, 1, 11\}$& $\{0, 11, 19\}$& $\{0, 18, 23\}$\\
&$\{4, 17, 18\}$& $\{4, 10, 18\}$& $\{9, 10, 20\}$& $\{6, 10, 20\}$& $\{1, 4, 9\}$& $\{1, 2, 4\}$& $\{1, 12, 13\}$\\
&$\{1, 2, 12\}$& $\{1, 9, 18\}$& $\{1, 7, 18\}$& $\{1, 17, 20\}$& $\{1, 15, 20\}$& $\{1, 16, 17\}$& $\{1, 6, 22\}$\\
&$\{1, 21, 22\}$& $\{1, 10, 21\}$& $\{1, 10, 23\}$& $\{1, 16, 23\}$& $\{1, 14, 15\}$& $\{1, 5, 14\}$& $\{1, 5, 8\}$\\
&$\{1, 7, 19\}$& $\{1, 8, 19\}$& $\{2, 9, 10\}$& $\{2, 10, 16\}$& $\{10, 12, 22\}$& $\{10, 17, 22\}$& $\{4, 10, 11\}$\\
&$\{7, 8, 10\}$& $\{8, 10, 11\}$& $\{10, 12, 13\}$& $\{6, 7, 10\}$& $\{10, 13, 14\}$& $\{3, 10, 19\}$& $\{10, 14, 19\}$\\
&$\{3, 10, 16\}$& $\{10, 15, 17\}$& $\{16, 18, 19\}$& $\{16, 19, 21\}$& $\{6, 7, 19\}$& $\{8, 9, 19\}$& $\{19, 20, 23\}$\\
&$\{15, 19, 20\}$& $\{7, 14, 15\}$& $\{14, 17, 18\}$& $\{3, 14, 18\}$& $\{13, 17, 20\}$& $\{3, 13, 17\}$& $\{3, 18, 19\}$\\
&$\{9, 12, 19\}$& $\{5, 12, 19\}$& $\{5, 15, 19\}$& $\{3, 15, 17\}$& $\{3, 6, 15\}$& $\{11, 13, 19\}$& $\{7, 12, 17\}$ \\
&$\{6, 12, 17\}$& $\{5, 16, 17\}$& $\{7, 17, 21\}$& $\{6, 17, 21\}$& $\{2, 4, 19\}$& $\{2, 13, 19\}$& $\{4, 19, 22\}$ \\
&$\{19, 22, 23\}$& $\{2, 9, 17\}$& $\{2, 8, 17\}$& $\{4, 11, 17\}$& $\{8, 11, 17\}$& $\{5, 17, 23\}$& $\{8, 13, 18\}$ \\
&$\{8, 13, 15\}$& $\{4, 13, 21\}$& $\{4, 7, 13\}$& $\{8, 15, 23\}$& $\{15, 16, 18\}$& $\{9, 15, 16\}$& $\{15, 18, 22\}$\\
&$\{11, 12, 15\}$& $\{11, 15, 21\}$& $\{12, 15, 23\}$& $\{9, 15, 21\}$& $\{4, 15, 22\}$& $\{4, 6, 15\}$& $\{4, 7, 9\}$ \\
&$\{4, 5, 21\}$& $\{4, 5, 14\}$& $\{4, 14, 23\}$& $\{3, 7, 8\}$& $\{5, 8, 23\}$& $\{3, 8, 21\}$& $\{5, 13, 22\}$ \\
&$\{5, 13, 20\}$& $\{2, 5, 7\}$& $\{2, 14, 16\}$& $\{2, 14, 20\}$& $\{2, 6, 8\}$& $\{2, 22, 23\}$& $\{2, 11, 22\}$ \\
&$\{8, 14, 22\}$& $\{8, 18, 22\}$& $\{8, 9, 21\}$& $\{6, 8, 14\}$& $\{13, 16, 21\}$& $\{11, 20, 22\}$& $\{7, 20, 22\}$\\
&$\{12, 14, 22\}$& $\{5, 9, 22\}$& $\{6, 16, 22\}$& $\{13, 16, 22\}$& $\{7, 11, 13\}$& $\{6, 9, 13\}$& $\{6, 13, 23\}$\\
&$\{9, 13, 23\}$& $\{3, 9, 23\}$& $\{2, 6, 23\}$& $\{11, 16, 23\}$& $\{11, 14, 23\}$& $\{7, 12, 23\}$& $\{7, 20, 23\}$\\
&$\{3, 7, 16\}$& $\{9, 14, 16\}$& $\{7, 14, 21\}$& $\{7, 11, 16\}$& $\{5, 6, 16\}$& $\{3, 12, 14\}$& $\{14, 20, 21\}$\\
&$\{3, 9, 12\}$& $\{3, 11, 21\}$& $\{2, 12, 21\}$& $\{2, 18, 20\}$& $\{2, 11, 18\}$& $\{2, 3, 5\}$& $\{3, 5, 20\}$ \\
&$\{3, 11, 20\}$& $\{5, 9, 11\}$& $\{9, 18, 20\}$& $\{6, 20, 21\}$& $\{5, 11, 12\}$& $\{5, 18, 21\}$& $\{6, 11, 18\}$\\
&$\{6, 12, 18\}$& $\{12, 18, 21\}$\\
\end{longtable}}

{\small\begin{longtable}{llllllll}
$\mathcal{A}_{6}$: &$\{b, 7, 8\}$& $\{b, 7, 20\}$& $\{b, 8, 10\}$& $\{b, 10, 12\}$& $\{b, 19, 20\}$& $\{b, 16, 19\}$& $\{b, 0, 22\}$\\
&$\{b, 12, 22\}$& $\{b, 0, 11\}$& $\{b, 2, 11\}$& $\{b, 2, 23\}$& $\{b, 14, 23\}$& $\{b, 14, 16\}$& $\{b, 3, 4\}$\\
&$\{b, 3, 6\}$& $\{b, 6, 15\}$& $\{b, 4, 18\}$& $\{b, 15, 18\}$& $\{1, 7, 15\}$& $\{1, 10, 15\}$& $\{5, 18, 19\}$\\
&$\{5, 11, 19\}$& $\{a, 0, 2\}$& $\{a, 2, 4\}$& $\{a, 15, 16\}$& $\{a, 15, 20\}$& $\{a, 16, 18\}$& $\{a, 18, 20\}$\\
&$\{a, 4, 6\}$& $\{a, 6, 8\}$& $\{a, 3, 8\}$& $\{a, 3, 14\}$& $\{a, 7, 12\}$& $\{a, 7, 10\}$& $\{a, 0, 19\}$\\
&$\{a, 19, 22\}$& $\{a, 10, 23\}$& $\{a, 12, 23\}$& $\{a, 11, 14\}$& $\{a, 11, 22\}$& $\{6, 7, 17\}$& $\{7, 17, 23\}$\\
&$\{6, 17, 18\}$& $\{10, 11, 21\}$& $\{7, 11, 21\}$& $\{7, 10, 11\}$& $\{1, 7, 23\}$& $\{6, 7, 15\}$& $\{2, 5, 6\}$\\
&$\{2, 3, 6\}$& $\{6, 9, 10\}$& $\{6, 10, 19\}$& $\{1, 6, 11\}$& $\{1, 6, 20\}$& $\{1, 11, 19\}$& $\{1, 2, 19\}$\\
&$\{1, 4, 12\}$& $\{0, 1, 4\}$& $\{1, 8, 16\}$& $\{1, 8, 14\}$& $\{1, 2, 10\}$& $\{0, 1, 18\}$& $\{1, 3, 18\}$ \\
&$\{1, 3, 16\}$& $\{1, 22, 23\}$& $\{1, 20, 22\}$& $\{1, 12, 14\}$& $\{2, 5, 10\}$& $\{3, 7, 20\}$& $\{3, 7, 12\}$\\
&$\{7, 13, 14\}$& $\{0, 7, 13\}$& $\{4, 5, 7\}$& $\{5, 7, 8\}$& $\{7, 9, 22\}$& $\{7, 9, 18\}$& $\{0, 7, 16\}$\\
&$\{2, 4, 7\}$& $\{7, 16, 18\}$& $\{7, 14, 22\}$& $\{2, 7, 19\}$& $\{7, 19, 21\}$& $\{2, 13, 14\}$& $\{2, 9, 14\}$\\
&$\{2, 3, 22\}$& $\{2, 15, 22\}$& $\{3, 9, 19\}$& $\{9, 10, 19\}$& $\{2, 12, 20\}$& $\{2, 12, 18\}$& $\{2, 13, 18\}$\\
&$\{2, 15, 16\}$& $\{2, 16, 21\}$& $\{2, 21, 23\}$& $\{2, 8, 9\}$& $\{0, 2, 8\}$& $\{2, 11, 17\}$& $\{2, 17, 20\}$\\
&$\{12, 13, 15\}$& $\{13, 15, 18\}$& $\{8, 10, 13\}$& $\{10, 13, 20\}$& $\{10, 17, 18\}$& $\{10, 18, 23\}$& $\{10, 12, 20\}$\\
&$\{5, 10, 15\}$& $\{3, 10, 14\}$& $\{10, 14, 17\}$& $\{3, 10, 22\}$& $\{0, 10, 16\}$& $\{0, 10, 22\}$& $\{4, 10, 21\}$ \\
&$\{4, 10, 16\}$& $\{3, 15, 20\}$& $\{3, 5, 15\}$& $\{12, 13, 16\}$& $\{11, 13, 16\}$& $\{0, 3, 9\}$& $\{3, 11, 16\}$\\
&$\{3, 11, 17\}$& $\{3, 17, 19\}$& $\{3, 13, 23\}$& $\{3, 4, 13\}$& $\{3, 12, 21\}$& $\{3, 18, 21\}$& $\{0, 3, 8\}$ \\
&$\{3, 5, 23\}$& $\{11, 13, 23\}$& $\{8, 13, 20\}$& $\{0, 6, 13\}$& $\{4, 13, 19\}$& $\{6, 13, 22\}$& $\{13, 19, 22\}$\\
&$\{5, 14, 16\}$& $\{16, 17, 20\}$& $\{16, 17, 19\}$& $\{0, 12, 17\}$& $\{0, 4, 17\}$& $\{0, 5, 18\}$& $\{0, 5, 20\}$\\
&$\{5, 11, 20\}$& $\{11, 18, 23\}$& $\{9, 12, 18\}$& $\{18, 21, 22\}$& $\{11, 18, 22\}$& $\{6, 8, 18\}$& $\{4, 8, 18\}$\\
&$\{14, 18, 19\}$& $\{14, 18, 20\}$& $\{6, 11, 12\}$& $\{0, 11, 15\}$& $\{6, 9, 16\}$& $\{0, 6, 21\}$& $\{0, 20, 21\}$ \\
&$\{0, 9, 14\}$& $\{0, 12, 23\}$& $\{0, 19, 23\}$& $\{0, 14, 15\}$& $\{9, 15, 23\}$& $\{9, 15, 22\}$& $\{4, 9, 23\}$\\
&$\{11, 14, 15\}$& $\{8, 12, 14\}$& $\{14, 17, 22\}$& $\{5, 14, 23\}$& $\{5, 6, 22\}$& $\{4, 5, 22\}$& $\{5, 8, 12\}$\\
&$\{5, 12, 16\}$& $\{9, 11, 12\}$& $\{4, 11, 20\}$& $\{4, 8, 11\}$& $\{8, 9, 11\}$& $\{4, 9, 20\}$& $\{9, 16, 20\}$ \\
&$\{4, 6, 14\}$& $\{4, 14, 21\}$& $\{6, 16, 23\}$& $\{6, 20, 23\}$& $\{14, 19, 20\}$& $\{6, 14, 21\}$& $\{6, 12, 19\}$\\
&$\{4, 16, 22\}$& $\{12, 17, 22\}$& $\{4, 12, 19\}$& $\{12, 15, 21\}$& $\{8, 19, 23\}$& $\{4, 15, 23\}$& $\{4, 15, 17\}$\\
&$\{8, 15, 17\}$& $\{8, 17, 23\}$& $\{8, 20, 22\}$& $\{20, 21, 23\}$& $\{16, 22, 23\}$& $\{8, 16, 21\}$& $\{8, 21, 22\}$\\
&$\{8, 15, 19\}$& $\{15, 19, 21\}$\\
\end{longtable}}

{\small\begin{longtable}{llllllll}
$\mathcal{A}_{7}$: &$\{b, 4, 7\}$& $\{b, 7, 10\}$& $\{b, 10, 19\}$& $\{b, 19, 22\}$& $\{b, 11, 22\}$& $\{b, 11, 12\}$& $\{b, 6, 8\}$\\
&$\{b, 8, 18\}$& $\{b, 6, 20\}$& $\{b, 20, 23\}$& $\{b, 3, 18\}$& $\{b, 0, 23\}$& $\{b, 2, 4\}$& $\{b, 2, 16\}$\\
&$\{b, 15, 16\}$& $\{b, 12, 15\}$& $\{b, 0, 14\}$& $\{b, 3, 14\}$& $\{1, 15, 18\}$& $\{1, 15, 19\}$& $\{5, 14, 19\}$\\
&$\{5, 19, 22\}$& $\{a, 2, 15\}$& $\{a, 2, 23\}$& $\{a, 15, 18\}$& $\{a, 7, 18\}$& $\{a, 7, 8\}$& $\{a, 3, 4\}$\\
&$\{a, 4, 23\}$& $\{a, 3, 6\}$& $\{a, 6, 19\}$& $\{a, 12, 14\}$& $\{a, 14, 16\}$& $\{a, 11, 16\}$& $\{a, 19, 20\}$\\
&$\{a, 20, 22\}$& $\{a, 0, 22\}$& $\{a, 0, 11\}$& $\{a, 8, 10\}$& $\{a, 10, 12\}$& $\{4, 15, 19\}$& $\{7, 8, 17\}$\\
&$\{7, 10, 17\}$& $\{1, 6, 7\}$& $\{1, 7, 14\}$& $\{7, 11, 12\}$& $\{7, 11, 20\}$& $\{6, 11, 21\}$& $\{11, 14, 21\}$ \\
&$\{9, 11, 19\}$& $\{7, 9, 19\}$& $\{7, 18, 19\}$& $\{5, 6, 7\}$& $\{5, 7, 20\}$& $\{0, 7, 12\}$& $\{0, 4, 7\}$\\
&$\{3, 7, 9\}$& $\{3, 7, 21\}$& $\{7, 13, 23\}$& $\{7, 13, 16\}$& $\{7, 14, 23\}$& $\{7, 16, 21\}$& $\{2, 7, 15\}$\\
&$\{2, 7, 22\}$& $\{7, 15, 22\}$& $\{2, 6, 16\}$& $\{2, 6, 19\}$& $\{2, 9, 18\}$& $\{2, 9, 20\}$& $\{1, 2, 22\}$\\
&$\{1, 3, 6\}$& $\{0, 1, 19\}$& $\{2, 17, 18\}$& $\{2, 10, 21\}$& $\{2, 3, 21\}$& $\{0, 8, 21\}$& $\{0, 21, 23\}$\\
&$\{9, 22, 23\}$& $\{9, 18, 23\}$& $\{0, 2, 10\}$& $\{2, 4, 17\}$& $\{1, 2, 14\}$& $\{2, 14, 19\}$& $\{2, 3, 23\}$\\
&$\{0, 2, 12\}$& $\{2, 5, 12\}$& $\{2, 5, 11\}$& $\{2, 8, 13\}$& $\{2, 8, 11\}$& $\{2, 13, 20\}$& $\{3, 8, 19\}$\\
&$\{3, 13, 19\}$& $\{3, 8, 13\}$& $\{3, 12, 23\}$& $\{10, 11, 19\}$& $\{0, 8, 19\}$& $\{13, 19, 23\}$& $\{16, 19, 23\}$\\
&$\{4, 16, 19\}$& $\{8, 12, 17\}$& $\{10, 21, 22\}$& $\{12, 13, 22\}$& $\{6, 12, 13\}$& $\{10, 13, 18\}$& $\{13, 18, 20\}$\\
&$\{6, 13, 14\}$& $\{10, 13, 15\}$& $\{13, 15, 16\}$& $\{11, 13, 14\}$& $\{0, 4, 13\}$& $\{0, 11, 13\}$& $\{4, 13, 22\}$\\
&$\{0, 3, 5\}$& $\{0, 5, 15\}$& $\{0, 6, 15\}$& $\{0, 14, 17\}$& $\{0, 17, 20\}$& $\{14, 16, 17\}$& $\{0, 16, 18\}$\\
&$\{0, 9, 16\}$& $\{0, 18, 22\}$& $\{0, 6, 9\}$& $\{0, 1, 10\}$& $\{0, 3, 20\}$& $\{6, 9, 15\}$& $\{3, 9, 15\}$\\
&$\{9, 12, 16\}$& $\{9, 11, 20\}$& $\{3, 10, 15\}$& $\{3, 10, 17\}$& $\{3, 14, 22\}$& $\{3, 12, 20\}$& $\{1, 3, 4\}$\\
&$\{1, 11, 23\}$& $\{1, 4, 11\}$& $\{1, 8, 22\}$& $\{1, 8, 20\}$& $\{1, 10, 18\}$& $\{1, 12, 16\}$& $\{1, 16, 20\}$\\
&$\{1, 12, 23\}$& $\{8, 12, 18\}$& $\{4, 8, 9\}$& $\{8, 9, 10\}$& $\{9, 12, 14\}$& $\{4, 9, 10\}$& $\{9, 14, 22\}$\\
&$\{5, 10, 12\}$& $\{6, 10, 11\}$& $\{6, 10, 23\}$& $\{8, 11, 15\}$& $\{5, 11, 23\}$& $\{3, 11, 22\}$& $\{3, 11, 18\}$\\
&$\{3, 5, 16\}$& $\{3, 16, 17\}$& $\{11, 15, 17\}$& $\{11, 17, 18\}$& $\{4, 11, 16\}$& $\{6, 8, 16\}$& $\{14, 15, 23\}$ \\
&$\{8, 14, 15\}$& $\{5, 15, 22\}$& $\{15, 17, 23\}$& $\{12, 15, 20\}$& $\{4, 15, 21\}$& $\{15, 20, 21\}$& $\{5, 10, 20\}$\\
&$\{10, 16, 23\}$& $\{5, 8, 14\}$& $\{6, 14, 20\}$& $\{8, 20, 23\}$& $\{14, 18, 21\}$& $\{4, 14, 18\}$& $\{4, 10, 14\}$\\
&$\{10, 14, 20\}$& $\{10, 16, 22\}$& $\{4, 5, 8\}$& $\{5, 16, 18\}$& $\{5, 6, 18\}$& $\{4, 5, 23\}$& $\{8, 16, 22\}$\\
&$\{8, 21, 23\}$& $\{16, 20, 21\}$& $\{6, 21, 22\}$& $\{6, 18, 23\}$& $\{17, 22, 23\}$& $\{12, 18, 22\}$& $\{4, 20, 22\}$ \\
&$\{4, 6, 12\}$& $\{4, 6, 17\}$& $\{4, 12, 21\}$& $\{4, 18, 20\}$& $\{6, 17, 22\}$& $\{12, 17, 19\}$& $\{12, 19, 21\}$\\
&$\{17, 19, 20\}$& $\{18, 19, 21\}$\\
\end{longtable}}

{\small\begin{longtable}{llllllll}
$\mathcal{A}_{8}$: &$\{a, 3, 16\}$& $\{a, 3, 13\}$& $\{a, 3, 17\}$& $\{a, 0, 3\}$& $\{a, 7, 20\}$& $\{a, 7, 17\}$& $\{a, 7, 21\}$\\
&$\{a, 4, 7\}$& $\{a, 11, 12\}$& $\{a, 1, 11\}$& $\{a, 11, 21\}$& $\{a, 8, 11\}$& $\{a, 4, 15\}$& $\{a, 5, 15\}$\\
&$\{a, 1, 15\}$& $\{a, 12, 15\}$& $\{a, 0, 1\}$& $\{a, 1, 20\}$& $\{a, 4, 5\}$& $\{a, 4, 9\}$& $\{a, 12, 13\}$\\
&$\{a, 12, 17\}$& $\{a, 16, 17\}$& $\{a, 20, 21\}$& $\{a, 20, 23\}$& $\{a, 8, 9\}$& $\{a, 9, 19\}$& $\{a, 9, 23\}$\\
&$\{a, 0, 23\}$& $\{a, 13, 23\}$& $\{a, 0, 5\}$& $\{a, 8, 13\}$& $\{a, 16, 21\}$& $\{a, 8, 19\}$& $\{a, 5, 19\}$\\
&$\{a, 16, 19\}$& $\{b, 3, 17\}$& $\{b, 3, 16\}$& $\{b, 0, 3\}$& $\{b, 3, 12\}$& $\{b, 5, 7\}$& $\{b, 7, 17\}$\\
&$\{b, 7, 21\}$& $\{b, 7, 16\}$& $\{b, 9, 11\}$& $\{b, 11, 21\}$& $\{b, 8, 11\}$& $\{b, 11, 20\}$& $\{b, 13, 15\}$\\
&$\{b, 1, 15\}$& $\{b, 4, 15\}$& $\{b, 0, 15\}$& $\{b, 0, 1\}$& $\{b, 0, 13\}$& $\{b, 1, 12\}$& $\{b, 1, 4\}$\\
&$\{b, 16, 17\}$& $\{b, 5, 16\}$& $\{b, 12, 13\}$& $\{b, 13, 23\}$& $\{b, 4, 5\}$& $\{b, 4, 19\}$& $\{b, 17, 20\}$\\
&$\{b, 5, 19\}$& $\{b, 9, 19\}$& $\{b, 8, 19\}$& $\{b, 20, 21\}$& $\{b, 9, 20\}$& $\{b, 8, 21\}$& $\{b, 8, 23\}$\\
&$\{b, 9, 23\}$& $\{b, 12, 23\}$& $\{3, 6, 7\}$& $\{3, 7, 10\}$& $\{0, 3, 7\}$& $\{3, 7, 13\}$& $\{2, 3, 11\}$\\
&$\{3, 10, 11\}$& $\{1, 3, 11\}$& $\{3, 5, 11\}$& $\{3, 9, 10\}$& $\{3, 10, 21\}$& $\{0, 3, 4\}$& $\{0, 10, 20\}$\\
&$\{0, 10, 19\}$& $\{0, 4, 10\}$& $\{0, 10, 17\}$& $\{2, 3, 15\}$& $\{3, 14, 15\}$& $\{3, 15, 22\}$& $\{3, 15, 18\}$\\
&$\{3, 17, 18\}$& $\{3, 5, 18\}$& $\{3, 8, 18\}$& $\{2, 3, 19\}$& $\{3, 14, 19\}$& $\{3, 19, 21\}$& $\{3, 12, 19\}$\\
&$\{3, 16, 23\}$& $\{3, 8, 23\}$& $\{3, 6, 23\}$& $\{3, 20, 23\}$& $\{7, 11, 14\}$& $\{1, 7, 11\}$& $\{4, 7, 11\}$\\
&$\{7, 11, 17\}$& $\{7, 14, 16\}$& $\{7, 12, 14\}$& $\{7, 14, 19\}$& $\{7, 12, 19\}$& $\{0, 7, 19\}$& $\{7, 19, 22\}$\\
&$\{4, 14, 20\}$& $\{0, 4, 14\}$& $\{4, 9, 14\}$& $\{4, 8, 14\}$& $\{7, 10, 15\}$& $\{7, 15, 20\}$& $\{7, 15, 21\}$\\
&$\{7, 9, 15\}$& $\{7, 22, 23\}$& $\{6, 7, 23\}$& $\{7, 18, 23\}$& $\{2, 7, 23\}$& $\{2, 7, 8\}$& $\{2, 7, 21\}$\\
&$\{2, 7, 20\}$& $\{2, 8, 16\}$& $\{2, 8, 23\}$& $\{2, 8, 12\}$& $\{2, 4, 12\}$& $\{2, 12, 16\}$& $\{2, 3, 12\}$\\
&$\{2, 4, 5\}$& $\{2, 4, 15\}$& $\{2, 4, 13\}$& $\{2, 9, 15\}$& $\{2, 15, 17\}$& $\{2, 5, 13\}$& $\{2, 13, 21\}$\\
&$\{2, 9, 13\}$& $\{1, 2, 5\}$& $\{0, 2, 5\}$& $\{2, 9, 21\}$& $\{2, 9, 19\}$& $\{2, 16, 17\}$& $\{2, 16, 20\}$\\
&$\{2, 11, 19\}$& $\{2, 19, 23\}$& $\{0, 2, 17\}$& $\{2, 17, 21\}$& $\{1, 2, 20\}$& $\{0, 2, 20\}$& $\{0, 1, 2\}$\\
&$\{1, 2, 11\}$& $\{2, 11, 23\}$& $\{11, 18, 19\}$& $\{11, 14, 19\}$& $\{4, 11, 19\}$& $\{10, 11, 23\}$& $\{11, 22, 23\}$\\
&$\{6, 11, 23\}$& $\{4, 19, 23\}$& $\{1, 19, 23\}$& $\{5, 19, 23\}$& $\{0, 1, 17\}$& $\{0, 17, 21\}$& $\{0, 9, 21\}$\\
&$\{0, 15, 21\}$& $\{0, 4, 21\}$& $\{0, 5, 13\}$& $\{0, 5, 16\}$& $\{0, 7, 8\}$& $\{0, 7, 18\}$& $\{5, 7, 18\}$\\
&$\{7, 18, 20\}$& $\{0, 6, 16\}$& $\{0, 6, 11\}$& $\{0, 6, 20\}$& $\{0, 6, 12\}$& $\{0, 8, 18\}$& $\{0, 11, 18\}$\\
&$\{0, 12, 18\}$& $\{8, 9, 18\}$& $\{8, 11, 18\}$& $\{9, 16, 18\}$& $\{5, 9, 18\}$& $\{9, 13, 18\}$& $\{0, 11, 12\}$\\
&$\{0, 11, 16\}$& $\{0, 12, 14\}$& $\{0, 9, 23\}$& $\{0, 9, 22\}$& $\{0, 9, 15\}$& $\{0, 15, 19\}$& $\{0, 20, 23\}$\\
&$\{0, 14, 16\}$& $\{0, 8, 14\}$& $\{0, 13, 19\}$& $\{0, 13, 22\}$& $\{0, 8, 22\}$& $\{0, 22, 23\}$& $\{14, 16, 19\}$ \\
&$\{15, 18, 19\}$& $\{15, 19, 22\}$& $\{9, 15, 19\}$& $\{19, 20, 22\}$& $\{19, 21, 22\}$& $\{1, 13, 18\}$& $\{1, 5, 18\}$\\
&$\{1, 12, 18\}$& $\{1, 18, 23\}$& $\{11, 15, 18\}$& $\{11, 15, 16\}$& $\{4, 11, 15\}$& $\{5, 11, 15\}$& $\{4, 5, 17\}$\\
&$\{15, 18, 23\}$& $\{15, 22, 23\}$& $\{10, 15, 23\}$& $\{6, 15, 23\}$& $\{8, 12, 22\}$& $\{4, 8, 22\}$& $\{8, 15, 22\}$\\
&$\{8, 16, 23\}$& $\{10, 15, 16\}$& $\{10, 15, 20\}$& $\{15, 17, 20\}$& $\{6, 15, 20\}$& $\{5, 6, 15\}$& $\{6, 15, 17\}$\\
&$\{1, 5, 15\}$& $\{13, 15, 21\}$& $\{13, 14, 15\}$& $\{1, 13, 15\}$& $\{8, 15, 21\}$& $\{8, 15, 16\}$& $\{8, 12, 15\}$\\
&$\{12, 15, 16\}$& $\{12, 14, 15\}$& $\{14, 15, 17\}$& $\{17, 18, 20\}$& $\{18, 20, 21\}$& $\{18, 19, 20\}$& $\{17, 18, 21\}$\\
&$\{16, 18, 21\}$& $\{4, 18, 21\}$& $\{4, 18, 23\}$& $\{4, 12, 18\}$& $\{4, 16, 18\}$& $\{4, 6, 19\}$& $\{4, 6, 7\}$\\
&$\{1, 4, 6\}$& $\{4, 6, 16\}$& $\{6, 7, 9\}$& $\{6, 11, 16\}$& $\{6, 12, 16\}$& $\{6, 11, 13\}$& $\{6, 9, 17\}$\\
&$\{5, 6, 9\}$& $\{6, 8, 9\}$& $\{5, 6, 8\}$& $\{3, 5, 6\}$& $\{5, 9, 10\}$& $\{5, 9, 14\}$& $\{6, 12, 20\}$\\
&$\{6, 12, 21\}$& $\{6, 13, 20\}$& $\{3, 6, 13\}$& $\{3, 12, 17\}$& $\{6, 8, 21\}$& $\{1, 6, 8\}$& $\{1, 6, 17\}$\\
&$\{6, 17, 19\}$& $\{1, 6, 21\}$& $\{6, 13, 19\}$& $\{6, 19, 21\}$& $\{1, 14, 21\}$& $\{1, 3, 21\}$& $\{1, 4, 21\}$\\
&$\{1, 3, 9\}$& $\{1, 3, 22\}$& $\{1, 9, 22\}$& $\{1, 16, 22\}$& $\{1, 5, 22\}$& $\{3, 14, 21\}$& $\{3, 4, 16\}$\\
&$\{3, 4, 20\}$& $\{3, 4, 22\}$& $\{3, 13, 22\}$& $\{3, 5, 8\}$& $\{3, 8, 20\}$& $\{3, 9, 20\}$& $\{3, 9, 14\}$\\
&$\{5, 8, 10\}$& $\{5, 8, 17\}$& $\{5, 14, 21\}$& $\{5, 14, 20\}$& $\{5, 11, 14\}$& $\{5, 17, 19\}$& $\{5, 13, 23\}$\\
&$\{5, 7, 13\}$& $\{5, 7, 22\}$& $\{5, 10, 21\}$& $\{5, 10, 11\}$& $\{5, 12, 17\}$& $\{5, 16, 22\}$& $\{5, 12, 22\}$\\
&$\{5, 16, 20\}$& $\{5, 12, 20\}$& $\{5, 12, 21\}$& $\{5, 20, 23\}$& $\{5, 21, 23\}$& $\{1, 9, 20\}$& $\{1, 14, 20\}$\\
&$\{9, 12, 20\}$& $\{1, 9, 14\}$& $\{1, 14, 23\}$& $\{1, 4, 23\}$& $\{11, 14, 20\}$& $\{13, 14, 16\}$& $\{11, 16, 17\}$\\
&$\{8, 10, 16\}$& $\{11, 12, 20\}$& $\{13, 16, 18\}$& $\{12, 18, 19\}$& $\{13, 17, 18\}$& $\{4, 12, 17\}$& $\{12, 14, 23\}$\\
&$\{1, 12, 19\}$& $\{1, 10, 12\}$& $\{12, 16, 22\}$& $\{4, 12, 22\}$& $\{4, 17, 23\}$& $\{4, 7, 22\}$& $\{4, 11, 13\}$\\
&$\{4, 9, 16\}$& $\{4, 9, 17\}$& $\{4, 8, 21\}$& $\{4, 8, 10\}$& $\{4, 10, 20\}$& $\{4, 10, 13\}$& $\{4, 13, 20\}$\\
&$\{1, 7, 8\}$& $\{7, 8, 9\}$& $\{1, 8, 17\}$& $\{1, 8, 13\}$& $\{1, 7, 16\}$& $\{7, 10, 16\}$& $\{8, 10, 20\}$\\
&$\{8, 12, 13\}$& $\{8, 11, 20\}$& $\{8, 13, 14\}$& $\{8, 14, 17\}$& $\{8, 17, 19\}$& $\{8, 19, 20\}$& $\{17, 20, 22\}$\\
&$\{9, 11, 17\}$& $\{11, 17, 22\}$& $\{9, 11, 22\}$& $\{11, 13, 22\}$& $\{9, 11, 21\}$& $\{9, 17, 22\}$& $\{7, 9, 12\}$\\
&$\{1, 7, 10\}$& $\{7, 12, 13\}$& $\{7, 13, 17\}$& $\{1, 10, 17\}$& $\{1, 10, 16\}$& $\{1, 13, 19\}$& $\{1, 16, 19\}$\\
&$\{16, 19, 20\}$& $\{10, 19, 21\}$& $\{9, 10, 13\}$& $\{9, 10, 12\}$& $\{9, 12, 21\}$& $\{9, 13, 16\}$& $\{9, 16, 23\}$\\
&$\{13, 16, 20\}$& $\{16, 21, 22\}$& $\{16, 21, 23\}$& $\{10, 11, 12\}$& $\{10, 12, 23\}$& $\{11, 13, 21\}$& $\{12, 21, 23\}$ \\
&$\{10, 17, 19\}$& $\{10, 13, 19\}$& $\{10, 13, 21\}$& $\{10, 17, 23\}$& $\{13, 20, 22\}$& $\{17, 21, 22\}$& $\{14, 21, 23\}$\\
&$\{20, 21, 22\}$& $\{13, 14, 17\}$& $\{13, 17, 23\}$& $\{14, 17, 23\}$\\
\end{longtable}}

{\small\begin{longtable}{llllllll}
$\mathcal{A}_{9}$: &$\{a, 0, 13\}$& $\{a, 0, 10\}$& $\{a, 0, 14\}$& $\{a, 0, 21\}$& $\{a, 10, 21\}$& $\{a, 10, 20\}$& $\{a, 10, 13\}$\\
&$\{a, 2, 13\}$& $\{a, 13, 16\}$& $\{a, 8, 21\}$& $\{a, 18, 21\}$& $\{a, 2, 12\}$& $\{a, 2, 16\}$& $\{a, 2, 5\}$\\
&$\{a, 5, 16\}$& $\{a, 6, 16\}$& $\{a, 5, 18\}$& $\{a, 5, 8\}$& $\{a, 8, 18\}$& $\{a, 8, 22\}$& $\{a, 4, 18\}$\\
&$\{a, 4, 17\}$& $\{a, 4, 14\}$& $\{a, 1, 4\}$& $\{a, 1, 14\}$& $\{a, 14, 17\}$& $\{a, 1, 12\}$& $\{a, 1, 22\}$\\
&$\{a, 12, 22\}$& $\{a, 9, 12\}$& $\{a, 9, 22\}$& $\{a, 9, 20\}$& $\{a, 6, 9\}$& $\{a, 6, 17\}$& $\{a, 6, 20\}$\\
&$\{a, 17, 20\}$& $\{b, 0, 10\}$& $\{0, 3, 10\}$& $\{0, 6, 10\}$& $\{b, 0, 2\}$& $\{b, 0, 21\}$& $\{b, 0, 9\}$\\
&$\{0, 2, 21\}$& $\{0, 7, 21\}$& $\{0, 2, 23\}$& $\{0, 2, 11\}$& $\{b, 2, 12\}$& $\{b, 2, 5\}$& $\{b, 2, 17\}$\\
&$\{2, 5, 20\}$& $\{2, 5, 9\}$& $\{2, 6, 12\}$& $\{2, 12, 19\}$& $\{b, 5, 8\}$& $\{b, 5, 14\}$& $\{b, 5, 20\}$\\
&$\{5, 8, 16\}$& $\{5, 8, 15\}$& $\{5, 16, 23\}$& $\{5, 11, 16\}$& $\{5, 20, 21\}$& $\{5, 18, 20\}$& $\{5, 13, 18\}$\\
&$\{5, 11, 18\}$& $\{5, 11, 21\}$& $\{5, 11, 17\}$& $\{5, 19, 21\}$& $\{5, 15, 21\}$& $\{4, 5, 15\}$& $\{5, 9, 15\}$\\
&$\{5, 9, 19\}$& $\{5, 9, 12\}$& $\{5, 13, 19\}$& $\{1, 5, 19\}$& $\{5, 12, 13\}$& $\{3, 5, 13\}$& $\{0, 5, 12\}$\\
&$\{5, 12, 23\}$& $\{5, 6, 23\}$& $\{5, 10, 23\}$& $\{b, 9, 12\}$& $\{8, 9, 12\}$& $\{b, 8, 9\}$& $\{b, 9, 18\}$\\
&$\{b, 8, 22\}$& $\{b, 8, 17\}$& $\{b, 12, 14\}$& $\{b, 12, 21\}$& $\{b, 4, 14\}$& $\{b, 1, 14\}$& $\{b, 6, 17\}$\\
&$\{b, 4, 17\}$& $\{b, 4, 6\}$& $\{b, 4, 13\}$& $\{b, 6, 16\}$& $\{b, 6, 21\}$& $\{b, 18, 21\}$& $\{b, 16, 18\}$\\
&$\{b, 18, 20\}$& $\{b, 13, 16\}$& $\{b, 1, 16\}$& $\{b, 1, 22\}$& $\{b, 1, 10\}$& $\{b, 10, 20\}$& $\{b, 10, 13\}$ \\
&$\{b, 13, 22\}$& $\{b, 20, 22\}$& $\{1, 3, 14\}$& $\{1, 14, 16\}$& $\{1, 9, 16\}$& $\{1, 13, 16\}$& $\{0, 13, 16\}$\\
&$\{0, 8, 13\}$& $\{0, 13, 20\}$& $\{0, 1, 22\}$& $\{1, 10, 22\}$& $\{1, 5, 10\}$& $\{1, 4, 10\}$& $\{1, 4, 13\}$\\
&$\{1, 4, 8\}$& $\{1, 5, 6\}$& $\{1, 5, 7\}$& $\{1, 6, 13\}$& $\{1, 7, 13\}$& $\{1, 6, 19\}$& $\{1, 6, 18\}$\\
&$\{1, 7, 17\}$& $\{1, 7, 21\}$& $\{1, 9, 19\}$& $\{1, 17, 19\}$& $\{1, 9, 23\}$& $\{1, 9, 15\}$& $\{1, 15, 17\}$ \\
&$\{1, 11, 17\}$& $\{1, 15, 21\}$& $\{1, 8, 15\}$& $\{1, 8, 12\}$& $\{1, 8, 18\}$& $\{1, 12, 20\}$& $\{1, 3, 12\}$\\
&$\{0, 1, 3\}$& $\{1, 2, 3\}$& $\{0, 1, 20\}$& $\{0, 1, 23\}$& $\{0, 3, 12\}$& $\{0, 3, 18\}$& $\{0, 12, 19\}$\\
&$\{0, 12, 15\}$& $\{0, 9, 20\}$& $\{0, 14, 20\}$& $\{0, 9, 11\}$& $\{0, 4, 9\}$& $\{0, 11, 22\}$& $\{0, 11, 17\}$\\
&$\{11, 17, 20\}$& $\{0, 14, 19\}$& $\{0, 14, 18\}$& $\{0, 15, 18\}$& $\{0, 4, 18\}$& $\{0, 4, 5\}$& $\{0, 4, 6\}$\\
&$\{0, 5, 7\}$& $\{0, 5, 17\}$& $\{0, 6, 22\}$& $\{0, 6, 7\}$& $\{0, 7, 17\}$& $\{0, 8, 17\}$& $\{0, 19, 22\}$ \\
&$\{0, 16, 19\}$& $\{0, 8, 23\}$& $\{0, 8, 15\}$& $\{0, 15, 16\}$& $\{0, 16, 23\}$& $\{8, 15, 20\}$& $\{1, 2, 18\}$ \\
&$\{1, 11, 18\}$& $\{1, 2, 21\}$& $\{1, 2, 23\}$& $\{1, 11, 21\}$& $\{1, 11, 20\}$& $\{1, 20, 23\}$& $\{2, 18, 21\}$\\
&$\{2, 15, 21\}$& $\{15, 21, 22\}$& $\{6, 18, 21\}$& $\{2, 7, 18\}$& $\{2, 8, 18\}$& $\{8, 10, 18\}$& $\{2, 9, 23\}$ \\
&$\{2, 17, 23\}$& $\{2, 7, 9\}$& $\{2, 9, 16\}$& $\{2, 6, 7\}$& $\{2, 7, 14\}$& $\{2, 6, 15\}$& $\{2, 6, 11\}$\\
&$\{2, 11, 13\}$& $\{2, 11, 20\}$& $\{10, 11, 20\}$& $\{10, 16, 20\}$& $\{2, 3, 13\}$& $\{2, 13, 15\}$& $\{2, 10, 15\}$\\
&$\{2, 3, 20\}$& $\{2, 3, 4\}$& $\{2, 8, 20\}$& $\{2, 8, 22\}$& $\{2, 8, 14\}$& $\{8, 22, 23\}$& $\{2, 4, 14\}$\\
&$\{2, 14, 17\}$& $\{2, 10, 17\}$& $\{4, 7, 14\}$& $\{2, 4, 10\}$& $\{2, 4, 22\}$& $\{2, 10, 19\}$& $\{2, 16, 19\}$\\
&$\{2, 16, 22\}$& $\{2, 19, 22\}$& $\{3, 4, 12\}$& $\{3, 6, 12\}$& $\{3, 4, 21\}$& $\{3, 4, 5\}$& $\{4, 5, 6\}$\\
&$\{4, 6, 23\}$& $\{5, 6, 14\}$& $\{3, 5, 14\}$& $\{3, 5, 22\}$& $\{5, 14, 22\}$& $\{3, 13, 21\}$& $\{3, 13, 20\}$\\
&$\{3, 14, 16\}$& $\{3, 14, 20\}$& $\{3, 17, 20\}$& $\{4, 17, 20\}$& $\{4, 8, 17\}$& $\{8, 9, 17\}$& $\{3, 8, 9\}$\\
&$\{3, 6, 21\}$& $\{3, 17, 21\}$& $\{6, 13, 21\}$& $\{3, 6, 8\}$& $\{3, 6, 10\}$& $\{3, 8, 16\}$& $\{3, 8, 10\}$\\
&$\{3, 10, 18\}$& $\{3, 9, 16\}$& $\{3, 16, 22\}$& $\{9, 16, 21\}$& $\{3, 9, 17\}$& $\{3, 9, 18\}$& $\{3, 17, 22\}$\\
&$\{3, 18, 22\}$& $\{4, 7, 8\}$& $\{4, 8, 13\}$& $\{4, 13, 15\}$& $\{4, 7, 12\}$& $\{4, 7, 16\}$& $\{4, 10, 19\}$\\
&$\{4, 10, 22\}$& $\{4, 11, 12\}$& $\{4, 12, 23\}$& $\{4, 15, 16\}$& $\{4, 15, 20\}$& $\{4, 16, 21\}$& $\{4, 16, 23\}$\\
&$\{4, 20, 23\}$& $\{4, 19, 20\}$& $\{16, 20, 23\}$& $\{20, 22, 23\}$& $\{4, 18, 19\}$& $\{4, 11, 18\}$& $\{4, 9, 19\}$\\
&$\{9, 13, 19\}$& $\{11, 16, 18\}$& $\{4, 9, 22\}$& $\{4, 9, 21\}$& $\{4, 11, 21\}$& $\{4, 11, 22\}$& $\{11, 12, 21\}$\\
&$\{5, 7, 10\}$& $\{5, 7, 17\}$& $\{5, 10, 22\}$& $\{5, 17, 22\}$& $\{7, 9, 17\}$& $\{9, 17, 23\}$& $\{9, 21, 23\}$\\
&$\{9, 20, 21\}$& $\{7, 9, 20\}$& $\{7, 9, 10\}$& $\{9, 10, 22\}$& $\{9, 13, 22\}$& $\{6, 7, 16\}$& $\{6, 7, 8\}$\\
&$\{6, 16, 20\}$& $\{16, 20, 22\}$& $\{7, 16, 22\}$& $\{7, 8, 16\}$& $\{7, 8, 20\}$& $\{8, 11, 16\}$& $\{8, 14, 20\}$\\
&$\{11, 14, 16\}$& $\{14, 15, 16\}$& $\{12, 14, 20\}$& $\{15, 16, 17\}$& $\{15, 20, 22\}$& $\{15, 18, 20\}$& $\{6, 18, 20\}$\\
&$\{6, 9, 18\}$& $\{6, 19, 20\}$& $\{9, 10, 18\}$& $\{9, 10, 11\}$& $\{10, 12, 18\}$& $\{6, 8, 11\}$& $\{6, 8, 19\}$\\
&$\{6, 14, 19\}$& $\{6, 9, 14\}$& $\{6, 9, 11\}$& $\{6, 11, 22\}$& $\{6, 14, 23\}$& $\{6, 17, 23\}$& $\{6, 13, 17\}$\\
&$\{6, 13, 15\}$& $\{9, 11, 14\}$& $\{11, 14, 22\}$& $\{10, 11, 14\}$& $\{10, 11, 13\}$& $\{10, 13, 17\}$& $\{9, 13, 15\}$\\
&$\{9, 13, 14\}$& $\{9, 14, 15\}$& $\{17, 21, 23\}$& $\{6, 10, 15\}$& $\{6, 10, 12\}$& $\{6, 12, 22\}$& $\{6, 15, 22\}$\\
&$\{12, 15, 22\}$& $\{7, 12, 22\}$& $\{7, 10, 12\}$& $\{7, 10, 14\}$& $\{7, 12, 20\}$& $\{7, 14, 18\}$& $\{7, 13, 20\}$\\
&$\{10, 12, 15\}$& $\{10, 14, 15\}$& $\{10, 14, 23\}$& $\{12, 15, 17\}$& $\{12, 19, 20\}$& $\{12, 16, 19\}$& $\{13, 20, 21\}$\\
&$\{7, 13, 21\}$& $\{7, 13, 18\}$& $\{7, 18, 22\}$& $\{7, 21, 22\}$& $\{14, 15, 18\}$& $\{14, 18, 23\}$& $\{13, 14, 23\}$\\
&$\{15, 17, 18\}$& $\{10, 16, 19\}$& $\{8, 10, 19\}$& $\{8, 10, 23\}$& $\{8, 13, 23\}$& $\{8, 11, 13\}$& $\{8, 11, 12\}$\\
&$\{8, 12, 21\}$& $\{10, 21, 23\}$& $\{11, 12, 13\}$& $\{12, 16, 21\}$& $\{10, 16, 21\}$& $\{10, 16, 17\}$& $\{10, 17, 21\}$\\
&$\{14, 17, 21\}$& $\{12, 14, 17\}$& $\{12, 13, 14\}$& $\{12, 13, 23\}$& $\{12, 18, 23\}$& $\{13, 14, 22\}$& $\{13, 17, 22\}$\\
&$\{13, 17, 19\}$& $\{13, 18, 19\}$& $\{13, 18, 23\}$& $\{14, 21, 22\}$& $\{17, 19, 22\}$& $\{17, 18, 19\}$& $\{18, 19, 22\}$\\
&$\{18, 22, 23\}$& $\{19, 20, 21\}$& $\{21, 22, 23\}$& $\{8, 14, 21\}$& $\{8, 14, 19\}$& $\{8, 19, 21\}$& $\{14, 19, 21\}$\\
&$\{12, 16, 17\}$& $\{12, 16, 18\}$& $\{12, 17, 18\}$& $\{16, 17, 18\}$\\
\end{longtable}}\qed

\section{S$_{6}(2,3,10)$}
\begin{lemma}\label{S610}
Let $X$ be a set of $10$ points and $Y$ be a $4$-subset of $X$. Then there exists an S$_{6}(2,3,10)$ with all blocks from $\binom{X}{3}\backslash \binom{Y}{3}$.
\end{lemma}
\proof Let $X=Z_{10}$ and $Y=Z_{4}$. The blocks of desired design are listed blow:
{\small\begin{longtable}{lllllllll}
$\{0, 1, 4\}$ & $\{0, 1, 5\}$& $\{0, 1, 6\}$& $\{0, 1, 7\}$& $\{0, 1, 8\}$& $\{0, 1, 9\}$& $\{0, 2, 4\}$& $\{0, 2, 5\}$& $\{7, 8, 9\}$\\
$\{0, 2, 6\}$ & $\{0, 2, 7\}$& $\{0, 2, 8\}$& $\{0, 2, 9\}$& $\{0, 3, 4\}$& $\{0, 3, 5\}$& $\{0, 3, 6\}$& $\{0, 3, 7\}$& $\{6, 8, 9\}$ \\
$\{0, 3, 8\}$ & $\{0, 3, 9\}$& $\{0, 4, 5\}$& $\{0, 4, 6\}$& $\{0, 4, 7\}$& $\{0, 5, 6\}$& $\{0, 5, 8\}$& $\{0, 6, 9\}$& $\{6, 7, 9\}$\\
$\{0, 7, 8\}$ & $\{0, 7, 9\}$& $\{0, 8, 9\}$& $\{1, 2, 4\}$& $\{1, 2, 5\}$& $\{1, 2, 6\}$& $\{1, 2, 7\}$& $\{1, 2, 8\}$& $\{6, 7, 8\}$\\
$\{1, 2, 9\}$ & $\{1, 3, 4\}$& $\{1, 3, 5\}$& $\{1, 3, 6\}$& $\{1, 3, 7\}$& $\{1, 3, 8\}$& $\{1, 3, 9\}$& $\{1, 4, 5\}$& $\{5, 7, 9\}$\\
$\{1, 4, 6\}$ & $\{1, 4, 8\}$& $\{1, 5, 6\}$& $\{1, 5, 9\}$& $\{1, 6, 7\}$& $\{1, 7, 8\}$& $\{1, 7, 9\}$& $\{1, 8, 9\}$& $\{5, 7, 8\}$\\
$\{2, 3, 4\}$ & $\{2, 3, 5\}$& $\{2, 3, 6\}$& $\{2, 3, 7\}$& $\{2, 3, 8\}$& $\{2, 3, 9\}$& $\{2, 4, 6\}$& $\{2, 4, 7\}$& $\{5, 6, 9\}$\\
$\{2, 4, 9\}$ & $\{2, 5, 7\}$& $\{2, 5, 8\}$& $\{2, 5, 9\}$& $\{2, 6, 7\}$& $\{2, 6, 8\}$& $\{2, 8, 9\}$& $\{3, 4, 7\}$& $\{5, 6, 8\}$\\
$\{3, 4, 8\}$ & $\{3, 4, 9\}$& $\{3, 5, 7\}$& $\{3, 5, 8\}$& $\{3, 5, 9\}$& $\{3, 6, 7\}$& $\{3, 6, 8\}$& $\{3, 6, 9\}$& $\{5, 6, 7\}$\\
$\{4, 5, 6\}$ & $\{4, 5, 7\}$& $\{4, 5, 8\}$& $\{4, 5, 9\}$& $\{4, 6, 8\}$& $\{4, 6, 9\}$& $\{4, 7, 8\}$& $\{4, 7, 9\}$& $\{4, 8, 9\}$\\
\end{longtable}}\qed

\end{CJK}

\begin{thebibliography}{99}
\bibitem{1}  R. J. R. Abel, C. J. Colbourn, and J. H. Dinitz, Mutually orthogonal Latin squares (MOLS), in: C. J. Colbourn, J. H. Dinitz (eds.) The CRC Handbook of Combinatorial Designs, 2nd edn. CRC Press, Boca Raton, FL, (2007) 160-193.
\bibitem{3} A.-E. Baert, V. Boudet, and A. Jean-Marie, Performance analysis of data replication in grid delivery networks, in International Conference on Complex Intelligent and Software Intensive Systems, Barcelona, (2008) 369-374.
\bibitem{AJ} A.-E. Baert, V. Boudet, A. Jean-Marie, and X. Roche, Minimization of download times in a distributed VOD system, in ICPP08: The International Conference on Parallel Processing, IEEE, Los Alamitos, CA, (2008) 173-180.
\bibitem{2} A.-E. Baert, V. Boudet, A. Jean-Marie, and X. Roche, Minimization of download time variance in a distributed VOD system, Scalable Comput. Pract. Exp., West University of Timisoara, 10 (2009) 75-86.
\bibitem{35} R. D. Baker,  Partitioning the planes of AG$_{2m}$(2) into $2$-designs, Discrete Math. 15 (1976) 205-211.
\bibitem{wbd} J.-C. Bermond, A. Jean-Marie, D. Mazauric, and J. Yu, Well-balanced designs for data placement, J. Combin. Des. 24 (2016) 55-76.
\bibitem{PCS} H. Cao, L. Ji, and L. Zhu, Large sets of disjoint packings on $6k + 5$ points, J. Combin. Theory Ser. A 108 (2004) 169-183.
\bibitem{Chen} D. Chen, C. C. Lindner, and D. R. Stinson, Further results on large sets of disjoint group-divisible designs, Discrete Math. 110 (1992) 35-42.
\bibitem{Fu} C.-M. K. Fu, The intersection problem for pentagon systems, Ph. D. Thesis, Auburn University, 1987.
\bibitem{652} H. Hanani, A class of three-designs, J. Combin. Theory Ser. A 26 (1979) 1-19.
\bibitem{46} H. Hanani, On some tactical configurations, Can. J. Math. 15 (1963) 702-722.
\bibitem{fd} A. Hartman, The fundamental construction for $3$-designs, Discrete Math. 124 (1994) 107-132.
\bibitem{13} A. Jean-Marie, X. Roche, V. Boudet, and A.-E. Baert, Combinatorial designs and availability, Research Report RR-7119 INRIA (2009).
\bibitem{PGDD1} L. Ji, A complete solution to existence of $H$ designs, J. Combin. Des. 27 (2019) 75-81.
\bibitem{Lsts} L. Ji, A new existence proof for large sets of disjoint Steiner triple systems, J. Combin. Theory Ser. A 112 (2005) 308-327.
\bibitem{45} L. Ji, On the 3BD-closed set B$_{3}(\{4, 5\})$, Discrete Math. 287 (2004) 55-67.
\bibitem{456} L. Ji, On the 3BD-closed set B$_{3}(\{4, 5,6\})$, J. Combin. Des. 12 (2004) 92-102.
\bibitem{JI} L. Ji, Partition of triples of order $6k + 5$ into $6k + 3$ optimal packings and one packing of size $8k + 4$, Graph Combin. 22 (2006) 251-260.
%\bibitem{HD} W. H. Mills, On the existence of H designs, Congr. Numer. 79 (1990), 129-141.
\bibitem{CS} H. Moh¨¢csy and D. K. Ray-Chaudhuri, Candelabra systems and designs, J. Statist. Plann. Inference 106 (2002) 419-448.
%\bibitem{sqs} L. Teirlinck, Some new $2$-resolvable Steiner quadruple systems, Des. Codes Cryptogr. 4 (1994), 5-10.
\bibitem{wbt} H. Wei, G. Ge, and C. J. Colbourn, The existence of well-balanced triple systems, J. Combin. Des. 24 (2016) 77-100.

\end{thebibliography}
\end{document}